\documentclass[11pt,letterpaper]{article}

\usepackage[utf8]{inputenc}

\usepackage{amsmath, amsthm}
\usepackage{MnSymbol} 
\usepackage{graphicx,color}
\usepackage[hyphens]{url}
\usepackage{dsfont}
\usepackage{booktabs}
\usepackage[normalem]{ulem}
\usepackage{mathtools}

\usepackage{hyperref}
\usepackage[nameinlink,capitalize]{cleveref}

\usepackage{multirow}
\usepackage{algorithm}
\usepackage[noend]{algpseudocode}
\usepackage{tikz} %

\usepackage{soul} %
\usepackage{amsmath}
\usepackage[T1]{fontenc}
\DeclareMathAlphabet{\mathpzc}{T1}{pzc}{m}{it}

\usepackage[giveninits=true, maxnames=10, style=alphabetic, defernumbers=true]{biblatex}
\AtBeginRefsection{\GenRefcontextData{sorting=ynt}}
\AtEveryCite{\localrefcontext[sorting=ynt]}

\usepackage{subcaption}

\usepackage{tabularx}
\newcolumntype{L}{>{\raggedright\arraybackslash}X}

\DeclareMathAlphabet\mathbb{U}{msb}{m}{n}

\numberwithin{equation}{section}
\newtheorem{theorem}{Theorem}[section]

\newtheorem{lemma}[theorem]{Lemma}
\newtheorem{remark}[theorem]{Remark}
\newtheorem{prop}[theorem]{Proposition}

\newtheorem{example}[theorem]{Example}

\newtheorem{defin}[theorem]{Definition}
\newtheorem{assumption}[theorem]{Assumption}
\newtheorem{intuition}[theorem]{Intuition}

\newcommand{\cD}{\mathcal{D}}

\newcommand{\cH}{\mathcal{H}}

\newcommand{\cM}{\mathcal{M}}
\newcommand{\cN}{\mathcal{N}}

\newcommand{\R}{\mathbb{R}}

\newcommand{\N}{\mathbb{N}}

\newcommand*{\E}{\mathbb{E}}

\newcommand*{\Otilde}{\widetilde{O}}

\newcommand*{\poly}{\mathrm{poly}}

\newcommand*{\eps}{\varepsilon}

\DeclareMathOperator*{\argmin}{arg\,min}

\renewcommand{\le}{\leqslant}
\renewcommand{\ge}{\geqslant}
\renewcommand{\leq}{\leqslant}
\renewcommand{\geq}{\geqslant}

\providecommand{\abs}[1]{\lvert{#1}\rvert}

\providecommand{\norm}[1]{\lVert{#1}\rVert}

\DeclareMathOperator{\KL}{\mathsf{KL}}

\newcommand{\TV}{\mathsf{TV}}

\newcommand{\Ren}{\mathsf{R}}

\usepackage{bbm}
\usepackage{blkarray}
\usepackage{cancel}
\usepackage{enumitem}
\usepackage{float}
\usepackage{mathrsfs}
\usepackage[framemethod=tikz]{mdframed}
\usepackage{microtype}
\usepackage{ragged2e}
\usepackage{sectsty}
\usepackage{thmtools}
\usepackage{thm-restate}
\usepackage[Symbolsmallscale]{upgreek}
\usepackage{array}

\usetikzlibrary{cd}

\hypersetup{citecolor=violet, colorlinks=true, linkcolor=blue}

\makeatletter
\patchcmd{\@counteralias}
{\@ifdefinable{c@#1}}
{\expandafter\@ifdefinable\csname c@#1\endcsname}
{}{}
\makeatother

\makeatletter
\newcommand{\opnorm}{\@ifstar\@opnorms\@opnorm}
\newcommand{\@opnorms}[1]{%
	\left|\mkern-1.5mu\left|\mkern-1.5mu\left|
	#1
	\right|\mkern-1.5mu\right|\mkern-1.5mu\right|
}
\newcommand{\@opnorm}[2][]{%
	\mathopen{#1|\mkern-1.5mu#1|\mkern-1.5mu#1|}
	#2
	\mathclose{#1|\mkern-1.5mu#1|\mkern-1.5mu#1|}
}
\makeatother

\newcommand{\PreserveBackslash}[1]{\let\temp=\\#1\let\\=\temp}
\newcolumntype{C}[1]{>{\PreserveBackslash\centering}p{#1}}

\newcommand\bs[1]{\boldsymbol{#1}}
\newcommand\mb[1]{\mathbf{#1}}

\newcommand\mc[1]{\mathcal{#1}}
\newcommand\mf[1]{\mathfrak{#1}}
\newcommand\ms[1]{\mathscr{#1}}

\newcommand\msf[1]{\mathsf{#1}}

\definecolor{MITBrown}{RGB}{164, 31, 50}

\DeclareMathOperator\cov{cov}

\DeclareMathOperator\law{law}

\DeclareMathOperator\one{\mathbbm{1}}

\DeclareMathOperator\var{var}

\renewcommand{\Pr}{\mathbb{P}}

\newcommand{\D}{\mathrm{d}}

\newcommand{\Coup}{\ms C}

\newcommand\deq{\coloneqq}
\newcommand\e{\mathrm{e}}

\newcommand\mmid{\mathbin{\|}}

\newcounter{dummy}
\makeatletter
\newcommand\myitem[1][]{\item[#1]\refstepcounter{dummy}\def\@currentlabel{#1}}
\makeatother

\tikzset{
    shadedNode/.style={rectangle, draw=none, fill=blue!20, inner sep=1mm}
}

\global\long\def\inner#1{\left\langle #1\right\rangle }%

\newcommand{\alg}{\msf{alg}}
\newcommand{\sh}{\msf{sh}}
\newcommand{\aux}{\msf{aux}}
\newcommand{\msx}{\msf{x}}
\newcommand{\msp}{\msf{p}}
\newcommand{\ttd}{\mathtt{d}}
\newcommand{\tte}{\mathtt{e}}

\newcommand{\Thetatilde}{\widetilde{\Theta}}

\newcommand{\Omegatilde}{\widetilde{\Omega}}

\newcommand{\Renyi}{\msf{R}}

\newcommand{\absomega}{|\omega|}

\usepackage[margin=1in]{geometry}

\makeatletter
\def\blfootnote{\gdef\@thefnmark{}\@footnotetext}
\makeatother

\allowdisplaybreaks

\begin{document}

    \title{Shifted Composition IV\@: 
    \\ Toward Ballistic Acceleration for Log-Concave Sampling
    	}
 	\author{
    		Jason M.\ Altschuler\\
		\small UPenn\\
		\texttt{\small alts@upenn.edu}
		\and
		Sinho Chewi \\
		\small Yale \\
		\texttt{\small sinho.chewi@yale.edu} \and
        Matthew S.\ Zhang \\
        \small UToronto \\
        \texttt{\small matthew.zhang@mail.utoronto.ca}
    }  
	\date{\today}
	\maketitle

    \begin{abstract}
        Acceleration is a celebrated cornerstone of convex optimization, enabling gradient-based algorithms to converge sublinearly in the condition number. A major open question is whether an analogous acceleration phenomenon is possible for log-concave sampling. Underdamped Langevin dynamics (ULD) has long been conjectured to be the natural candidate for acceleration, but a central challenge is that its degeneracy necessitates the development of new analysis approaches, e.g., the theory of hypocoercivity. Although recent breakthroughs established ballistic acceleration for the (continuous-time) ULD diffusion via space-time Poincar\'e inequalities, (discrete-time) algorithmic results remain entirely open: the discretization error of existing analysis techniques dominates any continuous-time acceleration. 
        \par In this paper, we give a new coupling-based local error framework for analyzing ULD and its numerical discretizations in KL divergence. This extends the framework in Shifted Composition III from uniformly elliptic diffusions to degenerate diffusions, and shares its virtues: the framework is user-friendly, applies to sophisticated discretization schemes, and does not require contractivity. Applying this framework to the randomized midpoint discretization of ULD establishes the first ballistic acceleration result for log-concave sampling (i.e., sublinear dependence on the condition number). Along the way, we also obtain the first $d^{1/3}$ iteration complexity guarantee for sampling to constant total variation error in dimension $d$.
    \end{abstract}

    \newpage

    \small
	\setcounter{tocdepth}{2}
	\tableofcontents	
	\normalsize
	\newpage

\section{Introduction}\label{sec:intro}

Sampling from a high-dimensional distribution $\pi$ on $\R^d$ is a fundamental algorithmic problem with applications throughout statistics, engineering, and the sciences; see for example the textbooks~\cite{james1980monte,  robert1999monte, liu2001monte, andrieu2003introduction, chewibook}. The standard approach is Markov Chain Monte Carlo (MCMC): construct a Markov chain whose stationary distribution (approximately) equals $\pi$ and run it until approximate stationarity. It is of central importance to bound the rate of convergence to stationarity as this dictates the number of required iterations, a.k.a.\ the mixing time of the algorithm.

This paper revisits the mixing time of discretizations of the underdamped Langevin diffusion (ULD; sometimes also called the kinetic Langevin diffusion) 
\begin{align}\label{eq:ULD}\tag{$\msf{ULD}$}
    \D X_t = P_t\,\D t\,, \qquad \D P_t = (-\nabla V(X_t)-\gamma P_t)\,\D t + \sqrt{2\gamma}\,\D B_t\,.
\end{align}
Here, $\{B_t\}_{t\ge 0}$ is a standard Brownian motion on $\mathbb{R}^d$, and $X_t$, $P_t$, $\gamma$, and $V = -\log \pi$ have the physical interpretations of position, momentum, friction, and potential, respectively. It is classically known that under mild conditions, the stationary distribution of~\ref{eq:ULD} has density $\bs\pi$ given by $\bs \pi(x,p) \propto \exp(-V(x) - \|p\|^2/2)$~\cite{tropper1977ergodic}; thus if (a discretization of) this diffusion is run to (approximate) stationarity, then the $x$-coordinate of the final iterate yields an (approximate) sample from the target distribution $\pi$. 

\paragraph*{Motivation for underdamped dynamics.} The motivation for studying~\ref{eq:ULD} is that underdamped dynamics are believed to produce faster MCMC algorithms. In the past decade it has been shown that, compared to their counterparts for overdamped Langevin dynamics, numerical discretizations of~\ref{eq:ULD} have improved dependence on both the dimension $d$ and sampling accuracy $\eps$~\cite{cheng2018underdamped, shen2019randomized, dalalyanrioudurand2020underdamped}.

\par Moreover, it is widely conjectured that underdamped algorithms should also have better dependence on the condition number $\kappa$, i.e., the ratio $\kappa = \beta/\alpha$ where $0 \prec \alpha I \preceq \nabla^2 V \preceq \beta I$ are bounds on the strong convexity (or some other isoperimetric quantity) and smoothness of the potential $V$. In particular, all existing sampling algorithms require a number of iterations that scales at least linearly in $\kappa$, and it is conjectured that underdamped dynamics might lead to convergence rates that scale instead as $\sqrt{\kappa}$. Such a convergence rate would be called \emph{accelerated} or \emph{ballistic}. Whether accelerated sampling is possible is a major open question in the field, see for example~\cite{Dalalyan19friendly, shen2019randomized, chen2022improved, tian22thesis, gopietal23loglaplace, jiang2023dissipation, wang2022accelerating, Zhang+23ULMC, AltChe24Warm,apers2024hamiltonian, fu2025hamiltonian, gouraud2025hmc,chewibook}.

\par This acceleration conjecture is motivated by intuition from several communities. From a probabilistic perspective, underdamped dynamics are a canonical example of non-reversible lifts, which augment the state space of a Markov chain to include a momentum variable. Non-reversible lifts are well-documented to lead to faster mixing times in both practice and theory, by up to a square root factor~\cite{CPL99,diaconis2000analysis,TCV11,Vuc16,BR17,eberle2024non,EGH25,LiLu25}. From a statistical physics perspective, incorporating momentum enables faster exploration of large regions where $\pi$ has non-negligible mass, a challenge that is sometimes called the ``entropic barrier'' for fast mixing. Relatedly, from a convex optimization perspective, underdamped dynamics are the sampling analogue of accelerated optimization algorithms which use momentum to enable faster traversal of long ravines in optimization landscapes, and are classically known to lead to iteration complexities scaling in $\sqrt{\kappa}$ rather than $\kappa$ for minimizing $\kappa$-conditioned potentials $V$~\cite{nesterov1983method}. However, despite the many close connections between optimization and sampling (see the textbook~\cite{chewibook}), the analogous question of acceleration has remained entirely open for sampling. In fact, a portion of the community has even speculated that acceleration is impossible in sampling~\cite[page 2]{shen2019randomized}.

\paragraph*{Analysis of underdamped dynamics.} A central challenge for analyzing underdamped dynamics is degeneracy: the Brownian motion acts only on the momentum coordinate, and not on the position. This is in contrast to overdamped dynamics, which are driven by full-dimensional Brownian motion. This partial injection of noise enables underdamped dynamics to have faster convergence, smoother sample paths, and lower discretization error. However, quantifying the regularity and convergence of underdamped dynamics necessitates the development of new analysis approaches and has been a major area of research in PDE since 
the early 1900s~\cite{langevin1908theorie, Kol34}.

\par In order to establish non-asymptotic bounds on the convergence rate of~\ref{eq:ULD}, Villani pioneered the theory of hypocoercivity~\cite{villani2009hypocoercivity}---named in homage to H\"ormander's celebrated theory of hypoellipticity for establishing regularity of degenerate diffusions~\cite{Hor67}. In the years since, hypocoercivity has played a predominant role in the search for optimal convergence rates for~\ref{eq:ULD}. A key recent breakthrough of~\cite{Alb+24KineticFP} used hypocoercivity techniques to establish \emph{space-time Poincar\'e inequalities}, which show ballistic acceleration for the first time for~\ref{eq:ULD}. These space-time Poincar\'e inequalities provide a square root speed-up for convex potentials $V$~\cite{lu2022explicit, CaoLuWan23Underdamped}, and were shown to be optimal by~\cite{eberle2024non}. However, as yet, such acceleration results remained confined to the (continuous-time) diffusion~\ref{eq:ULD} and remain an entirely open question for (discrete-time) algorithms. This is because existing analysis techniques are sufficiently lossy that the discretization error dominates the improved complexity of these accelerated results for continuous-time diffusions. 

\subsection{Contribution}

The goal of this paper is to provide a framework for analyzing~\ref{eq:ULD} and its discretizations in KL divergence using coupling arguments. This deviates from traditional PDE approaches 
based on hypocoercivity. Coupling arguments are well-suited both for analyzing mixing times of (discrete-time) algorithms since discretization errors can be naturally accumulated via coupling, as well as for establishing Harnack inequalities for (continuous-time) diffusions, which are of interest in their own right. Below we detail both these discrete-time and continuous-time results, starting with the latter as
it can be viewed as a special case of the former, where no discretization is used.

\subsubsection{Harnack inequalities for the underdamped Langevin diffusion}\label{sssec:intro:harnack}

There are two types of analysis approaches for proving non-asymptotic convergence rates for~\ref{eq:ULD}. On one hand, PDE approaches such as hypocoercivity, entropic approaches, and spectral gap bounds often lead to convergence results of the form 
\begin{align}\label{eq:intro:reverse-transport}
    \cD(\bs{\mu} \bs P_T \mmid \bs\nu \bs P_T) \leq C_1(T)\, \cD(\bs{\mu} \mmid \bs{\nu} )
\end{align}
in terms of divergences $\cD$ such as KL, $\chi^2$, or more generally a R\'enyi divergence. Here $\bs{P}_T$ denotes the Markov kernel corresponding to running~\ref{eq:ULD} for time $T$, and $C_1(T)$ is some decay function.
\par On the other hand, another common approach to analyzing the convergence rate of Markov chains is coupling arguments. This leads to weaker convergence results of the form
\begin{align}
    W(\bs{\mu} \bs P_T , \bs\nu \bs P_T) \leq C_2(T) \, W(\bs{\mu} ,  \bs{\nu} )
\end{align}
in terms of some Wasserstein metric $W$. This holds for contractive settings via the synchronous coupling and has been extended to other settings via more sophisticated couplings~\cite{eberle2019couplings}. 
Broadly speaking, coupling arguments are simpler and more flexible, but PDE approaches lead to convergence in stronger metrics. %

An influential idea pioneered by F.-Y.\ Wang~\cite{Wang1997LSINoncompact} is that one can strengthen coupling arguments by applying them to a certain auxiliary third process (in tandem with Girsanov's theorem) in order to obtain bounds of the form
\begin{align}
    \KL(\bs{\mu} \bs P_T \mmid \bs\nu \bs P_T) \leq C_3(T)\, W_2^2(\bs{\mu}, \bs{\nu} )\,,
    \label{eq:intro:rti}
\end{align}
where $\KL$ denotes the Kullback--Leibler divergence and
$W_2$ denotes the $2$-Wasserstein metric. Such inequalities are called \emph{reverse transport inequalities} and are of substantial interest as they encode regularity properties of the semigroup, apply beyond contractive settings, extend to R\'enyi divergences, and are equivalent to parabolic (dimension-free) Harnack inequalities by duality. 

\par In \S\ref{scn:continuous-time}, we develop a coupling-based argument for proving reverse transport inequalities for~\ref{eq:ULD}. Our bound tightly reflects both the short-time degeneracy which scales as $1/T^3$ for small $T > 0$, as well as the long-time convergence rate $1/T$ (weakly convex case) or $e^{-cT}$ (strongly convex case) for large $T \to \infty$. In the weakly convex case, this result is new.

\paragraph*{Comparison with~\cite{GuiWan12DegenBismut}.} Reverse transport inequalities were first shown for~\ref{eq:ULD} by A.\ Guillin and F.-Y.\ Wang~\cite{GuiWan12DegenBismut}. On one hand, their results hold more generally than ours, as they allow for higher degrees of degeneracy and weaker smoothness assumptions. On the other hand, our results have two key features that theirs do not. 
First, our reverse transport inequalities encode the convergence of~\ref{eq:ULD} in the weakly convex case, in the sense that our bound converges to $0$ as $T \to \infty$ (details in Remark~\ref{rem:harnack-decay-comparison}). Second, whereas the approach of~\cite{GuiWan12DegenBismut} establishes reverse transport inequalities through a chain of implications\footnote{Namely,~\cite{GuiWan12DegenBismut} proves a Bismut formula, which implies a gradient-entropy bound, which leads via integration to Harnack inequalities, which is equivalent by duality to reverse transport inequalities; see also Section 1.3 of the textbook~\cite{Wang13HarnackSPDE} for details on this sequence of implications.}, our approach uses a direct coupling-based argument which enables accounting for errors from numerical discretizations of this diffusion, as described next in \S\ref{sssec:intro:framework}. These two features are enabled by using a different coupling: our coupling is based on the one of~\cite{ArnThaWan06HarnackCurvUnbdd}
for uniformly elliptic diffusions (which in that setting is known to be amenable to incorporating discretization errors~\cite{scr1}). Previously, that coupling had not been successfully adapted to degenerate diffusions; this requires several new ideas, see the technical overview in \S\ref{ssec:intro:tech}.

\subsubsection{KL local error framework for underdamped Langevin discretizations}\label{sssec:intro:framework}

Coupling arguments are well-suited for analyzing discretizations of diffusions, since discretization errors are naturally measured in Wasserstein-type metrics and naturally accumulated using coupling.
In \S\ref{sec:discretization}, we develop a systematic framework that essentially\footnote{To simplify the analysis, we use a perturbed discretization $\hat{\bs{P}}{}'$ in the final iteration, resulting in bounds of the form $\KL(\bs{\mu} \hat{\bs{P}}{}^{N-1} \hat{\bs{P}}{}'  \mmid \bs{\nu} \bs P^N)$. This ``last-step hack'' is simple to implement algorithmically and only has a lower-order effect.} yields bounds of the form
\begin{align}
    \KL(\bs\mu 
    \hat{\bs P}{}^{N} 
    \mmid \bs\nu\bs P^N)
    \lesssim 
    C(Nh)
    \,\mc W_2^2(\bs \mu,\bs \nu) + \mathtt{Err}\,,
    \label{eq:intro:framework}
\end{align}
where $\bs{P}$ denotes the Markov kernel corresponding to running~\ref{eq:ULD} for some time $h > 0$, $\hat{\bs{P}}$ is some discretization thereof, and $\bs{\mu},\bs{\nu}$ are arbitrary initialization measures. Here, the constant 
$C(Nh) = C_3(T)$
is the constant in the reverse transport inequality~\eqref{eq:intro:rti} for the continuous-time dynamics, $\mc W_2$ denotes $2$-Wasserstein distance in a certain twisted norm, and $\mathtt{Err}$ encapsulates the discretization error accumulated from running  $\hat{\bs{P}}$ rather than $\bs{P}$ for $N$ iterations.

\par This extends our framework in part III~\cite{scr3} from elliptic diffusions to degenerate diffusions, and shares its virtues:
 \begin{itemize}
    \item \textbf{User-friendly.} The long-term discretization error $\mathtt{Err}$ is bounded in terms of short-term ``local error'' estimates, namely the weak and strong errors for one step of the discretization. The use of short-term estimates makes it easy to check the assumptions and apply the framework. 
    \item \textbf{KL local error framework.} The ability to capture both strong and weak error is essential for witnessing the improvement of state-of-the-art discretization schemes such as the randomized midpoint method which only lead to improved discretization bounds after multiple steps. Traditionally, local error analyses only yield final bounds in Wasserstein distance since they are based on coupling arguments~\cite{MilTre21StochNum}; our use of coupling arguments leads to final bounds in KL. 
    This is appealing for several reasons: KL is stronger than $W_2$ by Talagrand's inequality; KL implies TV guarantees by Pinsker's inequality (whereas $W_2$ does not), which is the standard metric for mixing in the Markov chain literature; and KL is of fundamental interest in its own right due to the interpretation of sampling as optimization of the functional $\KL(\cdot \mmid \pi)$~\cite{jordan1998variational}.
     
    \item \textbf{Beyond contractivity.} Another key benefit of performing the local error framework in KL rather than Wasserstein is that our framework does not require the underlying kernels to be contractive in a Wasserstein metric. In such settings where the Wasserstein-Lipschitz parameter $L > 1$, traditional coupling arguments lead to final discretization bounds that blow up exponentially as $L^{\Theta(N)}$ in the number of iterations $N$. In contrast, our framework only accumulates discretization errors linearly in $N$. This is essential for showing ballistic acceleration and for sampling without strong log-concavity.
\end{itemize}
This extension of our KL local error framework in part III~\cite{scr3} from elliptic diffusions to degenerate diffusions is both significantly more challenging from an analysis perspective (due to the degeneracy, see the techniques overview in \S\ref{ssec:intro:tech}) and also significantly more interesting from an application perspective (as underdamped dynamics lead to faster sampling algorithms).

\subsubsection{Applications to sampling}\label{sssec:intro:applications}

In \S\ref{sec:app}, we apply this framework to establish improved algorithmic guarantees for sampling from a distribution $\pi \propto \exp(-V)$ using first-order queries (i.e., using query information of the form $\nabla V(x)$). As mentioned earlier, this is a well-studied problem with applications across statistics, engineering, and the sciences, for example Bayesian statistics, numerical quadrature, and differential privacy; see the textbooks~\cite{james1980monte, robert1999monte, liu2001monte, andrieu2003introduction, chewibook}. We focus on the fundamental setting of well-conditioned $\pi$, meaning that $V$ is $\beta$-smooth (i.e., its gradient $\nabla V$ is $\beta$-Lipschitz) and is $\alpha$-strongly convex (or satisfies an isoperimetric inequality like the log-Sobolev inequality or Poincar\'e inequality with constant $1/\alpha$). Denote $\kappa := \beta/\alpha$.

\paragraph*{Ballistic acceleration.} Our main application is the first result toward acceleration for log-concave sampling. For context, the state-of-the-art iteration complexities for sampling scale at least linearly in the condition number $\kappa$. Over the years, such results have been proved for a variety of algorithms using a variety of analysis techniques; see the textbook~\cite{chewibook} and the references within. A major open question in the sampling community is whether this linear dependence on $\kappa$ is improvable. As mentioned earlier, on one hand intuition from probability (non-reversible lifts), statistical physics (entropic barrier), and optimization (momentum-based acceleration) provide a hope that $\kappa$ dependence might be improvable to $\sqrt{\kappa}$ via a diffusive-to-ballistic speedup. However, on the other hand, the failure to find algorithms which can achieve $o(\kappa)$ rates has led $\Otilde(\kappa)$ to be called a ``natural barrier'' that some have conjectured to be unimprovable; see, e.g.,~\cite[page 2]{shen2019randomized}. 

\par In Theorem~\ref{thm:rmulmc-spacetime}, we apply our framework to analyze (a variant of) the \emph{randomized midpoint discretization of the underdamped Langevin diffusion (RM--ULMC)} from~\cite{shen2019randomized} in order to obtain the first iteration complexity which is sublinear in $\kappa$.

\begin{theorem}[Informal version of Theorem~\ref{thm:rmulmc-spacetime}]\label{thm:rmulc-spacetime-informal}
    Let $\pi$ be a strongly log-concave and log-smooth measure over $\R^d$ with condition number $\kappa$.
    The randomized midpoint discretization of the underdamped Langevin diffusion~\eqref{eq:rm-ulmc}, when initialized from a measure $\mu_0$ with $\log \chi^2(\mu_0 \mmid \pi) = \widetilde O(d)$, outputs a
    sample which is $\varepsilon^2$ close to $\pi$ in KL divergence using $\Otilde(\kappa^{5/6} d^{5/3} / \varepsilon^{2/3})$ gradient evaluations.
\end{theorem}

It remains an interesting question whether $\kappa^{5/6}$ can be further improved to the fully ballistically accelerated rate of $\kappa^{1/2}$, and it is our hope that this first result toward that goal may provide a conceptually helpful building block even if the complexity itself is later improved. 

\par Our result uses, as a key component, the recent breakthrough of space-time Poincar\'e inequalities which established ballistic acceleration for the continuous-time diffusion~\ref{eq:ULD}~\cite{Alb+24KineticFP,lu2022explicit, CaoLuWan23Underdamped,eberle2024non}. As mentioned earlier, such acceleration results have so far been confined to the (continuous-time) diffusion~\ref{eq:ULD} and have remained an entirely open question for (discrete-time) algorithms. The key challenge is that existing techniques are sufficiently lossy that the discretization error dominates the improved complexity of these accelerated results for continuous-time diffusions. Our discretization framework enables witnessing this acceleration for algorithms for the first time. 

\par An important conceptual challenge with this discretization analysis is that in order to obtain ballistic acceleration, even in continuous time, it is essential to have low friction (scaling as $\gamma \asymp \sqrt{\alpha}$ rather than $\sqrt{\beta}$), and in this regime, the underdamped dynamics are not contractive. This prohibits using standard coupling-based discretization analyses since such arguments exponentially accumulate discretization errors. Crucially, our analysis framework bypasses this contractivity issue by performing the analysis in KL rather than Wasserstein, which enables accumulating discretization errors at only a linear rate, as described above in the overview of our framework. %

\paragraph*{Stronger metrics and state-of-the-art dependence in other parameters.} The above discussion is about dependence on the condition number $\kappa$; however, in sampling, it is also important to quantify the dependence with respect to the dimension $d$ and the error $\varepsilon$.
Compared to our result $\Otilde(\kappa^{5/6} d^{5/3}/\varepsilon^{2/3})$ in Theorem~\ref{thm:rmulc-spacetime-informal}, the original result of~\cite{shen2019randomized} showed an iteration complexity for RM--ULMC which scales\footnote{To be precise, the result of~\cite{shen2019randomized} reads $\Otilde(\kappa d^{1/3}/\eps^{2/3} + \kappa^{7/6}d^{1/6}/\eps^{1/3})$. For simplicity of exposition, we drop the second term as it can be removed by using a slightly refined ``double'' midpoint discretization, see Remark~\ref{rem:double-midpoint}.
} in the dimension $d$ and $W_2$ error $\eps$ as $\Otilde(\kappa d^{1/3}/\varepsilon^{2/3})$.
The worse dimension dependence in our result is an artefact of the use of the space-time Poincar\'e inequality of~\cite{CaoLuWan23Underdamped} and could potentially be improved to $\Otilde(\kappa^{5/6} d^{1/3}/\varepsilon^{2/3})$ either via a warm start or via a conjectural ``space-time log-Sobolev inequality''.

\par Nevertheless, absent such an inequality, we also use our framework to prove a second result, which no longer exhibits ballistic acceleration but matches the $\Otilde(\kappa d^{1/3}/\varepsilon^{2/3})$ rate of~\cite{shen2019randomized} in all parameters and strengthens it from $W_2$ to $\KL$.

\begin{theorem}[Informal version of Theorem~\ref{thm:rmulmc_cvx}, $\alpha > 0$ case]\label{thm:rulmc_cvx_informal}
    Let $\pi$ be a strongly log-concave and log-smooth measure over $\R^d$ with condition number $\kappa$.
    The randomized midpoint discretization of the underdamped Langevin diffusion~\eqref{eq:rm-ulmc}, when initialized from a measure $\mu_0$ with $W_2^2(\mu_0, \pi) = \poly(d)$, outputs a
    sample which is $\varepsilon^2$ close to $\pi$ in KL divergence using $\Otilde(\kappa d^{1/3} / \varepsilon^{2/3})$ gradient evaluations.
\end{theorem}

\par Obtaining any guarantees beyond the $W_2$ metric, such as in the total variation (TV) metric or the KL divergence, has been stated as a ``highly non-trivial open problem''~\cite{YuDal25Parallel}. Such guarantees are desired since mixing times are classically measured in TV, and also since stronger metrics than $W_2$ are required for applications such as warm starts for high-accuracy samplers as well as for differential privacy. However, obtaining such guarantees is challenging because the randomized midpoint discretization uses a ``look-ahead step'' at a random integration time in order to nearly debias each step of the discretization scheme (thereby substantially improving the weak error). This renders the process non-adapted, which prohibits standard applications of Girsanov's theorem or other results which can provide final bounds in divergences like KL. 

\par Recently, the work of~\cite{KanNag24PoissonMidpt} obtained a first TV guarantee for RM--ULMC which scales as $\Otilde(\kappa^{17/12} d^{5/12}/\varepsilon^{1/2})$.
Compared to our KL divergence guarantee (which is stronger than TV, by Pinsker's inequality), which scales as $\Otilde(\kappa d^{1/3}/\eps^{2/3})$, the two rates are incomparable---\cite{KanNag24PoissonMidpt} has a better dependence on $1/\varepsilon$, but our result holds in a stronger metric and has better dependence on $\kappa$ and $d$. %
In particular, for any constant TV error $\eps$ (e.g., $\eps = 1/100$), our results imply state-of-the-art dimension dependence of $\Otilde(d^{1/3})$. Previous results required at least $\Otilde(d^{1/2})$~\cite{FanYuaChe23ImprovedProx, scr3, AltChe24Warm} or $\widetilde O(d^{5/12})$~\cite{KanNag24PoissonMidpt}.

Finally, the concurrent work of~\cite{SriNag25PoisMidpt} claims a $W_2$ rate of roughly $\Otilde(\kappa^{7/6} d^{1/3}/\varepsilon^{1/3})$, which has an improved dependence on $1/\varepsilon$.
It is an interesting future direction to unify their approach with ours in order to obtain sharper KL bounds.

\subsection{Techniques}\label{ssec:intro:tech}

\paragraph*{Starting point.} The KL local error framework in this paper extends part III~\cite{scr3} from elliptic diffusions to degenerate diffusions, so we begin by explaining that simpler framework. The starting point is the seminal idea of~\cite{ArnThaWan06HarnackCurvUnbdd}: the deviation between two stochastic processes can be bounded by applying Girsanov’s theorem to a third, auxiliary process which interpolates the original two processes.  This leads to bounds that depend only on how much ``energy'' is required to ``shift'' the auxiliary process so that it interpolates the two processes. The upshot is that this argument yields a final bound in divergences like KL (which are important for downstream applications and are of fundamental interest in their own right, see the discussion in \S\ref{sssec:intro:framework}), while the analysis only requires bounding the distance between the auxiliary process and the process it should hit at the final time—which can be achieved via coupling arguments. This \emph{shifted Girsanov} argument of~\cite{ArnThaWan06HarnackCurvUnbdd} was originally designed for continuous-time processes generated by elliptic diffusions, and has since been successfully applied to many other settings; see~\cite{Wang12Coupling} for a survey.

The central idea of this series~\cite{scr1,scr2,scr3} is that the shifted Girsanov argument of~\cite{ArnThaWan06HarnackCurvUnbdd} can be naturally performed in discrete time using the \emph{shifted composition rule}---an information-theoretic principle developed in part I~\cite{scr1} that refines the classical KL chain rule by enabling certain ``shifts'' in the compared distributions. Among other benefits, this discrete-time argument is more general, applies to settings where Girsanov’s theorem does not, is user-friendly in that it provides long-time estimates from only short-time one-step estimates, and can incorporate both weak and strong error in order to obtain sharp bounds for state-of-the-art discretization schemes; see~\cite{scr3} for a detailed discussion. 

\paragraph*{Degenerate diffusions.} The overarching challenge in this paper is that the degeneracy of underdamped dynamics necessitates significantly more complicated shifts in order to make the auxiliary process interpolate. For simplicity, we begin by explaining this for the continuous-time diffusion~\ref{eq:ULD}, and then describe the additional difficulties posed by discretization below. 

Let $\{(X_t,P_t)\}_{t \geq 0}$ and $\{(\bar{X}_t,\bar{P}_t)\}_{t \geq 0}$ denote two copies of~\ref{eq:ULD}, initialized at $(x,p)$ and $(\bar{x},\bar{p})$, respectively. We construct the third, auxiliary process so that it starts at $(\bar{x},\bar{p})$ and hits $(X_T,P_T)$ at the final time $T$. The original idea of~\cite{ArnThaWan06HarnackCurvUnbdd} was to construct this auxiliary process by mirroring the dynamics of the original processes, with a synchronous coupling of the noise but an additional drift to enforce the interpolation condition at termination. A key challenge for degenerate diffusions is that this drift can \emph{only} be incorporated in the momentum coordinate:
\begin{align}\label{eq:shifted_uld:intro}
\begin{aligned}
    \D X^\aux_t &= P_t^\aux \, \D t\,, \\
    \D P^\aux_t &= \bigl(- \nabla V(X_t^\aux)-\gamma P_t^\aux + \eta_t^\msx\,(X_t - X_t^\aux) + \eta_t^\msp\,(P_t - P_t^\aux) \bigr) \, \D t + \sqrt{2\gamma} \, \D B_t\,,
\end{aligned}
\end{align}
where we choose $\eta^\msx$, $\eta^\msp$ so that $(X_t^\aux, P_t^\aux) \to (X_T, P_T)$ as $t\nearrow T$. (Indeed, incorporating a similar drift in the position coordinate would make this auxiliary process singular w.r.t.\ the original processes, leading to vacuous divergence bounds.)

By the interpolation assumption and an application of Girsanov’s theorem, 
\begin{align}
    \KL(\law(X_T, P_T) \mmid \law(\bar X_T, \bar P_T))
    &= \KL(\law(X_T^\aux, P_T^\aux) \mmid \law(\bar X_T, \bar P_T)) \nonumber\\
    &\le \frac{1}{4\gamma} \int_0^T \E[\norm{\eta_t^\msx \,(X_t - X_t^\aux) + \eta_t^\msp \,(P_t - P_t^\aux)}^2]\,\D t\,.\label{eq:cont_girsanov:intro}
\end{align}
As described earlier, this yields a KL bound in terms of the ``energy'' required to shift the auxiliary process so that it starts from one process and hits the other at termination. This requires bounds on the distances $\|X_t - X_t^\aux\|$ and $\|P_t - P_t^\aux\|$ which we perform by a coupling argument. 

\paragraph*{Twisted coordinates and effective friction.} The standard coupling arguments for~\ref{eq:ULD} require working in the ``twisted coordinate system'' $(x,x+\tfrac{2}{\gamma}\,p)$ and establish contractivity in the strongly convex $(\alpha > 0)$ and high friction regime $(\gamma \asymp \sqrt{\beta})$;
see e.g.,~\cite{eberle2019couplings}.
In our analysis, rather than coupling two copies of the~\ref{eq:ULD} system, we instead couple a standard~\ref{eq:ULD} process $\{(X_t, P_t)\}_{t\ge 0}$ with the auxiliary process $\{(X_t^{\aux}, P_t^{\aux})\}_{t\ge 0}$ in~\eqref{eq:shifted_uld:intro}.
Here, we make the key observation that the ``friction'' term arising from $\D (P_t - P_t^\aux)$ is $-(\gamma +\eta_t^\msp)\,(P_t - P_t^\aux)$. Thus the \emph{effective friction} of this coupling is $\gamma_t \deq \gamma + \eta_t^\msp$, so we consider the \emph{time-varying twisted coordinate system} $(x_t, z_t) \deq (x_t, x_t + \frac{2}{\gamma_t}\,p_t)$. 
We are then able to show that the auxiliary process contracts toward the original process at a certain (time-varying) rate when measured in this (time-varying) coordinate system defined in terms of the (time-varying) effective friction, even in regimes in which the original~\ref{eq:ULD} system is not contractive (namely, the non-convex or low friction settings, the latter being a requirement for our application to ballistic acceleration in \S\ref{sec:app}).

Running this coupling argument then yields the aforementioned distance bounds needed for~\eqref{eq:cont_girsanov:intro}.
This is the main idea of our continuous-time analysis in \S\ref{scn:continuous-time}.

\paragraph*{Discrete-time framework.} In \S\ref{sec:discretization} we extend this to analyzing discretizations of~\ref{eq:ULD}. Now the two (non-auxiliary) proceses to compare undergo \emph{different dynamics}: one evolves through~\ref{eq:ULD} while the other evolves through a discretization. At a high level, our approach synthesizes the aforementioned ideas: following the argument of~\cite{ArnThaWan06HarnackCurvUnbdd}, we construct an auxiliary interpolating process; following the KL local error framework of~\cite{scr3}, we run this argument in discrete time using the shifted composition rule to generalize and strengthen the standard use of Girsanov's theorem; and following the extension to degenerate diffusions described above, we run the coupling argument in a time-varying coordinate system to obtain the appropriate time-varying contraction. Much of this discrete-time analysis is therefore guided by the simpler continuous-time analysis described above, for example the discrete-time shifts are integrated versions of the continuous-time shifts in~\eqref{eq:shifted_uld:intro}.

\par However, there are several obstacles for this discrete-time analysis that are not automatically inherited from the continuous-time analysis. A technical hurdle is that in discrete-time, the diffusion and the shifting are interleaved rather than simultaneous, but their effects (as well as the effect of the time-varying coordinate system) cannot be considered separately; see \S\ref{sec:diffuse_then_shift} for a detailed discussion on this challenge and how we overcome it. The most important conceptual challenge arises from the fact that for our eventual choice of shifts, $\eta_t^\msx \gg \eta_t^\msp$ as $t\nearrow T$, and so obtaining the expected KL bounds requires the auxiliary process to be substantially closer to the target process in the position coordinate $x$ than in the momentum coordinate $p$, by a factor of $h$ where $h$ is the discretization step size. It turns out that many natural discretizations for~\ref{eq:ULD} have smaller errors in position than momentum by precisely this amount, which compensates for this effect; this additional factor of $h$ in both the regularity and the error reflects the fact that there is an additional integral in the position coordinate from underdamped dynamics. %
This challenge is inherent to our KL analysis framework but does not appear in standard $W_2$ coupling arguments. See \S\ref{ssec:kl_local_overview} for a more detailed technical overview of this discrete-time analysis framework.

\subsection{Related work}\label{ssec:related}

In addition to the literature described above, here we provide further context about related work. 

\paragraph*{Acceleration in the special case of Gaussian distributions.} A major motivation of this paper is achieving accelerated convergence rates (i.e., rates that scale in $\sqrt{\kappa}$ rather than $\kappa$) for sampling from target distributions $\pi \propto \exp(-V)$ with $\kappa$-conditioned potentials $V$. It is well-known that accelerated rates are possible in the special case of \emph{Gaussian} distributions, i.e., the special case of $\kappa$-conditioned \emph{quadratic} potentials $V$. This is because one can leverage classical results for the analogous problem in convex optimization: accelerated convergence rates for minimizing $\kappa$-conditioned quadratics $V$. For example, this connection can be made precise via idealized HMC with varying step sizes~\cite{wang2022accelerating}, idealized HMC with damping~\cite{jiang2023dissipation}, a family of (generalized) HMC methods~\cite{gouraud2025hmc}, (non-idealized) HMC with Metropolis--Hastings filters~\cite{apers2024hamiltonian}, the conjugate gradient method~\cite{nishimura2023prior}, or polynomial approximation methods~\cite{chewi2024query}. This can also be extended to sufficiently small perturbations of quadratic potentials~\cite{schuh2024global}. Moreover,~\cite{chewi2024query} obtained a matching $\Omega(\sqrt\kappa)$ lower bound for sampling from Gaussians in the high-dimensional regime $d\gg \kappa^2$ which holds for all first-order query algorithms.

\par However, whether acceleration is possible in the setting of general (non-quadratic) $\kappa$-conditioned potentials $V$ remained a major open question; see, e.g.,~\cite{Dalalyan19friendly, shen2019randomized, chen2022improved, tian22thesis, gopietal23loglaplace, jiang2023dissipation, wang2022accelerating, camrud2023second, Zhang+23ULMC, AltChe24Warm,apers2024hamiltonian, chewibook, fu2025hamiltonian}.
Indeed, prior to this paper, it was unknown whether a rate $\kappa^{1-\delta}$ was possible for any constant $\delta > 0$. Our result of $\Otilde(\kappa^{5/6})$ for arbitrary $\kappa$-conditioned potentials $V$ answers this in the affirmative and opens the doorway to ballistic acceleration. 

For the convenience of the reader, a further review of the literature on sampling is provided when we develop the applications in \S\ref{sec:app}.

\paragraph*{Shifted divergences and shifted composition.} Our argument builds upon and synthesizes ideas from a growing body of work. The \emph{shifted composition rule}---the namesake of this series and an essential part of our analysis---is a refinement and generalization of the \emph{shifted divergence} technique for differential privacy which was introduced in~\cite{pabi} and later made amenable for sampling applications in~\cite{AltTal22dp,AltTal23Langevin}. Shifted divergences were used to bound the mixing time of~\ref{eq:ULD} in~\cite{AltChe24Warm}, in particular to obtain the first convergence rate in R\'enyi divergence with dimension dependence $\sqrt{d}$, which provided faster warm starts for log-concave sampling. However, the analysis techniques in~\cite{AltChe24Warm} are insufficient to obtain the results in this paper for three fundamental reasons. 
\par First, the analysis in~\cite{AltChe24Warm} crucially relied on contractivity of~\ref{eq:ULD} in order to extend short-time regularity results to long-time mixing results; however,~\ref{eq:ULD} is \emph{not} contractive in the parameter regime ($\gamma \asymp \sqrt{\alpha}$) used for ballistic acceleration in \S\ref{sec:app}. Second, the analysis in~\cite{AltChe24Warm} could not directly account for numerical discretization error, instead relying on isolating this via a triangle inequality and then bounding it via a separate application of Girsanov's theorem. However, such arguments are inapplicable to the discretization schemes which are conjectured to lead to improved $\kappa$ dependence, since those discretization schemes make use of a ``look-ahead step'' for numerical integration, which prevents natural interpolations of the algorithm's iterates from being adapted, and thus makes Girsanov's theorem inapplicable in its standard form. Third, the Harnack inequalities for discretizations of~\ref{eq:ULD} in~\cite{AltChe24Warm} are unable to achieve the Harnack inequalities for the continuous-time diffusion in \S\ref{scn:continuous-time} by passing to the limit as the discretization size $h \searrow 0$. This is because the Harnack inequalities in~\cite{AltChe24Warm} can only exploit the regularity of the first step of the discretized algorithm, which leads to a vanishing amount of regularity as $h \searrow 0$, and thus yields a trivial bound in the continuous-time limit. 

The shifted composition rule generalizes and refines the shifted divergence technique in several key ways; see~\cite{scr1} for a detailed discussion. Most relevant to this paper is that (answering the third issue stated above) it enables sharp Harnack inequalities for discretized algorithms which can recover non-trivial Harnack inequalities for the continuous-time diffusion, as shown for elliptic diffusions
for backward regularity in~\cite{scr1} and forward regularity in~\cite{scr2}; and (answering the first two issues stated above) it does not require Wasserstein contractivity and is applicable to non-adapted discretization schemes such as the randomized midpoint discretization, as shown for elliptic diffusions in~\cite{scr3}. The present paper adapts these analysis techniques to degenerate diffusions, which requires several new ideas; see the technical overview section in \S\ref{ssec:intro:tech}.

\paragraph*{Harnack inequalities.}
Dimension-free Harnack inequalities were first established for elliptic diffusions via semigroup methods in~\cite{Wang1997LSINoncompact} and via coupling in~\cite{ArnThaWan06HarnackCurvUnbdd}.
Since then, there has been a large literature on these inequalities and applications; see, e.g.,~\cite{Wang12Coupling, Wang13HarnackSPDE, Wang14Diffusion} for overviews.
Applications to degenerate diffusions such as~\ref{eq:ULD} have also been considered in works such as~\cite{Pri06DDerivegen, Zhang10BismutHamiltonian, GuiWan12DegenBismut, WanZha13DerivDegen, Wang17HyperHam, LvHua21HarnackSPDEs, HuaMa22DistDepHam}. Finally, the works~\cite{monmarche2024entropic, chak2025reflection} provide Harnack inequalities for variations of the Hamiltonian Monte Carlo algorithm. Among these works, the closest to ours is~\cite{GuiWan12DegenBismut}, which is discussed in \S\ref{sssec:intro:harnack}.

        \section{Preliminaries}\label{sec:prelim}

First, we recall the definition of the $2$-Wasserstein distance. For probability measures $\mu$, $\nu$,
\begin{align*}
    W_2^2(\mu,\nu) \deq \inf_{\gamma \in \Coup(\mu,\nu)} \int \opnorm{x-y}^2\, \gamma(\D x,\D y)\,,
\end{align*}
where $\Coup$ denotes the set of couplings of $\mu$ and $\nu$, i.e., the set of joint distributions $\gamma$ with marginals $\mu$ and $\nu$. Note that in the definition of $W_2$, there is flexibility in the choice of base norm $\opnorm{\cdot}$. Since~\ref{eq:ULD} is contractive in a certain ``twisted norm'' (and not in the standard Euclidean norm), we will often take this choice for $\opnorm{\cdot}$. This will be specified when it is used, and in this case we will denote the Wasserstein distance as $\mc W_2$.

\par Next, we recall relevant preliminaries about R\'enyi divergences.

\begin{defin}[R\'enyi divergence]
    The R\'enyi divergence of order $q > 1$ between two probability measures $\mu$, $\nu$ is defined as
    \begin{align*}
        \Renyi_q(\mu \mmid \nu) \deq \frac{1}{q-1} \log \int \Bigl(\frac{\D \mu}{\D \nu}\Bigr)^q \, \D \nu\,,
    \end{align*}
    if $\mu \ll \nu$, and $+ \infty$ otherwise.
    
\end{defin}

The R\'enyi divergence of order $q = 1$ is interpreted in the limiting sense and coincides with the $\KL$ divergence:
\begin{align*}
    \Ren_1(\mu \mmid \nu) = \KL(\mu \mmid \nu) = \int \Bigl( \frac{\D \mu}{\D \nu} \log \frac{\D \mu}{\D \nu} \Bigr)\, \D \nu\,.
\end{align*}
Another important case of R\'enyi divergences is $q=2$, due to the relation to chi-squared divergence: 
\begin{align*}
    \Ren_2(\mu \mmid \nu) = \log (1 + \chi^2(\mu \mmid \nu)) = \log \int \Bigl( \frac{\D \mu}{\D \nu} \Bigr)^2\, \D \nu\,.
\end{align*}

We make use of the following two basic properties of R\'enyi divergences; proofs and further background on R\'enyi divergences can be found in the surveys~\cite{van2014renyi,mironov2017renyi}. 

\begin{prop}[Data processing inequality]\label{prop:data-processing}
    For any probability measures $\mu$, $\nu$, any R\'enyi order $q \in [1,\infty]$, and any Markov kernel $P$, 
    \begin{align*}
        \Ren_q(\mu P \mmid \nu P) \le \Ren_q(\mu \mmid \nu)\,.
    \end{align*}
\end{prop}

\begin{prop}[Weak triangle inequality]\label{prop:weak-triangle}
     For any probability measures $\mu$, $\nu$, $\pi$, any R\'enyi order $q \in [1,\infty]$, and any relaxation parameter $\lambda \in (0,1)$, 
    \begin{align*}
        \Ren_q(\mu \mmid \pi) \le \frac{q-\lambda}{q-1}\,\Ren_{q/\lambda}(\mu \mmid \nu) + \Ren_{(q-\lambda)/(1-\lambda)}(\nu \mmid \pi)\,.
    \end{align*}
    In particular (by setting $\lambda = 1-\varepsilon$, $q = 1+\varepsilon$, and letting $\varepsilon\searrow 0$),
    \begin{align*}
        \KL(\mu \mmid \pi) \le 2 \KL(\mu \mmid \nu) + \log(1+\chi^2(\nu \mmid \pi))\,.
    \end{align*}
    
\end{prop}

The namesake of this series is the \emph{shifted composition rule} for R\'enyi divergences, introduced in part I~\cite{scr1}. Here we recall only the simpler version for KL divergence---the \emph{shifted chain rule}---as that suffices for the purposes of this paper. This builds upon the standard chain rule for the KL divergence, which we recall is
\begin{align}
	\KL(\mb P^{\msf Y} \mmid \mb Q^{\msf Y})
	\leq
	\KL(\mb P^{\msf X, \msf Y} \mmid \mb Q^{\msf X, \msf Y})
	&= \KL(\mb P^{\msf X} \mmid \mb Q^{\msf X}) +  \int \KL(\mb P^{\msf Y\mid \msf X=x} \mmid \mb Q^{\msf Y\mid \msf X=x}) \, \mb P^{\msf X}(\D x)\,.
	\label{eq:chain-rule-kl}
\end{align}
Here $\msf X$, $\msf Y$ are jointly defined random variables on a standard probability space $\Omega$, and $\mb P$, $\mb Q$ are two probability measures over $\Omega$, with superscripts denoting the laws of random variables under these measures. (Strictly speaking, the equality in the second step is the chain rule, and the inequality in the first step is from the data-processing inequality. We join these two inequalities here because it provides contrast to the shifted chain rule, stated next.)

\begin{theorem}[Shifted chain rule]\label{thm:shifted_chain_rule}
    Let $\msf X$, $\msf X'$, $\msf Y$ be three jointly defined random variables on a standard probability space $\Omega$. Let $\mb P$, $\mb Q$ be two probability measures over $\Omega$, with superscripts denoting the laws of random variables under these measures. Then
	\begin{align*}
		\KL(\mb P^{\msf Y} \mmid \mb Q^{\msf Y})
		&\le \KL(\mb P^{\msf X'} \mmid \mb Q^{\msf X}) + \inf_{\gamma \in \Coup(\mb P^{\msf X}, \mb P^{\msf X'})} \int \KL(\mb P^{\msf Y\mid \msf X=x} \mmid \mb Q^{\msf Y\mid \msf X=x'}) \,\gamma(\D x, \D x')\,.
	\end{align*}
\end{theorem}

The key idea in the shifted chain rule is to introduce an auxiliary, third random variable $\msf X'$. This generalizes the original chain rule (when $\msf X' = \msf X$) and enables many new applications via different choices of $\msf X'$, as developed in this series of papers. This paper uses the flexibility of $\msf X'$ in order to analyze the evolution of two Markov processes updating with different kernels, as this is the setting for analyzing the error from numerically discretizing a diffusion. 

Finally, we mention that throughout, we use the standard notations $\asymp$, $\lesssim$, $\gtrsim$ to suppress universal constants, and we write $\Otilde(\cdot)$ and $\Omega(\cdot)$ to suppress polylogarithmic factors in the argument.

        \section{Harnack inequalities for the underdamped Langevin diffusion}\label{scn:continuous-time}

In this section, we establish reverse transport inequalities along the underdamped Langevin dynamics~\ref{eq:ULD}, which are dual to parabolic Harnack inequalities for the semigroup.

\subsection{Main result}\label{ssec:harnack_main}

We state our main assumption and theorem for~\ref{eq:ULD}.

\begin{assumption}[Hessian bounds]\label{as:regularity}
    Let $V$ be twice continuously differentiable on $\R^{d}$, $\alpha$-convex, and $\beta$-smooth, so that on all of $\R^d$,
    \begin{align*}
        \beta I \succeq \nabla^2 V \succeq \alpha I\,,
    \end{align*}
    for $\infty >\beta \geq \alpha \ge -\beta > -\infty$.
\end{assumption}

Note that $\alpha \in \R$ can have any sign. $V$ is said to be strongly convex, weakly convex, or semi-convex if $\alpha$ is positive, zero, or negative, respectively. The friction is said to be high if $\gamma \gtrsim \sqrt{\beta}$ and low if $\gamma \lesssim \sqrt{\beta}$; the constant $\sqrt{32}$ defining this transition has not been optimized.

\begin{theorem}[Harnack inequalities for the underdamped Langevin dynamics]\label{thm:uld_regularity}
    Adopt Assumption~\ref{as:regularity} and let $\bs P_T$ denote the Markov kernel corresponding to running~\ref{eq:ULD} for time $T > 0$.
    Then, for all $q\ge 1$ and all $x,\bar x,p,\bar p\in\R^d$,
    \begin{align}\label{eq:uld_regularity}
        \Ren_q(\delta_{x,p} \bs P_T \mmid \delta_{\bar x,\bar p}\bs P_T)
        &\le qC(\alpha,\beta,\gamma,T)\,\Bigl\{\norm{x-\bar x}^2 + \frac{1}{\gamma_0^2}\,\norm{p-\bar p}^2\Bigr\}\,,
    \end{align}
    where
    \begin{align*}
        C(\alpha,\beta,\gamma,T)
        \lesssim \frac{1}{\gamma} \,\Bigl(\frac{\omega}{\exp(c\omega T)-1}\Bigr)^3 + \gamma\,\frac{\omega}{\exp(c\omega T)-1}\,, \qquad
        \gamma_0
        &\gtrsim \gamma + \frac{\one_{T\le 1/\abs{\omega}}}{T}\,,
    \end{align*}
    where $c > 0$ is a universal constant ($c=1/48$ suffices) and
    \begin{align*}
        \omega
        \deq
        \begin{cases}
            \alpha/(3\gamma) & \text{ if } \gamma \ge \sqrt{32\beta} \text{ \textbf{ (high friction setting)}}
            \\
            -\sqrt\beta/3 & \text{ if }\gamma < \sqrt{32\beta} \text{ \textbf{ (low friction setting)}}
        \end{cases}
    \end{align*}
\end{theorem}

We make several remarks about Theorem~\ref{thm:uld_regularity}.

\begin{remark}[Short-time vs.\ long-time asymptotics]
    For small $T \lesssim |\omega|^{-1}$, in all cases we have
    \begin{align*}
        C(\alpha,\beta,\gamma,T) \asymp \frac{1}{\gamma T^3} + \frac{\gamma}{T}\,.
    \end{align*}
    The $1/T^3$ scaling reflects the degenerate hypoelliptic nature of~\ref{eq:ULD} and is sharp (e.g., when $V$ is a quadratic).
    Note that in the short-time regime $T\lesssim \gamma^{-1} \wedge |\omega|^{-1}$, we have $\gamma_0 \asymp 1/T$, thus the right-hand side of~\eqref{eq:uld_regularity} scales as $\gamma^{-1}\,\{\norm{x-\bar x}^2/T^3 + \norm{p-\bar p}^2/T\}$. This is sharp, as shown in Example~\ref{ex:opt-shift} below.

    On the other hand, for large $T$, the strongly convex, weakly convex and semi-convex/low friction cases correspond respectively to exponential decay $e^{-c\absomega T}$, polynomial decay $1/T$, and no decay. Note that the bounds are continuous in $\alpha$, and in particular the bounds for the weakly convex setting $(\alpha=0)$ coincide with the bounds for the strongly convex $(\alpha > 0)$ and semi-convex $(\alpha < 0)$ settings in the limit as $\alpha \to 0$. Note also that the low-friction setting is separated since then~\ref{eq:ULD} fails to be contractive---an important ingredient in our analysis---regardless of convexity $\alpha \ge 0$. 
\end{remark}

\begin{remark}[Exponential decay in the strongly convex case]\label{rem:harnack-decay-comparison}
    In the contractive case---namely, $\alpha > 0$ and $\gamma \asymp \sqrt \beta$---it is well-known that one has the long-time exponential contraction in Wasserstein distance:
    \begin{align*}
        W_\infty(\bs\mu \bs P_T, \bs\nu \bs P_T)
        \lesssim \exp\bigl(-\frac{c\alpha T}{\gamma}\bigr)\,W_\infty(\bs\mu,\bs\nu)\,.
    \end{align*}
    When combining this inequality with a short-time Harnack inequality, say for time $T=1$, this yields a contractive Harnack inequality as $T \to \infty$:
    \begin{align*}
        \Ren_q(\bs\mu\bs P_{T+1} \mmid \bs\nu\bs P_{T+1})
        &\lesssim C(\alpha,\beta,\gamma, 1)\,W_\infty^2(\bs\mu \bs P_T,\bs \nu \bs P_T)
        \lesssim C(\alpha,\beta,\gamma, 1)\exp\bigl(-\frac{c\alpha T}{\gamma}\bigr)\,W_\infty^2(\bs\mu,\bs\nu)\,.
    \end{align*}
    This provides an alternative approach in the case $\alpha > 0$, but fails to produce a bound which tends to zero as $T\to\infty$ in the weakly convex case $\alpha = 0$.
    In contrast, our method provides a unified approach to produce Harnack inequalities in all settings, which directly yields exponential decay in the strongly convex case and polynomial decay in the weakly convex case.
\end{remark}

\begin{remark}[Random initializations]
    Although Theorem~\ref{thm:uld_regularity} is stated for deterministic initializations $(x,p)$ and $(\bar x, \bar p)$, it can be immediately upgraded to a statement for random initializations by using the joint convexity of information divergences; see~\cite[\S 3.3]{scr1}.
\end{remark}

\begin{remark}[Harnack inequalities]\label{rem:harnack-equiv}
    The reason why we refer to the results of Theorem~\ref{thm:uld_regularity} as ``Harnack inequalities'' is that by duality, the statement for $q=1$ is equivalent to the following log-Harnack inequality: for all positive measurable functions $f : \R^d\times \R^d\to\R_{>0}$,
    \begin{align*}
        \bs P_T(\log f)(x,p) \le \log \bs P_T f(\bar x,\bar p) + C(\alpha,\beta,\gamma,T)\,\bigl\{\norm{x-\bar x}^2 + \frac{1}{\gamma_0^2}\,\norm{p-\bar p}^2\bigr\}\,.
    \end{align*}
    Similarly, the statement for $q > 1$ is equivalent to the following statement for $q^* \deq q/(q-1)$: for all positive measurable functions $f : \R^d\times\R^d\to\R_{>0}$,
    \begin{align*}
        (\bs P_T f(x,p))^{q^*}
        &\le \bs P_T(f^{q^*})(\bar x,\bar p) \exp\Bigl(q^* C(\alpha,\beta,\gamma,T)\,\bigl\{\norm{x-\bar x}^2 + \frac{1}{\gamma_0^2}\,\norm{p-\bar p}^2\bigr\}\Bigr)\,.
    \end{align*}
    See, e.g.,~\cite[\S 6.1--6.2]{scr1} for further discussion.
\end{remark}

\subsection{Construction of auxiliary process}\label{ssec:harnack_overview}

See \S\ref{ssec:intro:tech} for a high-level overview of our analysis. Here we briefly recall the key definitions from there, and then in the rest of the section we carry out this analysis. As explained there, our analysis builds upon the coupling argument for uniformly elliptic diffusions from~\cite{ArnThaWan06HarnackCurvUnbdd} (see also the exposition in~\cite[\S 4.1]{scr1}), enhancing it so that it can apply to degenerate diffusions. We therefore refer to those papers for further background and context, and in particular here we omit routine details justifying, e.g., existence and uniqueness of the SDE systems. Here we restrict to $\KL$ since the extension to R\'enyi divergences follows by a similar argument.

\paragraph*{Setup.} Consider two copies $\{(X_t, P_t)\}_{t \geq 0}$ and $ \{(\bar X_t, \bar P_t)\}_{t \geq 0}$ of~\ref{eq:ULD}, started at $(x,p)$ and $(\bar x,\bar p)$ respectively. %
Theorem~\ref{thm:uld_regularity} requires establishing a reverse transport inequality of the form
\begin{align*}
    \KL(\law(X_T, P_T) \mmid \law(\bar X_T, \bar P_T)) \leq C(T)\, \opnorm{(x,p) - (\bar x, \bar p)}^2
\end{align*}
with respect to some norm $\opnorm\cdot$,
for an appropriate quantity $C$ that depends on the time $T$ as well as other relevant properties of the diffusion.

\paragraph*{Auxiliary process.} 
Our analysis is based on the design of an auxiliary third process that starts at $(\bar x,\bar p)$ and hits $(X_T, P_T)$ at a prespecified time $T > 0$. Let $\eta_\cdot^\msx: [0,T) \to \R_+$, $\eta_\cdot^\msp: [0,T) \to \R_+$ be two smooth functions, chosen so that both grow to $+\infty$ as $t \nearrow T$.
We define the system $\{(X_t^\aux, P_t^\aux)\}_{t\ge 0}$ with initial condition $(\bar x,\bar p)$ to evolve according to
\begin{align}\label{eq:shifted_uld}
\begin{aligned}
    \D X^\aux_t &= P_t^\aux \, \D t\,, \\
    \D P^\aux_t &= \bigl(- \nabla V(X_t^\aux)-\gamma P_t^\aux + \eta_t^\msx\,(X_t - X_t^\aux) + \eta_t^\msp\,(P_t - P_t^\aux) \bigr) \, \D t + \sqrt{2\gamma} \, \D B_t\,.
\end{aligned}
\end{align}
We choose $\eta^\msx$, $\eta^\msp$ so that $(X_t^\aux, P_t^\aux) \to (X_T, P_T)$ as $t\nearrow T$.

\paragraph*{Girsanov's theorem on the auxiliary process.} By this interpolation condition and a standard application of Girsanov's theorem~\cite[\S 5.5]{legall2016stochasticcalc},
\begin{align}
    \KL(\law(X_T, P_T) \mmid \law(\bar X_T, \bar P_T))
    &= \KL(\law(X_T^\aux, P_T^\aux) \mmid \law(\bar X_T, \bar P_T)) \nonumber\\
    &\le \frac{1}{4\gamma} \int_0^T \E[\norm{\eta_t^\msx \,(X_t - X_t^\aux) + \eta_t^\msp \,(P_t - P_t^\aux)}^2]\,\D t\,.\label{eq:cont_girsanov}
\end{align}
The core of our analysis is obtaining bounds on $\E[\norm{X_t - X_t^\aux}^2]$ and $\E[\norm{P_t - P_t^\aux}^2]$.

\paragraph*{Time-dependent twisted coordinates and effective friction.} As discussed in \S\ref{ssec:intro:tech}, underdamped dynamics are \emph{not} contractive in the standard coordinate system, even in the setting of strongly convex potentials ($\nabla^2 V \succeq \alpha I$ for $\alpha > 0$). Instead, the standard coupling analysis for~\ref{eq:ULD} proceeds in the \emph{twisted coordinate system} $(x, x + \tfrac{2}{\gamma}\, p)$, which does lead to contraction for $\alpha > 0$ and suitable $\gamma$; see, e.g.,~\cite{eberle2019couplings}. To account for the additional shifting dynamics
when we couple $\{(X_t, P_t)\}_{t\ge 0}$ with the auxiliary process $\{(X_t^{\aux}, P_t^{\aux})\}_{t\ge 0}$, we consider the \emph{time-dependent twisted coordinate system}  $(x_t, z_t) \deq (x_t, x_t + \frac{2}{\gamma_t}\,p_t)$ with \emph{effective friction} $\gamma_t \deq \gamma + \eta_t^\msp$. Note that as $t\nearrow T$, the coordinate system degenerates to $(x_T, x_T)$, which is handled in the proof of Theorem~\ref{thm:uld_regularity} via a limiting argument. For ease of recall, we summarize this change of coordinates in the following matrix form: 
\begin{align}
\begin{bmatrix}
	X_t \\ Z_t
\end{bmatrix}
=
\underbrace{\begin{bmatrix}
		I & 0 \\ I & (2/\gamma_t)\, I
\end{bmatrix}}_{A_t}
\begin{bmatrix}
	X_t \\ P_t
\end{bmatrix} \qquad \text{ where } \qquad \gamma_t \deq \gamma+ \eta_t^\msp\,.\label{eq:def-coordinates}
\end{align}

\subsection{Warm up: integrated Brownian motion}

As an illustration of how these shifts can be defined, in the following lemma we consider a simplified version of~\ref{eq:ULD} where the potentials and friction are set to zero. 

\begin{example}[Optimal shifts for simplified dynamics]\label{ex:opt-shift}
    Consider $\{(X_t, P_t)\}_{t\ge 0}$, $\{(\bar X_t, \bar P_t)\}_{t\ge 0}$ defined using the following equation, starting at $(x, p)$, $(\bar x, \bar p)$ respectively:
    \begin{align*}
        \D X_t = P_t \,\D t\,, \qquad \D P_t = \sqrt{2\gamma}\,\D B_t\,.
    \end{align*}
Then, the optimal shift (for minimizing the KL bound~\eqref{eq:cont_girsanov}) can be shown to be
\begin{align*}
    \eta_t^\msp = \frac{4}{T-t}\,, \qquad \eta_t^\msx = \frac{6}{(T-t)^2}\,,
\end{align*}
where $\{(X_t^\aux, P_t^\aux)\}_{t\ge 0}$ is defined similarly to~\eqref{eq:shifted_uld}.
This yields, via Girsanov's theorem,
\begin{align*}
    \KL(\operatorname{law}(X_T, P_T) \mmid \operatorname{law}(\bar X_T, \bar P_T)) \leq \frac{\norm{p-\bar p}^2}{\gamma T} + \frac{3 \inner{p- \bar p, x- \bar x}}{\gamma T^2} + \frac{3\,\norm{x - \bar x}^2}{\gamma T^3}\,.
\end{align*}
Moreover, this bound is actually an equality, as can be verified from an exact computation.
See \S\ref{app:opt_shift} for the details for this example including the derivation of the optimal shift.
\end{example}

While this example considers a simplified version of~\ref{eq:ULD}, it suggests several salient points for~\ref{eq:ULD} in the weakly convex setting. 
\begin{enumerate}
    \item The optimal shift---which a priori could be an arbitrary function of $X_t, P_t,X_t^\aux, P_t^\aux$---indeed decomposes into the  form $\eta_t^\msx\,(X_t-X_t^\aux) + \eta_t^\msp\,(P_t - P_t^\aux)$.
    
    \item The shifts scale as $\eta_t^\msp \asymp \frac{1}{T-t}$ and $\eta_t^\msx \asymp (\eta_t^\msp)^2 \asymp \frac{1}{(T-t)^2}$.
    
    \item The resulting bounds can be sharp.
\end{enumerate}  

\subsection{Distance bound for auxiliary process}

\subsubsection{Distance decay identity}

Before making substitutions, we work out some intermediate steps for general choices of $\eta^\msx$, $\eta^\msp$, $\gamma$.

\begin{lemma}[Distance decay identity]\label{lem:decay}
	For the processes defined in \S\ref{ssec:harnack_overview},
    \begin{align*}
        \D \norm{(X_t, Z_t) - (X_t^\aux, Z_t^\aux)}^2 = -2\,\langle (X_t - X_t^\aux, Z_t - Z_t^\aux), \mc M_t\, (X_t - X_t^\aux, Z_t - Z_t^\aux)\rangle\,\D t\,,
    \end{align*}
    where $\mc M_t$ is a time-dependent, symmetric matrix which is given by
    \begin{align*}
        \mc M_t  \deq \begin{bmatrix}
            (\gamma_t/2)\, I & (-\gamma_t/2 + \eta_t^\msx/\gamma_t - \dot \gamma_t/(2\gamma_t))\, I + \mc H_t/\gamma_t \\
            * & (\gamma_t/2 + \dot\gamma_t/\gamma_t)\,I
        \end{bmatrix}\,,
    \end{align*}
    where $\alpha I \preceq \mc H_t \preceq \beta I$. 
\end{lemma}

\begin{proof}
	By It\^o's formula, 
	\begin{align}
		\D \Bigl\lVert
			\begin{bmatrix}
				X_t - X_t^{\aux} \\ Z_t - Z_t^{\aux}
			\end{bmatrix}
		\Bigr\rVert^2
		=
		2\, \Bigl\langle \begin{bmatrix}
			X_t - X_t^{\aux} \\ Z_t - Z_t^{\aux}
		\end{bmatrix},\, 
			\D \begin{bmatrix}
				X_t - X_t^{\aux} \\ Z_t - Z_t^{\aux}
			\end{bmatrix}
	 \Bigr\rangle\,,\label{eq:lem:decay}
	\end{align}
    since by synchronous coupling, $\D (X_t-X_t^\aux, Z_t - Z_t^\aux)$ has no stochastic term.
	To compute the latter quantity, note that by coupling~\ref{eq:ULD} with \eqref{eq:shifted_uld}, 
	\begin{align*}
		\D \begin{bmatrix}
			X_t - X_t^{\aux} \\ P_t - P_t^{\aux}
		\end{bmatrix}
		=
		\underbrace{\begin{bmatrix}
			0 & I \\ -\eta_t^\msx I - \cH_t & -\gamma_t I
		\end{bmatrix}}_{D_t}
		 \begin{bmatrix}
			X_t - X_t^{\aux} \\ P_t - P_t^{\aux}
		\end{bmatrix}
		\,, 
	\end{align*}
	where
    $\nabla V(X_t) - \nabla V(X_t^{\aux}) = \mc H_t\, (X_t - X_t^{\aux})$ with
    $\mc H_t = \int_0^1 \nabla^2 V((1-s)\,X_t + s\,X_t^{\aux}) \,\D s$
    by the fundamental theorem of calculus.
    Note that $\alpha I \preceq \cH_t \preceq \beta I$ by Assumption \ref{as:regularity}. By the change of coordinates~\eqref{eq:def-coordinates}, 
	\begin{align*}
		\D \begin{bmatrix}
			X_t - X_t^{\aux} \\ Z_t - Z_t^{\aux}
		\end{bmatrix}
		&=
		\D  \biggl( A_t \begin{bmatrix}
			X_t - X_t^{\aux} \\ P_t - P_t^{\aux}
		\end{bmatrix} \biggr)
		=
		\biggl( (\dot{A_t}  + A_t D_t) \begin{bmatrix}
			X_t - X_t^{\aux} \\ P_t - P_t^{\aux}
		\end{bmatrix} \biggr)\, \D t \\[0.25em]
		&=
			\biggl( ( \dot{A_t}  + A_t D_t) A_t^{-1} \begin{bmatrix}
			X_t - X_t^{\aux} \\ Z_t - Z_t^{\aux}
		\end{bmatrix} \biggr)\, \D t\,.
	\end{align*}
	The proof is complete by a simple computation of this $2 \times 2$ block matrix $(\dot{A_t}  + A_t D_t) A_t^{-1}$, plugging into~\eqref{eq:lem:decay}, and symmetrizing (which does not affect the resulting quadratic form). 
\end{proof}

\subsubsection{Contraction}\label{scn:simplify}

We now specify the choice of drift in the auxiliary process~\eqref{eq:shifted_uld} by defining the shifts $\eta_t^\msx$ and $\eta_t^\msp$. Here $c_0$ is an absolute constant, which we have not optimized and is taken to be at least $24$ below.

\begin{defin}[Shift parameters]\label{def:ct-shifts}
    For $t\in [0,T)$, we set
	\begin{align*}
		\eta_t^\msx \deq \frac{\gamma_t \eta_t^\msp}{2}\,, \qquad \eta_t^\msp \deq \frac{c_0 \omega}{\exp(\omega\,(T-t))-1}\,.
	\end{align*}
\end{defin}

Recall the definition of $\omega$ from Theorem~\ref{thm:uld_regularity}; define also the shorthand $\omega_+ \deq \omega \vee 0$.

\begin{figure}[htbp]
  \centering
  \begin{subfigure}[b]{0.48\textwidth}
    \centering
    \includegraphics[width=\linewidth]{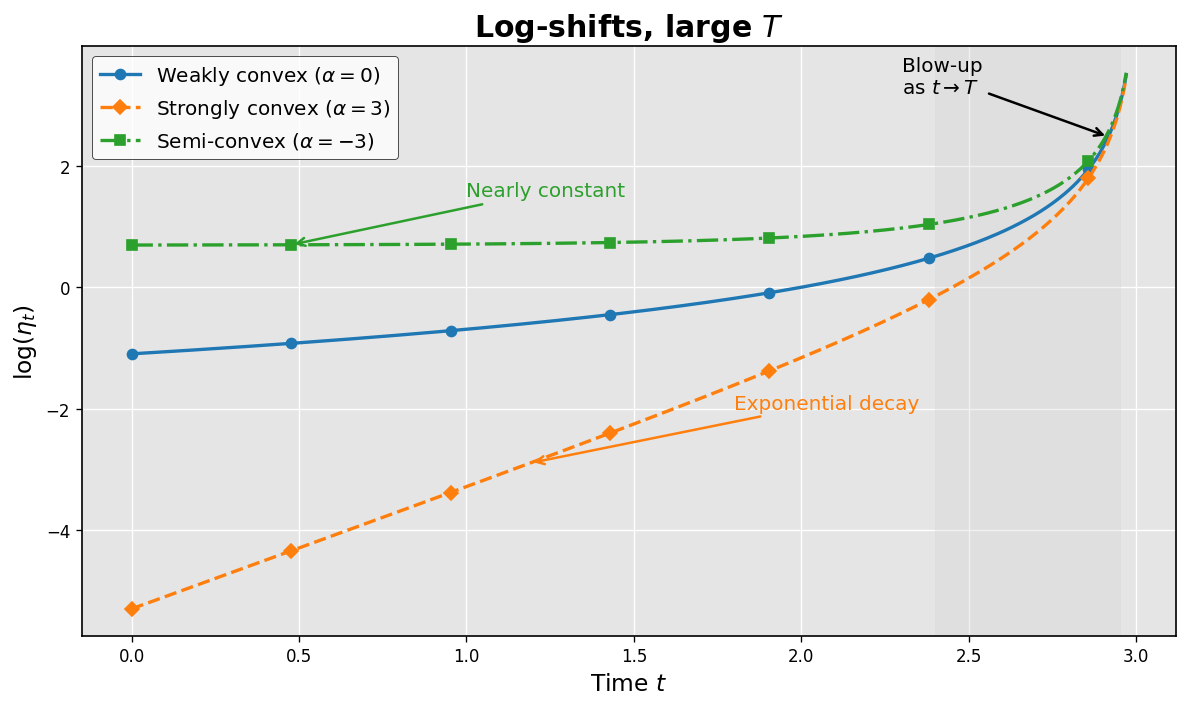}
  \end{subfigure}
  \hfill
  \begin{subfigure}[b]{0.48\textwidth}
    \centering
    \includegraphics[width=\linewidth]{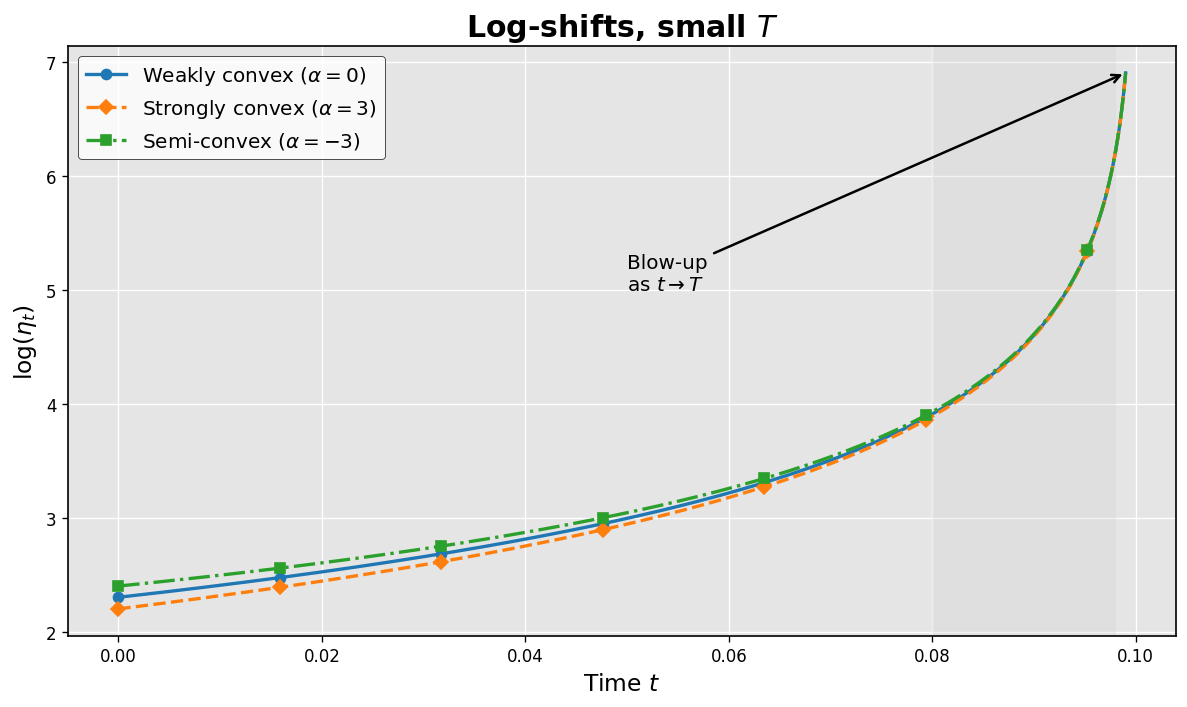}
  \end{subfigure}
  \caption{\footnotesize Shift schedules $\{\eta_t^\msp\}_{t\in [0,T]}$ from Definition~\ref{def:ct-shifts}, for different convexity parameters $\alpha$, and with values of moderate $T = 3$ (left) and small $T=0.1$ (right). These are all in the high friction setting.}\label{fig:shift-sched}
\end{figure}

\begin{intuition}
    These shifts are illustrated in Figure~\ref{fig:shift-sched}. They are chosen so that the KL bound~\eqref{eq:cont_girsanov} arising from Girsanov's theorem will be small. In the overdamped setting, this variational problem of optimizing shifts to minimize the Girsanov bound is computable in closed form~\cite{scr1}; here the variational problem is significantly more complicated, hence we simply state these shifts without proving optimality (but see Example~\ref{ex:opt-shift} for a derivation of the optimal shifts in a simpler setting).
    
    In all regimes, the shift is $\eta_t^\msp \asymp 1/(T-t)$ in the final time interval $t \ge T-\Theta(\absomega^{-1})$ in order to ensure the interpolation condition $(X_t^\aux, P_t^\aux) \to (X_T, P_T)$ as $t \nearrow T$. However, for earlier times $t \le T-\Theta(\absomega^{-1})$, the shift is exponentially small $\eta_t^\msp \asymp \absomega e^{-\absomega (T-t)}$ in the case $\omega > 0$, whereas it is $\eta_t^\msp \asymp \absomega$ in the case $\omega < 0$. This discrepancy is for a fundamental reason: the dynamics are contractive in the case $\omega > 0$, whereas they are expansive in the case $\omega < 0$. Intuitively, in the contractive setting, the auxiliary process can invest exponentially less ``effort'' at early times to achieve the final interpolation condition, since contractivity inflates this effort to be exponentially more valuable over time.

\end{intuition}

\begin{intuition}
    In the low friction regime, the shifts are of order $\sqrt{\beta}$ before the final time interval. This is because the effective friction $\gamma + \eta_t^\msp$ must be on this scale in order to drive the auxiliary process toward $(X_T,P_T)$. However, as $\gamma$ alone is not large enough to do this, $\eta_t^\msp$ must pick up the slack. 
\end{intuition}

\begin{intuition}
    In the KL bound~\eqref{eq:cont_girsanov} arising from Girsanov's theorem, we pay for the size $\eta_t^\msx \, (X_t - X_t^\aux) + \eta_t^\msp \, (P_t - P_t^\aux)$ of the shift.
    We can rewrite this as $(\eta_t^\msx - \frac{\gamma_t \eta_t^\msp}{2})\,(X_t - X_t^\aux) + \frac{\gamma_t \eta_t^\msp}{2}\,(Z_t - Z_t^\aux)$.
    From this expression, it makes sense to take $\eta_t^\msx = \frac{1}{2}\,\gamma_t \eta_t^\msp$ so that the first term is as small as possible. Because $\eta_t^\msx$ is a simple function of $\eta_t^\msp$, we will often use $\eta_t$ to refer to $\eta_t^\msp$, clarifying where needed.
\end{intuition}

We now lower bound the eigenvalues of $\cM_t$. By Lemma~\ref{lem:decay}, this ensures exponential decay of $\norm{(X_t, Z_t) - (X_t^\aux, Z_t^\aux)}^2$.

\begin{lemma}[Contraction for the shifted system]\label{lem:contraction-cont-time}
    Let $c_0 \geq 24$. In all cases, the choice of shift parameters in Definition~\ref{def:ct-shifts} ensures
    \begin{align*}
        \lambda_{\min}(\mc M_t) \geq \frac{\omega_+}{2} + \frac{\eta_t^\msp}{48}\,.
    \end{align*}
\end{lemma}

We first state a helper lemma which conveniently bounds some terms in $\cM_t$. The proof is straightforward by plugging in the definition of the shifts; details are deferred to \S\ref{app:lemma_bt}. 

\begin{lemma}[Helper lemma for Lemma~\ref{lem:contraction-cont-time}]\label{lem:random-bt-property}
    Let $c_0 \ge 24$.
    In all cases, the choice of shift parameters in Definition~\ref{def:ct-shifts} ensures, for all $\lambda \in [\alpha,\beta]$:
    \begin{enumerate}[label=(\alph*)]
        \item $\dot \gamma_t/\gamma_t > 0$.
    
        \item $b_t^\lambda \ge {\alpha\one_{\omega>0}} + \gamma_t \eta_t/8$.
        
        \item $b_t^\lambda/\gamma_t^2 \leq \frac{3}{4}$.
    \end{enumerate}
    Here, $b_t^\lambda \deq \lambda + \eta_t^\msx - \dot\gamma_t/2 = \lambda + \gamma_t \eta_t^\msp/2 - \dot\eta_t^\msp/2$. %
\end{lemma}

\begin{proof}[Proof of Lemma~\ref{lem:contraction-cont-time}]
We begin with two observations that simplify the analysis. First, observe that we can drop the $\dot\gamma_t/\gamma_t$ term in the bottom-right block of $\mc M_t$ as this term is non-negative. Second, observe that we can diagonalize $\mc M_t$ into $2 \times 2$ diagonal blocks of the form
\begin{align*}
    \mc M_t^\lambda = \begin{bmatrix}
        \gamma_t/2 & b_t^\lambda/\gamma_t -\gamma_t/2 \\ b_t^\lambda/\gamma_t-\gamma_t/2 & \gamma_t/2
    \end{bmatrix}\,,
\end{align*}
where $b_t^\lambda = \lambda + \eta_t^\msx - \dot\gamma_t/2 = \lambda + \gamma_t \eta_t^\msp/2 - \dot\eta_t^\msp/2$ and $\lambda \in [\alpha,\beta]$ is an eigenvalue of $\mc H_t$.
The eigenvalues of $\mc M_t^\lambda$ are $b_t^\lambda/\gamma_t$ and $\gamma_t - b_t^\lambda/\gamma_t$.
By Lemma~\ref{lem:random-bt-property}, $b_t^\lambda/\gamma_t \le \frac{3}{4}\,\gamma_t$ and $b_t^\lambda/\gamma_t \geq 0$, so that the minimum eigenvalue is at least $b_t^\lambda/(3\gamma_t)$.
By Lemma~\ref{lem:random-bt-property}, this implies
\begin{align}\label{eq:diffusion_rate_estimate}
    \lambda_{\min}(\mc M_t^\lambda)
    &\ge \frac{1}{3}\,\bigl(\frac{{\alpha\one_{\omega>0}}}{\gamma_t} + \frac{\eta_t}{8}\bigr)\,.
\end{align}
Finally, in the case $\omega > 0$, we want to simplify the rate by replacing $\gamma_t$ in the denominator with $\gamma$. If $\eta_t \le \gamma$, then we can replace the denominator of the first term in~\eqref{eq:diffusion_rate_estimate} by $2\gamma$. Otherwise, if $\eta_t \ge \gamma$, then $\eta_t \ge \gamma/2 + \eta_t/2 \ge \sqrt{32\beta}/2 + \eta_t/2 \ge (\alpha/\gamma + \eta_t)/2$. As a result, in either case we have
\begin{align*}
    \lambda_{\min}(\mc M_t^\lambda)
    &\ge \frac{\omega_+}{2} + \frac{\eta_t}{48}\,.
\end{align*}
This lower bound also holds in all of the other cases.
\end{proof}

\subsection{KL bounds via shifted Girsanov}

\begin{proof}[Proof of Theorem~\ref{thm:uld_regularity}]
    In this proof, we focus on the case $q=1$, i.e., we bound the KL divergence.
    The extension to R\'enyi divergences of order $q>1$ follows from the same argument as~\cite[Appendix A, \S D]{scr1}, since Lemma~\ref{lem:contraction-cont-time} yields almost sure contraction.
    We omit the details.

    By Girsanov's theorem in the form~\eqref{eq:cont_girsanov},
    \begin{align*}
        \KL(\law(X_T,P_T)\mmid \law(\bar X_T,\bar P_T))
        &\le \frac{1}{4\gamma} \int_0^T \E[\norm{\eta_t^\msx \,(X_t - X_t^\aux) + \eta_t^\msp \,(P_t - P_t^\aux)}^2]\,\D t \\
        &\lesssim \frac{1}{\gamma} \int_0^T \gamma_t^2\,(\eta_t^\msp)^2\, \underbrace{\E[\norm{(X_t, Z_t) - (X_t^\aux, Z_t^\aux)}^2]}_{\eqqcolon \mathtt d_t^2}\,\D t\,.
    \end{align*}
    Actually, since the coordinate system degenerates as $t\nearrow T$, the inequality requires a bit of justification.
    For any small $\delta > 0$, we can apply Girsanov's theorem to the stochastic processes on the time interval $[0,T-\delta]$ in order to bound $\KL(\law(X_{T-\delta}^\aux, P_{T-\delta}^\aux) \mmid \law(\bar X_{T-\delta}, \bar P_{T-\delta}))$.
    Below, we check that $(X_{T-\delta}^\aux, P_{T-\delta}^\aux) \to (X_T, P_T)$ in $L^2$ as $\delta \searrow 0$, which furnishes the desired bound by the weak lower semicontinuity of the KL divergence.
    
    By Lemma~\ref{lem:contraction-cont-time},
    \begin{align}\label{eq:cont_dist_bd}
        \mathtt d_t \le \exp\Bigl(-c'\int_0^t (\omega_+ + \eta_s^\msp)\,\D s\Bigr)\,\mathtt d_0\,,
    \end{align}
    where $c' > 0$ is an absolute constant (we can take $c' = 1/48$).
    To prove Theorem~\ref{thm:uld_regularity}, it remains to substitute in the values of $\eta_t^\msp$, $\gamma_t$, compute the final KL bound, and verify that the auxiliary process hits $(X_T, P_T)$.
    For pedagogical reasons, we carry this out here in the simplest case, namely the weakly convex case $\alpha = 0$ with high friction, deferring the other settings to \S\ref{app:uld_regularity}.

    In this setting,~\eqref{eq:cont_dist_bd} yields
    \begin{align*}
        \mathtt d_t \le \exp\Bigl(-c' \int_0^t \frac{c_0}{T-s}\,\D s\Bigr)\,d_0
        = \Bigl( \frac{T-t}{T}\Bigr)^{c'c_0}\,\mathtt d_0\,.
    \end{align*}
    Therefore, for $2c'c_0 > 3$ ($c_0 > 144$ suffices),
    \begin{align*}
        &\KL(\law(X_T,P_T)\mmid \law(\bar X_T,\bar P_T))
        \lesssim \frac{\mathtt d_0^2}{\gamma} \int_0^T \Bigl(\gamma^2 + \frac{1}{{(T-t)}^2} \Bigr)\,\frac{1}{{(T-t)}^2} \,\Bigl(\frac{T-t}{T}\Bigr)^{2c'c_0}\,\D t \\
        &\qquad = \frac{\gamma \mathtt d_0^2}{T^{c'c_0}} \int_0^T {(T-t)}^{2c'c_0-2}\,\D t + \frac{\mathtt d_0^2}{\gamma T^{c'c_0}} \int_0^T {(T-t)}^{2c'c_0-4}\,\D t
        \lesssim \Bigl(\frac{\gamma}{T} + \frac{1}{\gamma T^3}\Bigr)\,\mathtt d_0^2\,.
    \end{align*}
    Note that $\mathtt d_0^2 \lesssim \norm{x-\bar x}^2 + \gamma_0^{-2}\,\norm{p-\bar p}^2$, where $\gamma_0 = \gamma + 1/T$.

    Finally, to argue that the auxiliary process hits $(X_T, P_T)$, note that
    \begin{align*}
        \E[\norm{X_{T-\delta} - X_{T-\delta}^\aux}^2] + \E[\norm{P_{T-\delta} - P_{T-\delta}^\aux}^2]
        &\lesssim (1 \vee \gamma_{T-\delta}^2)\,\mathtt d_{T-\delta}^2
        = O(\delta^{2c'c_0-2}) \qquad\text{as}~\delta\searrow 0\,,
    \end{align*}
    so that $(X_{T-\delta}^\aux, P_{T-\delta}^\aux) \to (X_T, P_T)$ in $L^2$ for $c'c_0 > 2$.
\end{proof}

        \section{KL local error framework for underdamped Langevin discretizations}\label{sec:discretization}

In this section, we develop a local error framework in KL divergence for numerical discretizations of~\ref{eq:ULD}.
This extends the framework in~\cite{scr3} from elliptic diffusions to degenerate diffusions.

\subsection{Main result}

\begin{theorem}[KL local error framework for~\ref{eq:ULD}]\label{thm:kl_local_error}
    Let $\bs P$ denote the Markov kernel for~\ref{eq:ULD}, run for time $h$, and let $\hat{\bs P}$, $\hat{\bs P}{}'$ be two other Markov kernels (representing numerical discretizations). Assume that for all $x,\tilde x,p,\tilde p\in\R^d$, there are random variables $(X_h, P_h) \sim \delta_{x,p}\bs P$, $(\hat X_h,\hat P_h) \sim \delta_{x,p} \hat{\bs P}$ such that the following conditions hold.
    \begin{enumerate}
        \item (weak error) $h^{-1}\,\norm{\E\hat X_h - \E X_h} \vee \norm{\E\hat P_h - \E P_h} \le \mc E^{\rm w}(x,p)$.
        \item (strong error) $h^{-1}\,\norm{\hat X_h - X_h}_{L^2} \vee \norm{\hat P_h - P_h}_{L^2} \le \mc E^{\rm s}(x,p)$.
        \item (cross-regularity for $\hat{\bs P}{}'$) $\KL(\delta_{x,p} \hat{\bs P}{}' \mmid \delta_{\tilde x,\tilde p} \bs P) \lesssim (\gamma h^3)^{-1}\,\norm{x-\tilde x}^2 + (\gamma h)^{-1}\,\norm{p-\tilde p}^2 + {b(x,p)}^2$.
    \end{enumerate}
    Above, for $\mc F \in \{b, \mc E^{\rm w}, \mc E^{\rm s}\}$, we let $\bar{\mc F} \deq \max_{n=0,1,\dotsc,N-1}{\norm{\mc F}_{L^2(\bs\mu\hat{\bs P}{}^n)}}$.
    Adopt Assumption~\ref{as:regularity} and assume that $h\lesssim \gamma^{-1} \wedge \gamma/\beta$ for a sufficiently small implied absolute constant.
    Then, for all $\bs\mu,\bs\nu \in \mc P_2(\R^d\times\R^d)$,
    \begin{align*}
        \KL(\bs\mu \hat{\bs P}{}^{N-1} \hat{\bs P}{}' \mmid \bs\nu\bs P^N)
        &\lesssim C(\alpha,\beta,\gamma,Nh)\,\mc W_2^2(\bs \mu,\bs \nu) + \mathtt{Err} + \bar b^2\,.
    \end{align*}
    Here, the constant $C(\alpha,\beta,\gamma, Nh)$ is defined in Theorem~\ref{thm:uld_regularity}, $\mc W_2$ denotes $2$-Wasserstein distance in the twisted norm $(x,p)\mapsto \sqrt{\norm x^2 + \gamma_0^{-2}\,\norm p^2}$, and $\mathtt{Err}$ is bounded in the following cases. Denote $T = Nh$.
    \begin{enumerate}[label=(\roman*)]
        \item \textbf{Strongly convex and high friction setting:} Let $\alpha > 0$, $\gamma = \sqrt{32\beta}$, $\omega \deq \alpha/(3\gamma)$, and assume additionally that $h \lesssim 1/(\beta^{1/2}\kappa)$.
        Then,
        \begin{align*}
            \mathtt{Err}
            &\lesssim \frac{1}{\alpha h^2}\,(\bar{\mc E}^{\rm w})^2 + \Bigl[\frac{1}{\beta^{1/2} h} \log\frac{1}{\omega h}\Bigr]\,(\bar{\mc E}^{\rm s})^2\,.
        \end{align*}
        \item \textbf{Weakly convex and high friction setting:} Let $\alpha = 0$ and $\gamma = \sqrt{32\beta}$. Then,
        \begin{align*}
            \mathtt{Err}
            &\lesssim \frac{T}{\beta^{1/2} h^2}\,(\bar{\mc E}^{\rm w})^2 + \Bigl[\frac{1}{\beta^{1/2} h} \log\frac{T}{h} + \beta^{1/2} T\Bigr]\,(\bar{\mc E}^{\rm s})^2\,.
        \end{align*}
        \item \textbf{Semi-convex setting:} %
        Let $\alpha = -\beta$ and $\gamma \le \sqrt{32\beta}$.
        Then,
        \begin{align*}
            \mathtt{Err}
            &\lesssim \frac{T}{\gamma h^2}\,(\bar{\mc E}^{\rm w})^2 + \frac{1}{\gamma h}\, \Bigl[\log\frac{\beta^{-1/2} \wedge T}{h} + \beta^{1/2} T\Bigr]\,(\bar{\mc E}^{\rm s})^2\,.
        \end{align*}
    \end{enumerate}
\end{theorem}

This ``master theorem'' provides a systematic framework to analyze the discretization error of various schemes and settings, as we illustrate in \S\ref{sec:app}.
It captures the dependence of the final error on three key factors:
\begin{itemize}
    \item The first term $C(\alpha,\beta,\gamma,Nh)\,\mc W_2^2(\bs\mu,\bs\nu)$ captures the dependence on the \emph{initialization} and it is inherited from Theorem~\ref{thm:uld_regularity}.
    \item The second term $\mathtt{Err}$ captures the dependence on the \emph{local errors} $\mc E^{\rm w}$, $\mc E^{\rm s}$ and constitutes the heart of the theorem.
    Note that in the context of~\ref{eq:ULD}, there are local errors associated with both the position and momentum coordinates.
    Since the position integrates the momentum, for simplicity we have implicitly assumed that the position local errors are a factor of $h$ smaller than the momentum local errors, which is satisfied for the discretizations in \S\ref{sec:app}.
    \item The final term $\bar b^2$ captures the dependence on the \emph{cross-regularity} assumption. Establishing cross-regularity bounds for general discretization schemes $\hat{\bs P}$ can be quite non-trivial, see~\cite{Zhang25Anticipating} for techniques for this.
    Instead, here we bypass this issue by changing the algorithm at the last step: we use a different kernel $\hat{\bs P}{}'$, namely the underdamped Langevin Monte Carlo algorithm (recalled in \S\ref{ssec:ulmc}), for which it is easier to establish cross-regularity, as stated next.
\end{itemize}

\begin{lemma}[Cross-regularity for~\ref{eq:ULMC}]\label{lem:ulmc-cross-reg}
    For $q \geq 2$, if $h \lesssim \gamma^{-1} \wedge \gamma/\beta \wedge 1/(\beta^{1/2} q^{1/2})$, for $\hat{\bs P}{}' = \hat{\bs P}{}^{\msf{ULMC}}$,
    \begin{align}\label{eq:ulmc-cross-reg}
        \Ren_q(\delta_{x,p} \hat{\bs P}{}' \mmid \delta_{\bar x,
        \bar p} \bs P)
        &\lesssim \frac{\norm{x-\bar x}^2}{\gamma h^3} + \frac{\norm{p-\bar p}^2}{\gamma h} + \frac{\beta^2 h^3 q}{\gamma}\,\norm p^2 + \beta^2 dh^4 q + \frac{\beta^2 h^5 q}{\gamma}\,\norm{\nabla V(x)}^2\,.
    \end{align}
\end{lemma}

We make two remarks about Theorem~\ref{thm:kl_local_error}. First, for simplicity, the strongly convex bound has been optimized for the case where $T \geq \omega^{-1}$. If $T < \omega^{-1}$, the bound for the weakly convex case captures the correct behaviour. Second, compared to Theorem~\ref{thm:uld_regularity}, we have omitted the semi-convex case with $\alpha > -\beta$. This is because in applications of sampling, when $V$ is non-convex, it is typically assumed that $V$ is smooth in the sense that $\norm{\nabla^2 V}_{\rm op} \le \beta$, which corresponds to taking $\alpha = -\beta$.

\subsection{Technical overview}\label{ssec:kl_local_overview}

We now turn to proving Theorem~\ref{thm:kl_local_error}. The subsequent analysis heavily builds upon the preceding paper~\cite{scr2}, which is therefore recommended as a prerequisite for the present paper.

\paragraph{Auxiliary process.}
We let $\psi \deq (X, P)$ denote the joint system.
Recall that $\bs P$ denotes the kernel for~\ref{eq:ULD}, run for time $h$, and we denote
\begin{align*}
    \psi_n \sim \bs\nu_n \deq \bs\nu_0 \bs P^n\,.
\end{align*}
Similarly, we denote the algorithm iterates by
\begin{align*}
    \psi_n^\alg \sim \bs \mu_n^\alg \deq \bs \mu_0 (\bs P^\alg)^n \quad\text{for}~n < N\,, \qquad \psi_N^\alg \sim \bs\mu_N^\alg \deq \bs \mu_0 (\bs P^\alg)^{N-1} (\bs P^\alg)'\,.
\end{align*}
Since we shall soon introduce several additional stochastic processes, we change notation from $\hat{\bs P}$ to $\bs P^\alg$ in order to better label the various processes. %

We next introduce the auxiliary process, which transitions according to the kernel $\bs P$ but is shifted to hit the algorithm.
Since we also need to refer to the auxiliary process after shifting, but before applying $\bs P$, we also give this an explicit name.
\begin{align*}
    \psi_n^\aux \sim \bs \nu_n^\aux\,, \qquad \psi_n^\sh \sim \bs \nu_n^\sh\,.
\end{align*}
The auxiliary process is updated as follows: for $1 < n < N$,
\begin{align*}
    \psi_n^\aux \sim \bs P(\psi_{n-1}^\sh, \cdot)\,, \qquad
    \psi_n^\sh \deq (X_n^\sh, P_n^\sh) \deq (X_n^\aux, P_n^\aux + \upeta_{n}^\msx\,(X_n^\alg-X_n^\aux) + \upeta_{n}^\msp\,(P_n^\alg - P_n^\aux))\,.
\end{align*}
Here, we define the integrated shifts
\begin{align*}
    \upeta_n^\msx \deq \int_{nh}^{(n+1)h} \eta_t^\msx\,\D t\,, \qquad \upeta_n^\msp \deq \int_{nh}^{(n+1)h} \eta_t^\msp\,\D t\,.
\end{align*}
The shifts $\eta^\msx$, $\eta^\msp$ are defined similarly to the ones in \S\ref{ssec:harnack_overview}, with slight modifications for convenience of the argument; the precise definition is given as Definition~\ref{def:dt-shifts} in \S\ref{sec:diffuse_then_shift}.

\paragraph{Last step: cross-regularity.}
Given the definition of the auxiliary process, we apply the shifted chain rule (Theorem~\ref{thm:shifted_chain_rule}) with the following choices: under $\mb P$, we let $\msf X = \psi_{N-1}^\alg$, $\msf X' = \psi_{N-1}^\aux$, and $\msf Y = \psi_N^\alg$; thus,
$\msf X \sim \bs\mu_{N-1}^\alg$, $\msf X' \sim \bs\nu_{N-1}^\aux$, $\msf Y \sim \bs\mu_N^\alg$. Under $\mb Q$, we let $\msf X = \bar\psi_{N-1} \sim \bs\nu_{N-1}$ and $\msf Y \mid \bar\psi_{N-1} \sim \bs P(\bar\psi_{N-1},\cdot)$, so that $\msf Y \sim \bs\nu_N$.
This yields, by the cross-regularity assumption,
\begin{align}
    &\KL(\bs\mu^\alg_N \mmid \bs\nu_N)
    \le \KL(\bs\nu_{N-1}^\aux\mmid \bs\nu_{N-1}) + \E \KL((\bs P^\alg)'(\psi_{N-1}^\alg,\cdot) \mmid \bs P(\psi^\aux_{N-1},\cdot)) \nonumber\\[0.25em]
    &\qquad \le \KL(\bs\nu_{N-1}^\aux\mmid \bs\nu_{N-1}) + O\biggl( \frac{\E[\norm{X_{N-1}^\alg - X_{N-1}^\aux}^2]}{\gamma h^3} + \frac{\E[\norm{P_{N-1}^\alg - P_{N-1}^\aux}^2]}{\gamma h} + \E[b(X_{N-1}^\alg, P_{N-1}^\alg)^2] \biggr)\,.\label{eq:final_cross_reg}
\end{align}
The last term is at most $\bar b^2$.
Therefore, we are left to control the first term $\KL(\bs\nu_{N-1}^\aux\mmid \bs\nu_{N-1})$, as well as the final distances $\E[\norm{X_{N-1}^\alg - X_{N-1}^\aux}^2]$, $\E[\norm{P_{N-1}^\alg - P_{N-1}^\aux}^2]$.

\paragraph{Shifted chain rule.}
For the first term $\KL(\bs\nu_{N-1}^\aux\mmid \bs\nu_{N-1})$, we repeatedly apply the shifted chain rule (Theorem~\ref{thm:shifted_chain_rule}) with the following choices: under $\mb P$, we let $\msf X = \psi_n^\sh$, $\msf X' = \psi_n^\aux$, and $\msf Y = \psi_{n+1}^\aux$; thus, $\msf X \sim \bs\nu_n^\sh$, $\msf X' \sim \bs\nu_n^\aux$, $\msf Y \sim \bs \nu_{n+1}^\aux$. Under $\mb Q$, we let $\msf X = \bar\psi_n \sim \bs\nu_n$ and $\msf Y \mid \bar\psi_n \sim \bs P(\bar\psi_n,\cdot)$, so that $\msf Y \sim \bs \nu_{n+1}$.
It yields, for $n < N-1$,
\begin{align*}
    \KL(\bs\nu_{n+1}^\aux \mmid \bs\nu_{n+1})
    &\le \KL(\bs\nu_n^\aux \mmid \bs\nu_n) + \E \KL(\bs P(\psi_n^\sh,\cdot) \mmid \bs P(\psi_n^\aux, \cdot))\,.
\end{align*}
We now apply the KL divergence bound for the underdamped Langevin diffusion (Theorem~\ref{thm:uld_regularity}); since we are only considering a very short time interval (equal to the step size $h \lesssim \gamma^{-1} \wedge \absomega^{-1}$; when $\omega = 0$ this condition reduces to $h \lesssim \gamma^{-1}$), all three regimes coincide and show that
\begin{align*}
    \E \KL(\bs P(\psi_n^\sh,\cdot) \mmid \bs P(\psi_n^\aux, \cdot))
    &\lesssim \frac{1}{\gamma h}\,\E[\norm{\upeta^\msx_n\,(X_n^\alg - X_n^\aux) + \upeta^\msp_n\,(P_n^\alg - P_n^\aux)}^2]\,.
\end{align*}
The justification for this is that $\psi_n^\sh$, $\psi_n^\aux$ only differ in the momentum coordinate, so that in the bound above we only pick up the dependence on the initial momentum difference in Theorem~\ref{thm:uld_regularity}---which, crucially, is better by a factor of $h^2$ than the dependence on the initial position difference.

Iterating, this leads to:
\begin{align}\label{eq:total_kl_bd}
    \KL(\bs\nu_{N-1}^\aux \mmid \bs\nu_{N-1})
    &\le \sum_{n=0}^{N-2} \E \KL(\bs P(\psi_n^\sh,\cdot) \mmid \bs P(\psi_n^\aux,\cdot))
    \lesssim \frac{1}{\gamma h} \sum_{n=0}^{N-2} (\upeta_{n}^\msx d_n^\msx + \upeta_{n}^\msp d_n^\msp)^2\,,
\end{align}
where we define the distances $d_n^\msx = \norm{X_n^\alg - X_n^\aux}_{L^2}$, $d_n^\msp = \norm{P_n^\alg - P_n^\aux}_{L^2}$.

\paragraph{Distance recursion with local errors.}
The next step is to bound the auxiliary distances $d_n^\msx$, $d_n^\msp$.
However, as can be anticipated from the analysis in \S\ref{scn:continuous-time}, obtaining a contractive recursion for these distances requires switching to a time-varying coordinate system.
The second difference is that since the algorithm and the auxiliary processes evolve via different kernels---$\hat{\bs P}$ and $\bs P$ respectively---the distance recursion now involves error terms which must be tightly controlled in terms of the local errors $\mc E^{\rm w}$, $\mc E^{\rm s}$.

We use $\uppsi$ to denote the twisted time-varying coordinates.
Namely,
\begin{align*}
    \uppsi_n
    \deq (X_n, Z_n)
    \deq \bigl(X_n,\, X_n + \frac{2}{\gamma_{nh}}\,P_n\bigr)\,, \qquad \gamma_{nh} \deq \gamma + \eta_{nh}^\msp\,.
\end{align*}
We write $\mathtt d$ to denote distances in this coordinate system, namely,
\begin{align*}
    \ttd^\aux_n \deq \norm{\uppsi^\alg_n - \uppsi_n^\aux}_{L^2}\,, \qquad \ttd^\sh_n \deq \norm{\uppsi^\alg_n - \uppsi_n^\sh}_{L^2}\,.
\end{align*}
Although we eventually want to bound $d_n^\msx$, $d_n^\msp$, which are distances between $\alg$ and $\aux$, the following lemma gives a recursion in terms of the distance between $\alg$ and $\sh$, for reasons that are explained in \S\ref{ssec:alg_distance_recursion}.
However, this is still enough for our purposes, because it implies a bound between $\alg$ and $\aux$.

\begin{lemma}[Auxiliary distance recursion]\label{lem:dsh-recursion}
    For all $n < N-1$,
    \begin{align*}
        (\mathtt d^\sh_{n+1})^2
        &\leq L_n\, (\mathtt d^\sh_n)^2 + \frac{1}{\gamma_{nh}^2}\, O\Bigl(\frac{(\bar{\mc E}^{\rm w})^2}{(\omega_+ + \eta_{nh}^\msp)\,h} + \bigl(1+\frac{\beta^2 h}{\gamma_{nh}^2\,(\omega_+ + \eta_{nh}^\msp)}\bigr)\, (\bar{\mc E}^{\rm s})^2\Bigr)\,.
    \end{align*}
    Here, $L_n < 1$ is an effective contraction constant (derived in Lemma~\ref{lem:discrete_contraction}).
\end{lemma}

\paragraph{Computing the final bounds.}
With Lemma~\ref{lem:dsh-recursion} in hand, we can solve the distance recursion to bound $\ttd_n^\sh$, which leads to bounds on $d_n^\msx$ and $d_n^\msp$.
These bounds are then substituted into~\eqref{eq:total_kl_bd}. It leads to a complicated summation, which is then bounded via integral approximation and asymptotics.
The final distances $d_{N-1}^\msx$, $d_{N-1}^\msp$ must also be substituted into the cross-regularity term in~\eqref{eq:final_cross_reg}.
This concludes the overview of the proof of Theorem~\ref{thm:kl_local_error}.

\paragraph{Organization of the remainder of this section.}
The remainder of this section is devoted to the main ideas in the proof of Theorem~\ref{thm:kl_local_error}, deferring tedious calculations to \S\ref{app:local}.
\begin{itemize}
    \item In \S\ref{sec:diffuse_then_shift}, we derive the contraction constant $L_n$ in Lemma~\ref{lem:dsh-recursion}.
    \item In \S\ref{ssec:alg_distance_recursion}, we prove the auxiliary distance recursion (Lemma~\ref{lem:dsh-recursion}).
    \item In \S\ref{ssec:integral_estimates}, we explain how to express the final KL divergence bound in terms of certain integrals, which can be controlled via straightforward but tedious casework.
    \item In \S\ref{sec:ulmc_cross_reg}, we prove the cross-regularity result for~\ref{eq:ULMC} (Lemma~\ref{lem:ulmc-cross-reg}).
\end{itemize}

\subsection{Effective contraction}\label{sec:diffuse_then_shift}

First, we specify our shifts, which are slight modifications of the shifts for our continuous-time analysis in Definition~\ref{def:ct-shifts}. The main difference is that we introduce another free parameter $A > 0$ which tempers the explosion of the shifts as $t\nearrow T$ and consequently provides flexibility for our proofs. Here, $c_0$, $A$ are absolute constants such that both $c_0$, $A/c_0$ are sufficiently large.

\begin{defin}[Modified shift parameters]\label{def:dt-shifts}\mbox{}
The modified shifts are defined, for $t\in [0,T]$, to be
\begin{align*}
    \eta_t^\msx \deq \frac{\gamma_t \eta_t^\msp}{2}\,, \qquad 
    \eta_t^\msp
    \deq \frac{c_0\omega}{\exp(\omega\,(T-t+Ah)) - 1}\,.
\end{align*}
\end{defin}

Throughout, we assume $h \lesssim 1/\absomega$, so that $\eta_t^\msp \le \eta_T^\msp \lesssim c_0/(Ah)$.

Our eventual goal is to recursively control the distance between the algorithm and the auxiliary process.
In going from iteration $n$ to iteration $n+1$, there are three separate effects: (1) the \emph{application of $\bs P$}; (2) the \emph{shift}; and (3) the \emph{change of coordinate system} (recall that our coordinate system varies with time, due to the effective friction $\gamma_t$).
The key point is that these three effects are not separately contractive with the correct rate. %
Their effects are intertwined, and therefore \emph{they must be considered simultaneously}.

The specific problem under consideration is the following.
Let $t_- \in [0,T-h]$ and $t_+ \deq t_-+h$ denote the start and end times respectively, and let $\psi \deq (X, P)$ denote a copy of~\ref{eq:ULD} started at $\psi_{t_-} = (x, p)$.
Our goal in this section is to show that if $\psi^\aux \deq (X^\aux, P^\aux)$ is the auxiliary process shifted toward $\psi$, then the two processes contract in the time-varying coordinate system.
To see how this goal relates with the auxiliary process distance recursion, see the next subsection, \S\ref{ssec:alg_distance_recursion}.

More precisely, the auxiliary process can be schematically written
\begin{align}\label{eq:shift_then_diffuse}
    \psi^\aux_{t_+}
    &= \msf{ULD} \circ \msf{Shift}(t_-, \psi^\aux_{t_-})\,,
\end{align}
where $\msf{Shift}(t, \psi^\aux_t) \deq (X_t, P_t + (\int_{t-h}^t \eta^\msx)\,(X_t-X_t^\aux) + (\int_{t-h}^t \eta^\msp)\,(P_t-P_t^\aux))$ implements the shift, and $\msf{ULD}(\cdot)$ runs the~\ref{eq:ULD} dynamics for time $h$.
Let us now discuss the considerations which dictate our choice of proof strategy.
\begin{itemize}
    \item Since~\ref{eq:ULD} is a diffusion process which does not admit a closed-form expression for its transition kernel (unless $V$ is quadratic), it is not possible to study its effect ``in one shot''.
    Instead, it is necessary to use a continuous-time argument based on It\^o's formula.
    \item Such a continuous-time argument is not compatible with applying the shift at the beginning: indeed, even in a \emph{fixed} coordinate system with norm $\opnorm\cdot$, the It\^o argument would yield
    \begin{align*}
        \opnorm{\psi_{t_+} - \msf{ULD} \circ \msf{Shift}(t_-, \psi^\aux_{t_-})}
        &\le L_{\msf{ULD}}\,\opnorm{\psi_{t_-} - \msf{Shift}(t_-, \psi^\aux_{t_-})}\,,
    \end{align*}
    but it is unclear how to proceed since the shift operation is not, by itself, a contraction.
    \item Instead, we observe that the situation is better for the process $\psi^\sh$ directly after shifting,
    \begin{align}\label{eq:diffuse_then_shift}
        \psi_{t_+}^\sh
        &= \msf{Shift}(t_+,\cdot) \circ \msf{ULD}(\psi_{t_-}^\sh)\,,
    \end{align}
    and we instead establish a contraction for this ``diffuse then shift'' process.
\end{itemize}

The reason why~\eqref{eq:diffuse_then_shift} is easier than~\eqref{eq:shift_then_diffuse} is because we can view the shifting operation as a time-varying map applied to the output of~\ref{eq:ULD}, and we can incorporate this into the It\^o argument.
Introduce the following shorthand:
\begin{align*}
    \mc I\eta_t \deq \int_{t_-}^t \eta_s\,\D s
\end{align*}
is the integral of $\eta$, and for a map $F : \R^{2d} \to \R^d$, we let
\begin{align*}
    \delta F(X,P) \deq F(X,P) - F(X^\aux, P^\aux)
\end{align*}
denote the difference.
So, for example, $\delta X \deq X - X^\aux$, $\delta Z \deq Z - Z^\aux$, etc.

Consider the following linear map:
\begin{align}\label{eq:phi_map}
    \phi_{t}(X, P, X^\aux, P^\aux) \deq \biggl(\begin{bmatrix}
        1 & 0 \\
        -\mc I\eta_t^\msx & 1 - \mc I\eta_t^\msp
    \end{bmatrix} \otimes I\biggr) \begin{bmatrix}
        \delta X \\
        \delta P
    \end{bmatrix}\,.
\end{align}
This definition is such that
\begin{align*}
    (X_{t_+} - X_{t_+}^\sh, P_{t_+} - P_{t_+}^\sh) = \phi_{t_+}(X_{t_+}, P_{t_+}, X^\aux_{t_+}, P^\aux_{t_+})\,.
\end{align*}
This now suggests that we can control this quantity by replacing every instance of $t_+$ on the right-hand side with $t$ and differentiating w.r.t.\ $t$, where $(X_t^\aux, P_t^\aux)_{t\in [t_-,t_+]}$ evolves via~\ref{eq:ULD}.

Before doing so, however, we must change coordinates; we denote by $\upphi_t$ the corresponding map in the $(X, Z)$ coordinates. More precisely,
\begin{align}\label{eq:def_upphi}
    \upphi_{t}(X, Z, X^\aux, Z^\aux) = \biggl(\begin{bmatrix}
        1 & 0 \\
        \mc I\eta^\msp_t - 2\mc I\eta^\msx_t/\gamma_t & 1 - \mc I\eta^\msp_t
    \end{bmatrix}\otimes I\biggr) \begin{bmatrix}
        \delta X \\
        \delta Z
    \end{bmatrix}\,,
\end{align}
so that
\begin{align*}
    (X_{t_+} - X_{t_+}^\sh, Z_{t_+} - Z_{t_+}^\sh) = \upphi_{t_+}(X_{t_+}, Z_{t_+}, X^\aux_{t_+}, Z^\aux_{t_+})\,.
\end{align*}
By abuse of notation, let $\phi_t$, $\upphi_t$ also denote the corresponding matrices in the definitions of these two linear maps.

By differentiating in time, we obtain the following dynamics.

\begin{lemma}[Dynamics for the ``diffuse then shift'' process]\label{lem:discrete-decay}
    Assume that ${(X_t,P_t)}_{t\in [t_-,t_+]}$ and ${(X_t^\aux, P_t^\aux)}_{t\in [t_-,t_+]}$ evolve by~\ref{eq:ULD} and are synchronously coupled.
    Then,
\begin{align*}
    \D \norm{\upphi_t(\delta X_t, \delta Z_t)}^2 \le -2\,(1-\varepsilon_t)\,\lambda_{\min}(\mc M_t)\,\norm{\upphi_t(\delta X_t, \delta Z_t)}^2\,,
\end{align*}
where $\mc M_t$ is defined in Lemma~\ref{lem:decay},
\begin{align*}
    \varepsilon_t
    &\deq 
    \frac{\norm{\mc N_t}_{\rm op} + \norm{I-\upphi_t^{-1}}_{\rm op}\,(\norm{\mc M_t}_{\rm op} + \norm{\mc M_t^{\rm skew}}_{\rm op} + \norm{\mc N_t}_{\rm op})}{\lambda_{\min}(\mc M_t)}\,,
\end{align*}
and $\mc M_t^{\operatorname{skew}}$, $\mc N_t$ are defined as
\begin{align*}
    \mc M_t^{\operatorname{skew}} &\deq \begin{bmatrix}
        0 & (-\eta^\msx_t/\gamma_t + \dot{\gamma}_t/(2\gamma_t))\, I - \mc H_t/\gamma_t \\
        (\eta^\msx_t/\gamma_t - \dot{\gamma}_t/(2\gamma_t))\,I + \mc H_t/\gamma_t & 0
    \end{bmatrix}\,, \\
    \mc N_t &\deq \begin{bmatrix}
        0 & 0\\
        (-\mc I\eta^\msx_t  + {\gamma \,\mc I\eta^\msp_t} + \dot{\gamma}_t\,\mc I\eta^\msp_t/\gamma_t -2\dot \gamma_t \,\mc I\eta^\msx_t/\gamma_t^2)\,I - 2\,\mc I\eta^\msp_t\,\mc H_t/\gamma_t & (\mc I\eta^\msx_t-\gamma \,\mc I\eta^\msp_t - \dot{\gamma}_t\, \mc I\eta^\msp_t/\gamma_t)\,I
    \end{bmatrix}\,.
\end{align*}
\end{lemma}
\begin{proof} 
Due to the synchronous coupling, the difference process evolves via
\begin{align*}
    \D \delta X_t = \delta P_t \, \D t\,, \qquad
    \D \delta P_t = (- \delta \nabla V(X_t) - \gamma\, \delta P_t)\, \D t\,.
\end{align*}
In particular, the evolution is deterministic.
By a similar application of the fundamental theorem of calculus as in Lemma~\ref{lem:decay}, $\delta\nabla V(X_t) = \mc H_t\,\delta X_t$, where $\mc H_t$ is an integrated Hessian with eigenvalues in $[\alpha,\beta]$ by Assumption~\ref{as:regularity}.

Switching to the $(X, Z)$ coordinates,
\begin{align*}
    \D \delta X_t &= \frac{\gamma_t}{2}\,(\delta Z_t - \delta X_t)\,\D t\,, \\
    \D \delta Z_t &= \frac{\gamma_t}{2}\,(\delta Z_t - \delta X_t) \, \D t + \frac{2}{\gamma_t}\, \Bigl\{-\mc H_t\,\delta X_t-\frac{\gamma \gamma_t}{2}\,(\delta Z_t - \delta X_t) \Bigr\} \, \D t - \frac{2\dot{\gamma}_t}{\gamma_t^2}\,\frac{\gamma_t}{2}\,(\delta Z_t - \delta X_t)\, \D t \\
    &= \Bigl\{-\frac{\gamma_t}{2} - \frac{2\mc H_t}{\gamma_t} + \gamma + \frac{\dot\gamma_t}{\gamma_t}\Bigr\}\,\delta X_t\,\D t + \Bigl\{\frac{\gamma_t}{2}-\gamma - \frac{\dot\gamma_t}{\gamma_t}\Bigr\}\,\delta Z_t\,\D t\,.
\end{align*}
Therefore, since $\upphi_t$ is a linear map,
    \begin{align*}
        \D \upphi_t(\delta X_t, \delta Z_t) = (\partial_t \upphi_t)(\delta X_t, \delta Z_t)\,\D t + \upphi_t(\D\delta X_t,\D\delta Z_t)\,\D t\,.
    \end{align*}
    We can compute
    \begin{align*}
        (\partial_t \upphi_t)(\delta X, \delta Z) = \biggl(\begin{bmatrix}
            0 & 0 \\
            \eta_t^\msp -2\eta_t^\msx/\gamma_t + 2\dot{\gamma}_t\, \mc I\eta^\msx_t/\gamma_t^2 & -\eta_t^\msp
        \end{bmatrix}\otimes I\biggr) \begin{bmatrix}
            \delta X \\
            \delta Z
        \end{bmatrix}\,.
    \end{align*}
    Some tedious algebra then yields
    \begin{align}\label{eq:dynamics_diffuse_shift}
        \D \upphi_t(\delta X_t, \delta Z_t) = \begin{bmatrix}
            (-\gamma_t /2)\,I & (\gamma_t/2)\,I \\
            a_1(t) \,I + 2\,(\mc I\eta^\msp_t-1)\, \mc H_t/\gamma_t & a_2(t)\,I
        \end{bmatrix} \begin{bmatrix}
            \delta X_t \\
            \delta Z_t
        \end{bmatrix} \D t\,,
    \end{align}
    where
    \begin{align*}
        a_1(t) &= \eta_t^\msp - \frac{2\eta_t^\msx}{\gamma_t}+ \frac{2\dot{\gamma}_t \,\mc I\eta^\msx_t}{\gamma_t^2} + \mc I \eta^\msx_t - \frac{\gamma_t}{2} + \gamma + \frac{\dot{\gamma}_t}{\gamma_t} - \gamma\,\mc I \eta^\msp_t - \frac{\dot\gamma_t\,\mc I \eta^\msp_t}{\gamma_t} \,, \\
        a_2(t) &= -\eta^\msp_t -\mc I \eta^\msx_t + \frac{\gamma_t}{2} - \gamma -\frac{\dot{\gamma}_t}{\gamma_t} + \gamma\,\mc I \eta^\msp_t + \frac{\dot{\gamma}_t\,\mc I\eta^\msp_t}{\gamma_t}\,.
    \end{align*}

    If we compare this with the matrix $\mc M_t$ from Lemma~\ref{lem:decay}, we can write this as
    \begin{align*}
        \D\norm{\upphi_t(\delta X_t, \delta Z_t)}^2
        &= -2\,\langle \upphi_t(\delta X_t, \delta Z_t), (\mc M_t + \mc M_t^{\rm skew} + \mc N_t)\,\upphi_t^{-1}\,\upphi_t(\delta X_t, \delta Z_t)\rangle\,,
    \end{align*}
    where $\mc M_t^{\rm skew}$, $\mc N_t$ are defined in the lemma statement.
    Note that if we set $\mc I\eta^\msx_t = \mc I\eta^\msp_t = 0$, which is equivalent to the $t \searrow t_-$ limit, then $\mc N_t = 0$ and this recovers Lemma~\ref{lem:decay}.

    From here, we have
    \begin{align*}
        \D\norm{\upphi_t(\delta X_t, \delta Z_t)}^2
        &= -2\,\langle \upphi_t(\delta X_t, \delta Z_t), (\mc M_t + \mc N_t)\,\upphi_t(\delta X_t, \delta Z_t)\rangle \\
        &\qquad{} +2\,\langle \upphi_t(\delta X_t, \delta Z_t), (\mc M_t + \mc M_t^{\rm skew} + \mc N_t)\,(I-\upphi_t^{-1})\,\upphi_t(\delta X_t, \delta Z_t)\rangle \\
        &\le -2\lambda_{\min}(\mc M_t)\,\norm{\upphi_t(\delta X_t,\delta Z_t)}^2 + 2\,\norm{\mc N_t}_{\rm op}\,\norm{\upphi_t(\delta X_t, \delta Z_t)}^2 \\
        &\qquad{} + 2\,\norm{I-\upphi_t^{-1}}_{\rm op}\,(\norm{\mc M_t}_{\rm op} + \norm{\mc M_t^{\rm skew}}_{\rm op} + \norm{\mc N_t}_{\rm op})\,\norm{\upphi_t(\delta X_t, \delta Z_t)}^2 \\
        &= -2\,(1-\varepsilon_t)\,\lambda_{\min}(\mc M_t)\,\norm{\upphi_t(\delta X_t, \delta Z_t)}^2\,,
    \end{align*}
    which completes the proof.
\end{proof}

Lemma~\ref{lem:discrete-decay} shows that the ``diffuse then shift'' system contracts at a rate matching the one obtained in Lemma~\ref{lem:contraction-cont-time}, up to an additional error term $\varepsilon_t$ which we control in \S\ref{app:discrete_contraction}.
It yields the following lemma.

\begin{lemma}[Contraction for the ``diffuse then shift'' process]\label{lem:discrete_contraction}
    Assume that $h \lesssim \gamma^{-1} \wedge \gamma/\beta$ and that $c_0/A \lesssim 1$ for sufficiently small implied constants.
    Then, for $\psi_{t_-}^\sh$, $\psi_{t_+}^\sh$ given in~\eqref{eq:diffuse_then_shift}, $\uppsi_{t_-}^\sh$, $\uppsi_{t_+}^\sh$ the corresponding points in the twisted coordinates, and an absolute constant $c > 0$,
    \begin{align*}
        \norm{\uppsi_{t_+}^\sh - \uppsi_{t_+}}
        \le \exp\Bigl(-c\int_{t_-}^{t_+} (\omega_+ + \eta_t^\msp)\,\D t\Bigr)\,\norm{\uppsi_{t_-}^\sh - \uppsi_{t_-}}\,.
    \end{align*}
\end{lemma}

\subsection{Distance recursion for auxiliary process}\label{ssec:alg_distance_recursion}

In this section, we derive the auxiliary distance recursion (Lemma~\ref{lem:dsh-recursion}).
As a first step, we establish the following recursion, which is the analogue of the usual $W_2$ local error recursion (see, e.g.,~\cite[Lemma B.5]{scr2}).
Two key differences here are that the coordinate system is changing with time, and the recursion is stated for the $\sh$ process.

\begin{lemma}[Intermediate auxiliary distance recursion]\label{lem:intermediate_recursion}
    For all ${n < N-1}$,
    \begin{align*}
        (\mathtt d^\sh_{n+1})^2
        &\le L_n^2 \,(\mathtt d^\sh_n)^2 + O\bigl((\mathtt E^{\rm w}_n + \zeta_n \mathtt E^{\rm s}_n)\,\mathtt d^\sh_n + (\mathtt E^{\rm s}_n)^2\bigr)\,,
    \end{align*}
    where:
    \begin{itemize}
        \item $L_n$ is the effective contraction constant at iteration $n$ for the ``diffuse then shift'' process given in Lemma~\ref{lem:discrete_contraction}, i.e.,
            \begin{align*}
                L_n = \exp\Bigl(-c\int_{nh}^{(n+1)h} (\omega_+ + \eta_t^\msp)\,\D t\Bigr)\,.
            \end{align*}
        \item $\mathtt E$ denotes the local errors in the twisted coordinates. Namely, introducing a third process $\tilde\psi$ which starts at $\psi_n^\alg$ and evolves by the diffusion $\bs P$, we define
            \begin{align*}
                \mathtt E_n^{\rm w} \deq \sqrt{\E[\norm{\E[\uppsi^\alg_{n+1} - \tilde \uppsi_{n+1} \mid \ms F_n]}^2]}\,, \qquad
                \mathtt E_n^{\rm s} \deq \sqrt{\E[\norm{\uppsi^\alg_{n+1}- \tilde \uppsi_{n+1}}^2]}\,,
            \end{align*} 
            where the inner expectations are conditional on the filtration $\ms F_n$ at iteration $n$.
        \item $\zeta_n$ is the ``coupling constant'', defined in~\eqref{eq:zeta} below.
    \end{itemize}
\end{lemma}
\begin{proof}
    The point of introducing the process $\tilde\psi$ is to leverage the contraction of the ``diffuse then shift'' system in Lemma~\ref{lem:discrete_contraction}. However, one issue is that $\psi^\sh$ evolves over $[nh,(n+1)h]$ via $\bs P$ and then is shifted toward $\psi^\alg_{n+1}$, \emph{not} toward $\tilde \psi_{n+1}$. %
    Therefore, we introduce one more random variable $\psi^{\sh\to\sim}_{n+1}$, which is obtained by shifting $\psi^\aux_{n+1}$ towards $\tilde \psi_{n+1}$.
    In symbols,
    \begin{align*}
        \tilde\psi_{n+1} - \psi^{\sh\to\sim}_{n+1}
        &\deq \phi_{(n+1)h}(\tilde \psi_{n+1}- \psi^\aux_{n+1})\,, \\
        \psi_{n+1}^\alg - \psi^{\sh}_{n+1}
        &\deq \phi_{(n+1)h}(\psi^\alg_{n+1}- \psi^\aux_{n+1})\,,
    \end{align*}
    where $\phi_{(n+1)h}$ is the matrix from~\eqref{eq:phi_map}.
    Due to linearity of $\phi_{(n+1)h}$ and coordinate changes,
    \begin{align*}
        \uppsi^{\sh\to\sim}_{n+1} - \uppsi^\sh_{n+1}
        &= (I-\upphi_{(n+1)h})\,(\tilde \uppsi_{n+1} - \uppsi^\alg_{n+1})\,.
    \end{align*}

    Expanding the square,
    \begin{align*}
        (\mathtt d^\sh_{n+1})^2
        &= \E[\norm{\uppsi^\alg_{n+1} - \uppsi^\sh_{n+1}}^2] 
        = \E[\norm{\tilde\uppsi_{n+1} - \uppsi^{\sh\to\sim}_{n+1} + \uppsi^\alg_{n+1} - \tilde\uppsi_{n+1} + \uppsi^{\sh\to\sim}_{n+1} - \uppsi_{n+1}^\sh}^2] \\
        &= \E[\norm{\tilde\uppsi_{n+1} - \uppsi^{\sh\to\sim}_{n+1} +\upphi_{(n+1)h}\,(\uppsi^\alg_{n+1} - \tilde\uppsi_{n+1})}^2] \\
        &=\E[\norm{\tilde{\uppsi}_{n+1} - \uppsi^{\sh\to\sim}_{n+1}}^2]
        + 2\,\E\langle \tilde{\uppsi}_{n+1} - \uppsi^{\sh\to\sim}_{n+1}, \upphi_{(n+1)h}\,(\uppsi^\alg_{n+1} - \tilde{\uppsi}_{n+1})\rangle \\
        &\qquad{} +\E[\norm{\upphi_{(n+1)h}\,(\uppsi^\alg_{n+1} - \tilde{\uppsi}_{n+1})}^2]\,.
    \end{align*}
    For the first term, we apply the contraction for the ``diffuse then shift'' process (Lemma~\ref{lem:discrete_contraction}, so that the first term is at most $L_n^2 \,(\mathtt d_n^\sh)^2$.
    By Lemma~\ref{lem:phi_bds}, the last term is controlled by the strong error in the twisted coordinates.

For the middle term, we add and subtract $\uppsi^\alg_n - \uppsi^\sh_n$ to obtain
\begin{align*}
    &2\,\E\langle \tilde{\uppsi}_{n+1} - \uppsi^{\sh\to\sim}_{n+1} - (\uppsi^\alg_n - \uppsi^\sh_n), \upphi_{(n+1)h}\,(\uppsi^\alg_{n+1} - \tilde{\uppsi}_{n+1})\rangle + 2\,\E\langle \uppsi^\alg_n - \uppsi^\sh_n, \upphi_{(n+1)h}\,(\uppsi^\alg_{n+1} - \tilde{\uppsi}_{n+1})\rangle \\
    &\qquad \le 2\,\norm{\upphi_{(n+1)h}}_{\rm op}\,\zeta_n\,\mathtt E_n^{\rm s}\,\mathtt d_n^\sh + 2\,\E\langle \uppsi^\alg_n - \uppsi^\sh_n, \upphi_{(n+1)h}\,\E[\uppsi^\alg_{n+1} - \tilde{\uppsi}_{n+1} \mid \ms F_n]\rangle \\
    &\qquad\lesssim (\mathtt E_n^{\rm w} + \zeta_n \mathtt E_n^{\rm s})\,\mathtt d_n^\sh\,,
\end{align*}
where we defined
\begin{align}\label{eq:zeta}
    \zeta_n \deq \frac{\norm{\tilde{\uppsi}_{n+1} - \uppsi^{\sh\to\sim}_{n+1} - (\uppsi^\alg_n - \uppsi^\sh_n)}_{L^2}}{\norm{\uppsi_n^\alg - \uppsi_n^\sh}_{L^2}}
\end{align}
and we applied Lemma~\ref{lem:phi_bds} to control $\norm{\upphi_{(n+1)h}}_{\rm op}$.
\end{proof}

Typically, we expect the constant $\zeta_n$ to be rather negligible.
The following lemma, which bounds $\zeta_n$, is deferred to \S\ref{app:zeta}.

\begin{lemma}[Bound for $\zeta_n$]\label{lem:zeta}
For our choice of shifts, $h \lesssim \gamma^{-1} \wedge \gamma/\beta$ and $c_0/A \lesssim 1$ for sufficiently small implied constants, and $\gamma \le \sqrt{32\beta}$, we have
\begin{align*}
    \zeta_n \lesssim \frac{\beta h}{\gamma_{nh}} + h \eta_{nh}^\msp\,.
\end{align*}
\end{lemma}

We are now ready to prove Lemma~\ref{lem:dsh-recursion}.

\begin{proof}[Proof of Lemma~\ref{lem:dsh-recursion}]
    Since $\eta_{nh} \lesssim 1/h$, %
    \begin{align*}
        (\mathtt E^{\rm s}_n)^2
        &\lesssim \E[\norm{X_{n+1}^\alg - \tilde X_{n+1}}^2] + \frac{1}{\gamma_{(n+1)h}^2}\,\E[\norm{P_{n+1}^\alg - \tilde P_{n+1}}^2]
        \lesssim \bigl(h^2 + \frac{1}{\gamma_{(n+1)h}^2}\bigr)\,(\bar{\mc E}^{\rm s})^2
        \lesssim \frac{1}{\gamma_{nh}^2}\,(\bar{\mc E}^{\rm s})^2\,,
    \end{align*}
    and similarly, $\mathtt E^{\rm w}_n \lesssim \bar{\mc E}^{\rm w}/\gamma_{nh}$.
    By Lemma~\ref{lem:intermediate_recursion}, it leads to
    \begin{align*}
        (\mathtt d^\sh_{n+1})^2 \leq L_n^2\, (\mathtt d^\sh_n)^2 + O\Bigl(\frac{\bar{\mc E}^{\rm w} + \zeta_n \bar{\mc E}^{\rm s}}{\gamma_{nh}}\, \mathtt d^\sh_n + \frac{(\bar{\mc E}^{\rm s})^2}{\gamma_{nh}^2}\Bigr)\,.
    \end{align*}

    After applying Young's inequality, we have
    \begin{align*}
        (\mathtt d^\sh_{n+1})^2 \leq L_n\, (\mathtt d^\sh_n)^2 + O\Bigl(\frac{(\bar{\mc E}^{\rm w} + \zeta_n \bar{\mc E}^{\rm s})^2}{\gamma_{nh}^2\,(L_n - L_n^2)} + \frac{(\bar{\mc E}^{\rm s})^2}{\gamma_{nh}^2}\Bigr)\,.
    \end{align*}
    From Lemma~\ref{lem:zeta}, $\zeta_n^2 \lesssim \beta^2 h^2/\gamma_{nh}^2 + h^2\,(\eta_{nh}^\msp)^2$.
    Since $(\omega_+ + \eta_{nh}^\msp)\,h \lesssim 1$, then $L_n \gtrsim 1$, and $L_n - L_n^2\gtrsim h\eta_{nh}^\msp$.
    Thus, $h^2\,(\eta_{nh}^\msp)^2/(L_n - L_n^2) \lesssim h\eta_{nh}^\msp \lesssim 1$, so we can neglect this term.
    Therefore,
    \begin{align*}
        (\mathtt d^\sh_{n+1})^2
        &\leq L_n\, (\mathtt d^\sh_n)^2 + \frac{1}{\gamma_{nh}^2}\, O\Bigl(\frac{(\bar{\mc E}^{\rm w})^2}{(\omega_+ + \eta_{nh}^\msp)\,h} + \bigl(1+\frac{\beta^2 h}{\gamma_{nh}^2\,(\omega_+ + \eta_{nh}^\msp)}\bigr)\, (\bar{\mc E}^{\rm s})^2\Bigr)\,. \qedhere
    \end{align*}
\end{proof}

\subsection{Computing the bounds via integral estimates}\label{ssec:integral_estimates}

Next, we convert back from the bounds on $\ttd^\sh$ to $d^\msx$ and $d^\msp$.

\begin{lemma}[Converting back to the original coordinates]\label{lem:distance-control-shift-inversion}
    If $h \lesssim \gamma^{-1} \wedge \gamma/\beta$ and $c_0/A \lesssim 1$ for sufficiently small implied constants, then
    \begin{align*}
        \ttd_n^\sh \asymp d_n^\msx + \frac{1}{\gamma_{nh}}\,d_n^\msp\,.
    \end{align*}
\end{lemma}
\begin{proof}
    By Lemma~\ref{lem:phi_bds},
    \begin{align*}
        \ttd_n^\sh
        &= \norm{\uppsi_n^\alg - \uppsi_n^\sh}_{L^2}
        = \norm{\upphi_{nh}(\uppsi_n^\alg - \uppsi_n^\aux)}_{L^2}
        \asymp \norm{\uppsi_n^\alg - \uppsi_n^\aux}_{L^2}
        \asymp d_n^\msx + \frac{1}{\gamma_{nh}}\,d_n^\msp\,. \qedhere
    \end{align*}
\end{proof}

We now relate the final KL bounds to certain integrals.

\begin{lemma}[KL bound as an integral]\label{lem:kl_integral}
    Assume that $h \lesssim \gamma^{-1} \wedge \gamma/\beta$ and $c_0/A \lesssim 1$ for sufficiently small implied constants.
    Define
    \begin{align*}
        \lambda_t \deq \omega_+ + \eta_t\,, \qquad \Lambda_t \deq \int_0^t \lambda_s\,\D s\,, \qquad \ms E_t \deq \frac{1}{\gamma_t^2}\,\Bigl[\frac{(\bar{\mc E}^{\rm w})^2}{\lambda_t h} + \bigl(1+\frac{\beta^2 h}{\gamma_t^2\lambda_t}\bigr)\, (\bar{\mc E}^{\rm s})^2\Bigr]\,,
    \end{align*}
    where $\eta_t \deq \eta_t^\msp$.
    Then,
    \begin{align*}
        \KL(\bs\mu_N^\alg \mmid \bs\nu_N)
        &\lesssim \Bigl[\frac{\exp(-c\Lambda_{Nh})}{\gamma h^3} + \frac{1}{\gamma} \int_0^{Nh} \exp(-c\Lambda_t)\,\gamma_t^2 \eta_t^2\,\D t\Bigr]\,(d_0^\msx + \gamma_0^{-1} d_0^\msp)^2 + \bar b^2 \\
        &\qquad{} + \frac{1}{\gamma h^4} \int_0^{Nh} \exp(-c\,(\Lambda_{Nh} - \Lambda_t)) \,\ms E_t\,\D t \\
        &\qquad{} + \frac{1}{\gamma h} \int_0^{Nh} \gamma_t^2 \eta_t^2 \int_0^t \exp(-c\,(\Lambda_t - \Lambda_s))\,\ms E_s\,\D s \,\D t\,.
    \end{align*}
\end{lemma}
\begin{proof}
    First, by unrolling the recursion in Lemma~\ref{lem:dsh-recursion},
\begin{align*}
    (\mathtt d_n^\sh)^2 &\leq \exp(-c \Lambda_{nh})\, (\mathtt d_0^\sh)^2 + \frac{1}{h} \sum_{k=1}^n h\exp(-c\,(\Lambda_{nh} - \Lambda_{kh})) \, \frac{1}{\gamma_{kh}^2}\,O\Bigl(\frac{(\bar{\mc E}^{\rm w})^2}{\lambda_{kh} h} + \bigl(1+\frac{\beta^2 h}{\gamma_{kh}^2\lambda_{kh}}\bigr)\, (\bar{\mc E}^{\rm s})^2\Bigr) \\ %
    &\leq \exp(-c \Lambda_{nh})\, (\mathtt d_0^\sh)^2 + \frac{1}{h} \int_0^{(n-1)h} \exp(-c\,(\Lambda_{nh} - \Lambda_{t+h}))\,\frac{1}{\gamma_t^2}\,O\Bigl(\frac{(\bar{\mc E}^{\rm w})^2}{\lambda_t h} + \bigl(1+\frac{\beta^2 h}{\gamma_t^2\lambda_t}\bigr)\, (\bar{\mc E}^{\rm s})^2\Bigr)\,\D t \\
    &\leq \exp(-c \Lambda_{nh})\, (\mathtt d_0^\sh)^2 + \frac{1}{h} \int_0^{(n-1)h} \exp(-c\,(\Lambda_{(n-1)h} - \Lambda_t))\,\frac{1}{\gamma_t^2}\,O\Bigl(\frac{(\bar{\mc E}^{\rm w})^2}{\lambda_t h} + \bigl(1+\frac{\beta^2 h}{\gamma_t^2\lambda_t}\bigr)\, (\bar{\mc E}^{\rm s})^2\Bigr)\,\D t\,.
\end{align*}
We have replaced the summation in the first expression with an integral, noting that $\gamma_t, \lambda_t, \Lambda_t$ are all increasing in $t$. Finally, in the last line, we used that for all $n, t$, $\Lambda_{nh} - \Lambda_{t+h} \ge \Lambda_{(n-1)h} - \Lambda_t - h\sup_{t \in [0, T]} \lambda_t \geq \Lambda_{(n-1)h} - \Lambda_t - O(1)$.

From the final expression~\eqref{eq:total_kl_bd} for the KL divergence and Lemma~\ref{lem:distance-control-shift-inversion},
\begin{align*}
    \gamma \KL(\bs\nu_{N-1}^\aux \mmid \bs\nu_{N-1})
    &\lesssim \frac{1}{h} \sum_{n=0}^{N-2} (\upeta_{n}^\msx d_n^\msx + \upeta_{n}^\msp d_n^\msp)^2 \lesssim \sum_{n=0}^{N-2} h \gamma_{nh}^2 \eta_{nh}^2\, (\mathtt d_n^\sh)^2 \\
    &\lesssim \int_0^{Nh} \exp(-c\Lambda_t) \,\gamma_t^2 \eta_t^2\, (\mathtt d_0^\sh)^2 \, \D t \\
    &\qquad+ \frac{1}{h}\int_0^{Nh} \gamma_t^2 \eta_t^2 \int_0^t \frac{\exp(-c\,(\Lambda_t - \Lambda_s))}{\gamma_s^2} \, \Bigl(\frac{(\bar{\mc E}^{\rm w})^2}{\lambda_s h} + \bigl(1+\frac{\beta^2 h}{\gamma_s^2\lambda_s}\bigr)\, (\bar{\mc E}^{\rm s})^2\Bigr)\,\D s\,\D t\,,
\end{align*}
where we again bound summations via integrals, paying absolute constant factors, via routine arguments.
Finally, for the cross-regularity term, since $\gamma_{(N-1)h} \lesssim 1/h$,
\begin{align*}
    \frac{(d^\msx_{N-1})^2}{h^3} + \frac{(d^\msp_{N-1})^2}{h}
    &\lesssim \frac{(\ttd^\sh_{N-1})^2}{h^3} + \frac{\gamma_{(N-1)h}^2\,(\ttd^\sh_{N-1})^2}{h}
    \lesssim \frac{(\ttd^\sh_{N-1})^2}{h^3}\,. \qedhere
\end{align*}
\end{proof}

To finish the proof, it remains to compute the integrals based on our choices for $\eta^\msx$, $\eta^\msp$.
We defer the tedious calculations to \S\ref{app:integrals}.

\subsection{Cross-regularity for ULMC}\label{sec:ulmc_cross_reg}

In this subsection, we prove the cross-regularity bound for~\ref{eq:ULMC} (Lemma~\ref{lem:ulmc-cross-reg}).

\begin{proof}[Proof of Lemma~\ref{lem:ulmc-cross-reg}]
    Rather than applying Girsanov's theorem with shifting, we opt to use the weak triangle inequality for R\'enyi divergences (Proposition~\ref{prop:weak-triangle}), which gives for $q \ge 2$,
    \begin{align*}
        \Ren_q(\delta_{x,p} \hat{\bs P}{}' \mmid \delta_{\bar x,
    \bar p} \bs P) &\leq \frac{q-1/2}{q-1}\,\Ren_{2q}(\delta_{x,p} \hat{\bs P}{}' \mmid \delta_{x,p} \bs P)
    + \Ren_{2q-1}(\delta_{x,p} \bs P \mmid \delta_{\bar x, \bar p}\bs P)\,.
    \end{align*}
    The first term is bounded as in~\cite[\S 6]{chewibook} for $h$ satisfying the conditions of the lemma:
    \begin{align*}
         \Ren_{2q}(\delta_{x,p} \hat{\bs P}{}' \mmid \delta_{x,p} \bs P)
         \lesssim \frac{\beta^2 h^3 q}{\gamma} \norm{p}^2 + \beta^2 dh^4 q + \frac{\beta^2 h^5q}{\gamma} \norm{\nabla V(x)}^2\,.   
    \end{align*}
    On the other hand, the second R\'enyi divergence is bounded via Theorem~\ref{thm:uld_regularity} by
    \begin{align*}
        \Ren_{2q-1}(\delta_{x,p} \bs P \mmid \delta_{\bar x, \bar p}\bs P)
        &\lesssim \frac{q}{\gamma}\,\Bigl(\frac{\norm{x-\bar x}^2}{h^3} + \frac{\norm{p-\bar p}^2}{h} \Bigr)\,. \qedhere
    \end{align*}
\end{proof}

	\section{Application to sampling}\label{sec:app}

In this section, we apply our KL local framework (Theorem~\ref{thm:kl_local_error}) in order to prove new algorithmic guarantees for sampling from a target probability measure $\pi\propto \exp(-V)$ over $\R^d$.
Throughout, we let $\bs\pi \deq \pi \otimes \cN(0,I)$ denote the augmented target distribution. Although the sampling applications we state eventually follow from Theorem~\ref{thm:kl_local_error}, they require additional bookkeeping in order to recursively bound the error terms; these details are deferred to \S\ref{app:sampling}.

\subsection{Underdamped Langevin Monte Carlo}\label{ssec:ulmc}

As a warm-up, we apply our techniques to the simplest discretization of~\ref{eq:ULD}, namely, the stochastic exponential Euler discretization, or simply \emph{underdamped Langevin Monte Carlo} (ULMC):
\begin{align}\label{eq:ULMC}\tag{$\msf{ULMC}$}
    \D \hat X_t = \hat P_t\,\D t\,, \qquad \D \hat P_t = \{-\nabla V(\hat X_{t_-}) - \gamma \hat P_t\}\,\D t +\sqrt{2\gamma}\,\D B_t\,, \qquad t_- \deq \lfloor t/h\rfloor \,h\,.
\end{align}
The~\ref{eq:ULMC} algorithm is simple enough that it does not require the full power of Theorem~\ref{thm:kl_local_error} (indeed, a na\"{\i}ve application of Girsanov's theorem suffices to bound the discretization error), but it serves as a first illustration of our framework.
Moreover, even in this case, our framework yields a novel result, namely, a convergence guarantee when $\pi$ is only assumed to be weakly log-concave.

We begin by computing the local errors. In this case, we do not see any improvement from the weak error, so we only need to check the strong error.

\begin{lemma}[Local errors for~\ref{eq:ULMC}]\label{lem:local_error_ulmc}
    Assume that $-\beta I \preceq \nabla^2 V \preceq \beta I$, $\gamma \lesssim \sqrt \beta$, and $h \lesssim 1/\beta^{1/2}$. Then, the strong error for~\ref{eq:ULMC} is bounded as follows:
    \begin{align*}
        \mc E^{\rm s}(x,p)
        &\lesssim \beta h^2\,\norm p + \beta^{5/4} d^{1/2} h^{5/2} + \beta h^3 \,\norm{\nabla V(x)}\,.
    \end{align*}
\end{lemma}
\begin{proof}
    This is a standard computation, see, e.g.,~\cite[\S 5]{chewibook}.
\end{proof}

Based on Lemma~\ref{lem:local_error_ulmc}, we expect the strong error to scale as $\beta^{1/2} d^{1/2} h^2$.
A careful application of Theorem~\ref{thm:kl_local_error} yields the following results. The proof is deferred to \S\ref{app:ulmc}.

\begin{theorem}[\ref{eq:ULMC}, convex case]\label{thm:ulmc_cvx}
    Let $\hat{\bs P}$ denote the kernel for~\ref{eq:ULMC}, and let $\bs P$ denote the kernel for~\ref{eq:ULD} run for time $h$.
    Assume that $0 \preceq \alpha I \preceq \nabla^2 V \preceq \beta I$ and set $\gamma = \sqrt{32\beta}$.
    Write $W^2 \deq \mc W_2^2(\bs\mu,\bs\pi)$, where $\mc W_2$ denotes the $2$-Wasserstein distance in the twisted norm $(x,p) \mapsto \sqrt{\norm x^2 + \gamma_0^{-2}\,\norm p^2}$.
    \begin{enumerate}
        \item \underline{$\alpha > 0$ case:} If $\varepsilon \in (0, \frac{d^{1/2}}{\kappa^{1/2}}]$ and $h = \Theta( \frac{\varepsilon}{\beta^{1/2} \kappa^{1/2} d^{1/2}})$, then $\KL(\bs\mu \hat{\bs P}{}^N \mmid \bs\pi) \le \varepsilon^2$ for all
        \begin{align*}
            N \gtrsim \frac{\kappa^{3/2} d^{1/2}}{\varepsilon} \log \frac{\alpha W^2}{\varepsilon^2}\,.
        \end{align*}

        \item \underline{$\alpha = 0$ case:} If $\varepsilon\in  (0,\beta^{1/2} W]$ and $h = \Theta(\frac{\varepsilon^2}{\beta d^{1/2} W + \beta^{3/2} W^2})$, then $\KL(\bs\mu\hat{\bs P}{}^N \mmid \bs\pi) \le \varepsilon^2$ for
        \begin{align*}
            N \asymp \frac{\beta^{3/2} d^{1/2} W^3 + \beta^2 W^4}{\varepsilon^4}\,.
        \end{align*}
    \end{enumerate}
\end{theorem}

We now compare this result with prior work.
In the strongly convex case ($\alpha > 0$), an iteration complexity of $\kappa^2 d^{1/2}/\varepsilon$ was established in the $W_2$ metric in~\cite{cheng2018underdamped}, and was subsequently improved to $\kappa^{3/2} d^{1/2}/\varepsilon$ in~\cite{dalalyanrioudurand2020underdamped}. Regarding convergence in metrics stronger than $W_2$, the $\kappa^{3/2} d^{1/2}/\varepsilon$ rate was first established under higher-order smoothness in~\cite{ma2021there} and then without this assumption in in~\cite{Zhang+23ULMC}, thus matching Theorem~\ref{thm:ulmc_cvx}; and for R\'enyi divergences of order $q > 1$, the rate $\kappa^{3/2} d^{5/2}/\varepsilon$ is implicit in~\cite{GaneshT20}; the rate $\kappa d^2/\varepsilon$ was shown in~\cite{Zhang+23ULMC} (the improved dependence on $\kappa$ is due to the use of the space-time Poincar\'e inequality, see \S\ref{ssec:space_time} below), and the rate $\kappa^{3/2} d^{1/2}/\varepsilon$ was shown in~\cite{AltChe24Warm}, which used this result as a warm start for the Metropolis-adjusted Langevin algorithm.
Thus, in this case, Theorem~\ref{thm:ulmc_cvx} is consistent with known results in the literature and is not new.

On the other hand, the result for the weakly convex case ($\alpha = 0$) is new.
Since there are few %
sampling guarantees for the weakly convex setting in the literature, it is worth comparing with results for other algorithms.
The rate of Theorem~\ref{thm:kl_local_error} reads $(d/\varepsilon^2)^{1/2}\,(\beta W^2/\varepsilon^2)^{3/2} + (\beta W^2/\varepsilon^2)^2$, which is worse than the $(d/\varepsilon^2)^{1/3}\,(\beta W^2/\varepsilon^2)^{4/3}$ rate of~\cite{scr1} for the randomized midpoint discretization of the overdamped Langevin diffusion, and incomparable with the $(d/\varepsilon^2)\,(\beta W^2/\varepsilon^2)$ rate of~\cite{durmus2019analysis} for averaged Langevin Monte Carlo.

Finally, some of the aforementioned results from the literature also hold in settings beyond the log-concave case.
Our framework can also address such cases, but for~\ref{eq:ULMC} it would merely reproduce results from~\cite{Zhang+23ULMC}, so we omit them.
Instead, we next present results in the non-log-concave case for the more sophisticated randomized midpoint discretization which could not be achieved by existing techniques.

\subsection{Randomized midpoint discretization}\label{ssec:rmulmc}

We now turn to our main application: analysis of the randomized midpoint discretization of~\ref{eq:ULD}.
Proofs for this subsection are given in \S\ref{app:rmulmc}.

This discretization was introduced in~\cite{shen2019randomized} and was studied further in~\cite{hebalasubramanianerdogdu2020randomizedmidpoint, KanNag24PoissonMidpt, YuKarDal24RandMidpt}; it is even an order-optimal discretization in the sense of information-based complexity~\cite{caoluwang2021uldlowerbd}.
Below, our formulation differs slightly from the original algorithm of~\cite{shen2019randomized} but follows the same general intuition:
\begin{align*}\tag{$\msf{RM}$--$\msf{ULMC}$}\label{eq:rm-ulmc}
\begin{aligned}
    \hat X_{nh + \msf u_n h}^+ &\deq \hat X_{nh} + \frac{1 - e^{-\gamma \msf u_n h}}{\gamma}\, \hat P_{nh} - \frac{1}{\gamma}\, \bigl(\msf u_n h - \frac{1-e^{-\gamma \msf u_n h}}{\gamma}\bigr)\, \nabla V(\hat X_{nh}) + \xi_{nh,\msf u_n h}^{(1)}\,, \\
    \hat X_{nh + \msf v_n h}^{++} &\deq \hat X_{nh} + \frac{1 - e^{-\gamma \msf v_n h}}{\gamma}\, \hat P_{nh} - \frac{1}{\gamma}\, \bigl(\msf v_n h - \frac{1-e^{-\gamma \msf v_n h}}{\gamma}\bigr)\, \nabla V(\hat X_{nh}) + \xi_{nh,\msf v_n h}^{(1)}\,, \\
    \hat X_{(n+1)h} &\deq \hat X_{nh} + \frac{1 - e^{-\gamma h}}{\gamma}\, \hat P_{nh} - \frac{1}{\gamma}\,\bigl(h - \frac{1-e^{-\gamma h}}{\gamma}\bigr)\,\nabla V(\hat X_{nh + \msf u_n h}^+) + \xi_{nh,h}^{(1)}\,,\\
    \hat P_{(n+1)h} &\deq e^{-\gamma h}\, \hat P_{nh} - \frac{1}{\gamma}\, (1-e^{-\gamma h})\, \nabla V(\hat X_{nh + \msf v_n h}^{++}) + \xi_{nh,h}^{(2)}\,.
\end{aligned}
\end{align*}

\begin{remark}[Double midpoint implementation]\label{rem:double-midpoint}
    This implementation of~\ref{eq:rm-ulmc} differs slightly from existing implementations as it generates \emph{two} midpoint discretizations  $\hat X_{nh + \msf u_n h}^+$ and $\hat X_{nh + \msf v_n h}^{++}$. Note that these are actually samples within the same stochastic process $(\hat X_t^+)_{t \in [nh, (n+1)h]}$ but we use different notation for the two variables for clarity. This enables removing the second term in the existing rates for~\ref{eq:rm-ulmc}, which are of the form $\Otilde(\kappa \,(d/\eps^2)^{1/3}  + \kappa^{7/6}\, (d/\eps^2)^{1/6})$. See Remark~\ref{rmk:double_midpt_justify} for a more detailed justification of the particular form of~\ref{eq:rm-ulmc}.
\end{remark}

Here, $\{(\msf u_n,\msf v_n)\}_{n\in\N}$ is a sequence of i.i.d.\ random variables on $[0, 1]^2$, independent of the Brownian motion, with distribution %
\begin{align}\label{eq:interpolant-dist}
    \Pr(\msf u \in A) = \int_{A \cap [0,1]} \frac{h\,(1- e^{-\gamma (1-u)h})}{h-(1-e^{-\gamma h})/\gamma}\,\D u\,, \qquad
    \Pr(\msf v \in A) = \int_{A \cap [0,1]} \frac{\gamma h e^{-\gamma (1-v)h}}{1-e^{-\gamma h}}\,\D v\,.
\end{align}
For our analysis, only the marginal distributions of $\msf u$ and $\msf v$ are relevant and they can be coupled arbitrarily.
Also, for $n = 0,1, 2, \dotsc$, $\{\xi_{nh, t}\}_{t \in [0,h]} = \{(\xi_{nh, t}^{(1)}, \xi_{nh, t}^{(2)})\}{}_{t \in [0,h]}$ is the random process on $\R^{2d}$ given by the following It\^o integrals:
\begin{align}\label{eq:rmulmc-brownian}
\begin{aligned}
    \xi_{nh,t}^{(1)} &\deq \sqrt{\frac{2}{\gamma}} \int_{nh}^{nh+t} (1-e^{-\gamma\,(nh+t-s)}) \, \D B_s\,, \\
    \xi_{nh,t}^{(2)} &\deq \sqrt{2\gamma} \int_{nh}^{nh+t} e^{-\gamma\,(nh+t-s)} \, \D B_s\,.
\end{aligned}
\end{align}
For fixed $(\msf u_n,\msf v_n)$, the random variables $(\xi_{nh, \msf u_n h}^{(1)}, \xi_{nh,\msf v_n}^{(1)}, \xi_{nh, h}^{(1)}, \xi_{nh, h}^{(2)})$ are jointly Gaussian, with zero mean and an explicit covariance formula given in the lemma below, and independent of any $\{(\xi_{mh, t}^{(1)}, \xi_{mh,t}^{(2)})\}_{t \in [0,h]}$ when $m \neq n$.

\begin{lemma}[{Implementation of~\ref{eq:rm-ulmc}}]
    Let $0 \le s\le t \le h$. Consider the random variables defined in~\eqref{eq:rmulmc-brownian}. Then,
    \begin{align*}
        \var(\xi_{nh, h}^{(2)}) &= (1-e^{-2\gamma h})\,I\,, \\
        \cov(\xi_{nh, s}^{(1)},\, \xi_{nh, t}^{(1)}) &= \frac{2}{\gamma}\,\Bigl(s - \frac{1-e^{-\gamma s} +e^{-\gamma (t-s)} - e^{-\gamma t}}{\gamma} + \frac{e^{-\gamma(t-s)} - e^{-\gamma(t+s)}}{2\gamma} \Bigr) \,I\,, \\[0.25em]
        \cov(\xi_{nh, t}^{(1)},\, \xi_{nh, h}^{(2)}) &= \frac{1}{\gamma}\, (e^{-\gamma (h-t)} - 2e^{-\gamma h} + e^{-\gamma(h+t)})\, I \,.
    \end{align*}
\end{lemma}
\begin{proof}
    This is a straightforward computation using the It\^o isometry.
\end{proof}

The intuition behind~\ref{eq:rm-ulmc} is the following.
The explicit solution of~\ref{eq:ULD} is given by
\begin{align}\label{eq:uld_sol}
\begin{aligned}
    X_{(n+1)h}
    &= X_{nh} + \frac{1-e^{-\gamma h}}{\gamma} \,P_{nh} - \frac{1}{\gamma} \int_{nh}^{(n+1)h} (1-e^{-\gamma\,((n+1)h-t)})\,\nabla V(X_t)\,\D t + \xi_{nh,h}^{(1)}\,, \\
    P_{(n+1)h}
    &= e^{-\gamma h}\,P_{nh} - \int_{nh}^{(n+1)h} e^{-\gamma\,((n+1)h-t)}\,\nabla V(X_t)\,\D t + \xi^{(2)}_{nh,h}\,.
\end{aligned}
\end{align}
On the other hand, the distribution of $(\msf u,\msf v)$ is chosen so that if we take the expectation over $(\msf u_n,\msf v_n)$ with the Brownian motion held fixed,
\begin{align}\label{eq:rmulmc_expectation}
\begin{aligned}
    \E_{\msf u_n} \hat X_{(n+1)h}
    &= X_{nh} + \frac{1-e^{-\gamma h}}{\gamma}\,\hat P_{nh} - \frac{1}{\gamma} \int_{nh}^{(n+1)h} (1-e^{-\gamma\,((n+1)h-t)})\,\nabla V(\hat X_{t}^+)\,\D t + \xi^{(1)}_{nh,h}\,, \\
    \E_{\msf v_n} \hat P_{(n+1)h}
    &= e^{-\gamma h}\,\hat P_{nh} - \int_{nh}^{(n+1)h} e^{-\gamma\,((n+1)h-t)}\,\nabla V(\hat X_t^{++})\,\D t + \xi_{nh,h}^{(2)}\,.
\end{aligned}
\end{align}
This leads to small weak errors.
Moreover, the specific form of~\ref{eq:rm-ulmc} is chosen to facilitate the computation of the weak and strong errors, see Remark~\ref{rmk:double_midpt_justify} for further discussion.
We prove the following statement.

\begin{lemma}[{Local errors for~\ref{eq:rm-ulmc}}]\label{lem:local_error_rmulmc}
    Assume that $-\beta I \preceq \nabla^2 V \preceq \beta I$, $h \lesssim \beta^{-1/2} \wedge \gamma/\beta$, and $\gamma \lesssim \sqrt{\beta}$. Then, the local errors for~\ref{eq:rm-ulmc} are bounded as follows:
    \begin{enumerate}
        \item (Weak error) $\mc E^{\rm w}(x,p) \lesssim \beta^2 h^4\,\norm p + \beta^2 \gamma^{1/2} d^{1/2} h^{9/2} + \beta^2 h^5\,\norm{\nabla V(x)}$.
        \item (Strong error) $\mc E^{\rm s}(x,p) \lesssim \beta h^2\,\norm p + \beta\gamma^{1/2} d^{1/2} h^{5/2} + \beta h^3\,\norm{\nabla V(x)}$.

    \end{enumerate}
\end{lemma}

\subsubsection{Log-concavity}

We first state the rates when $V$ is convex.

\begin{theorem}[\ref{eq:rm-ulmc}, convex case]\label{thm:rmulmc_cvx}
    Let $\hat{\bs P}$, $\hat{\bs P}{}'$ denote the kernels for~\ref{eq:rm-ulmc} and~\ref{eq:ULMC} respectively, and let $\bs P$ denote the kernel for~\ref{eq:ULD} run for time $h$.
    Assume that $0 \preceq \alpha I \preceq \nabla^2 V \preceq \beta I$ and set $\gamma = \sqrt{32\beta}$.
    Write $W^2 \deq \mc W_2^2(\bs\mu,\bs\pi)$, where $\mc W_2$ denotes the $2$-Wasserstein distance in the twisted norm $(x,p) \mapsto \sqrt{\norm x^2 + \gamma_0^{-2}\,\norm p^2}$.
    \begin{enumerate}
        \item \underline{$\alpha > 0$ case:} If $\varepsilon \in (0, \widetilde O(\frac{d^{1/2}}{\kappa^{3/2}})]$ and $h = \Thetatilde( \frac{\varepsilon^{2/3}}{\beta^{1/2} d^{1/3}})$, then $\KL(\bs\mu \hat{\bs P}{}^{N-1}\hat{\bs P}{}' \mmid \bs\pi) \le \varepsilon^2$ for all
        \begin{align*}
            N = \Omegatilde\Bigl(\frac{\kappa d^{1/3}}{\varepsilon^{2/3}} \log \frac{\alpha W^2}{\varepsilon^2}\Bigr)\,.
        \end{align*}

        \item \underline{$\alpha = 0$ case:} If $\varepsilon\in  (0,\widetilde O(\frac{\beta^{3/4} W^{3/2}}{d^{1/4}})]$ and $h = \Thetatilde(\frac{\varepsilon}{\beta^{3/4} d^{1/4} W^{1/2} + \beta W})$, then $\KL(\bs\mu\hat{\bs P}{}^{N-1}\hat{\bs P}{}' \mmid \bs\pi) \le \varepsilon^2$ for
        \begin{align*}
            N = \Thetatilde\Bigl(\frac{\beta^{5/4} d^{1/4} W^{5/2} + \beta^{3/2} W^3}{\varepsilon^3}\Bigr)\,.
        \end{align*}
    \end{enumerate}
\end{theorem}

In the strongly convex case, the $\kappa d^{1/3}/\varepsilon^{2/3}$ rate was established in the $W_2$ metric in~\cite{shen2019randomized}.
Theorem~\ref{thm:rmulmc_cvx} is the first result to establish this rate in KL divergence, and therefore in total variation distance as well (by Pinsker's inequality).
On the other hand, the recent work of~\cite{KanNag24PoissonMidpt} claimed the rate $\kappa^{17/12} d^{5/12}/\varepsilon^{1/2}$ in total variation distance, which is incomparable (it has worse dependence on $\kappa$ and $d$, but better dependence in $1/\varepsilon$).
We highlight in particular that in the constant accuracy (e.g., $\varepsilon = 1/100$) and constant $\kappa$ regime, our result is the first to show that one can obtain a sample with $\varepsilon$ total variation error with dimension dependence scaling as $d^{1/3}$; the prior best result in this regime is the aforementioned $d^{5/12}$ rate.
We remark that this regime is particularly relevant for the problem of obtaining warm starts for high-accuracy samplers~\cite{AltChe24Warm}.

Even more recently, the concurrent work of~\cite{SriNag25PoisMidpt} claims an incomparable $W_2$ rate of roughly $\kappa^{7/6} d^{1/3}/\varepsilon^{1/3}$.
It is an interesting future direction to determine if the strengths of that result and Theorem~\ref{thm:rmulmc_cvx} can be combined. 

In the weakly convex case, no previous result for~\ref{eq:rm-ulmc} was known.
Our rate, which reads $(d/\varepsilon^2)^{1/4}\,(\beta W^2/\varepsilon^2)^{5/4} + (\beta W^2/\varepsilon^2)^{3/2}$, compares favorably with the rates listed in \S\ref{ssec:ulmc}.

\subsubsection{Discretization bound}

In the non-convex case, we state a general discretization bound for~\ref{eq:rm-ulmc}, which can be combined with any convergence result for~\ref{eq:ULD}, as illustrated in the following two subsections. 

\begin{theorem}[{\ref{eq:rm-ulmc} discretization bound}]\label{thm:rmulmc_discretization}
    Let $\hat{\bs P}$, $\hat{\bs P}{}'$ denote the kernels for~\ref{eq:rm-ulmc} and~\ref{eq:ULMC} respectively, and let $\bs P$ denote the kernel for~\ref{eq:ULD} run for time $h$.
    Assume that $-\beta I \preceq \nabla^2 V \preceq \beta I$, $\gamma \le \sqrt{32\beta}$, and $h = \widetilde O(\frac{\gamma}{\beta} \wedge \frac{\gamma^{1/3}}{\beta^{5/6} T^{1/3}})$.
    Then, for $T \deq Nh$ and $C_{\chi^2} \deq \log(1+\chi^2(\bs\mu\mmid\bs\pi))/d$,
    \begin{align*}
        \KL(\bs\mu\hat{\bs P}{}^{N-1}\hat{\bs P}{}' \mmid \bs\mu \bs P^N) = \widetilde O\Bigl( (1+C_{\chi^2})\,(1 + \beta^{1/2} T)\,\frac{\beta^2 dh^3}{\gamma} \Bigr)\,.
    \end{align*}
\end{theorem}

\subsubsection{Entropic hypocoercivity}

Next, we combine Theorem~\ref{thm:rmulmc_discretization} with convergence in KL divergence of~\ref{eq:ULD}, which is established via the entropic hypocoercivity argument of~\cite{villani2009hypocoercivity}.

First, recall that a measure $\pi \in \mc P(\R^d)$ satisfies a logarithmic Sobolev inequality (LSI) with constant $1/\alpha$ if for all compactly supported and smooth functions $f: \R^d \to \R_+$,
\begin{align}\label{eq:lsi}\tag{$\msf{LSI}$}
    \operatorname{ent}_\pi f \deq \E_\pi \Bigl[f \log \frac{f}{\E_\pi f} \Bigr] \leq \frac{1}{2\alpha}\, \E_\pi[\norm{\nabla f}^2]\,.
\end{align}
Moreover, recall that the strong log-concavity condition $\nabla^2 V \succeq \alpha I$ implies the validity of~\eqref{eq:lsi}, but the LSI condition is more general.
Under this assumption, the following result of~\cite{villani2009hypocoercivity} shows the decay of a carefully chosen Lyapunov functional.

\begin{lemma}[Entropic hypocoercivity]\label{lem:entropic-hypocoercivity}
    Suppose that $\pi$ satisfies~\eqref{eq:lsi} with constant $1/\alpha$, and $-\beta I \preceq \nabla^2 V \preceq \beta I$.
    Consider the Lyapunov functional over $\mc P(\R^d\times \R^d)$,
    \begin{align*}
        \mathfrak L(\bs\nu) \deq \KL(\bs\nu \mmid \bs\pi) + \E_{\bs\nu} \biggl\langle\nabla \log \frac{\bs\nu}{\bs\pi}\,,\; \Bigl(\begin{bmatrix}
            b_1/\beta & b_2/\sqrt{\beta} \\
            b_2/\sqrt{\beta} & b_3
        \end{bmatrix} \otimes I \Bigr)\,\nabla \log \frac{\bs\nu}{\bs\pi} \biggr\rangle\,.
    \end{align*}
    There exist absolute constants $b_1, b_2, b_3\in\R$ and $c>0$ such that for any $\bs\nu_0\in \mc P(\R^d\times\R^d)$, and $\bs P_t$ the kernel of~\ref{eq:ULD} up to time $t$ with friction $\gamma \asymp \sqrt\beta$, we have %
    \begin{align*}
        \KL(\bs\nu_0\bs P_t \mmid \bs\pi)
        \le \mf L(\bs\nu_0 \bs P_t) \leq \exp\bigl({-\frac{c\alpha t}{\sqrt{\beta}}}\bigr)\, \mf L(\bs\nu_0)\,.
    \end{align*}
\end{lemma}

With this, we can establish the following result.

\begin{theorem}[{\ref{eq:rm-ulmc} under~\ref{eq:lsi}}]\label{thm:rmulmc-lsi}
    Let $\hat{\bs P}$, $\hat{\bs P}{}'$ denote the kernels for~\ref{eq:rm-ulmc} and~\ref{eq:ULMC} respectively, and let $\bs P$ denote the kernel for~\ref{eq:ULD} run for time $h$.
    Assume that $\pi$ satisfies~\eqref{eq:lsi} with constant $1/\alpha$, $-\beta I \preceq \nabla^2 V \preceq \beta I$, $\gamma = \sqrt{32\beta}$, and $\kappa \deq \beta/\alpha$.
    Write $C_{\chi^2} \deq \log(1+\chi^2(\bs\mu \mmid \bs\pi))/d$.
    Then, for all $\varepsilon \in (0, \widetilde O(1 \wedge \sqrt{\mf L(\bs \mu)})]$ and $h = \Thetatilde(\frac{\varepsilon^{2/3}}{(1+C_{\chi^2})^{1/3}\,\beta^{1/2} \kappa^{1/3} d^{1/3} \log^{1/3}(\mf L(\bs\mu)/\varepsilon^2)})$, we have $\TV(\bs\mu\hat{\bs P}{}^{N-1}\hat{\bs P}{}' \mmid \bs\pi) \le \varepsilon$ provided that
    \begin{align*}
        N = \Thetatilde\Bigl(\frac{(1+C_{\chi^2})^{1/3}\, \kappa^{4/3} d^{1/3}}{\varepsilon^{2/3}} \log^{4/3} \frac{\mf L(\bs\mu)}{\varepsilon^2}\Bigr)\,.
    \end{align*}
\end{theorem}

Regarding the dependence on the initialization, in the strongly log-concave setting $\nabla^2 V\succeq \alpha I$ with mode at $0$, one can find an initialization $\bs\mu$ satisfying $\mf L(\bs\mu) \lesssim d$ and $C_{\chi^2} \lesssim \log\kappa$ (see, e.g.,~\cite[\S D]{Zhang+23ULMC}).
In general, the dependence on $\mf L(\bs\mu)$ is only logarithmic and so does not affect the final result substantially, and bounds for $C_{\chi^2}$ can be found in~\cite[\S A]{Che+24LMC}.

The only other result for~\ref{eq:rm-ulmc} under~\ref{eq:lsi} is the result of~\cite{KanNag24PoissonMidpt}, which establishes the rate $\kappa^{17/12} d^{5/12}/\varepsilon^{1/2}$ in total variation distance; again, this is incomparable to our result.

\subsubsection{Space-time Poincar\'e inequality}\label{ssec:space_time}

Finally, we apply our framework to establish a first step towards ballistic acceleration for sampling algorithms. In particular, we show how our discretization framework lets us exploit the recent breakthrough which developed a space-time Poincar\'e inequality to establish an accelerated mixing guarantee for~\ref{eq:ULD} in the $\chi^2$ divergence~\cite{Bri23TimeAvg, CaoLuWan23Underdamped, Alb+24KineticFP, brigati2024explicit, eberle2024non}. 

We begin by introducing a weaker version of~\eqref{eq:lsi}. We say a measure $\pi$ satisfies a \emph{Poincar\'e inequality} (PI) with constant $1/\alpha$ if for all compactly supported and smooth functions $f: \R^d \to \R$,
\begin{align}\label{eq:pi}\tag{$\msf{PI}$}
    \operatorname{var}_\pi f \deq \E_\pi[(f-\E_\pi f)^2] \leq \frac{1}{\alpha}\, \E_\pi[\norm{\nabla f}^2]\,.
\end{align}
The convergence result follows from a ``space-time'' version of the Poincar\'e inequality.

\begin{lemma}[Space-time Poincar\'e inequality]\label{lem:spacetime-poincare}
    Suppose that $\pi$ satisfies~\eqref{eq:pi} with constant $1/\alpha$, and that $0 \preceq \nabla^2 V \preceq \beta I$. Then, letting $\bs P_t$ denote the kernel for~\ref{eq:ULD} up to time $t$ with friction $\gamma\asymp\sqrt\alpha$, for any measure $\bs\nu_0 \in \mc P(\R^d\times\R^d)$, there is an absolute constant $c > 0$ with
    \begin{align*}
        \chi^2(\bs\nu_0 \bs P_t \mmid \bs\pi) \leq \frac{1}{c} \exp(-c\sqrt{\alpha} t)\, \chi^2(\bs\nu_0 \mmid \bs\pi)\,.
    \end{align*}
\end{lemma}

This is referred to as a diffusive-to-ballistic speed-up because this shows that the (non-reversible) spectral gap of underdamped Langevin is of order $\sqrt{\alpha}$, compared to the spectral gap of overdamped Langevin which is $\alpha$ (by definition of the Poincar\'e inequality with constant $\alpha$). 

\begin{theorem}[{\ref{eq:rm-ulmc} with space-time Poincar\'e}]\label{thm:rmulmc-spacetime}
    Let $\hat{\bs P}$, $\hat{\bs P}{}'$ denote the kernels for~\ref{eq:rm-ulmc} and~\ref{eq:ULMC} respectively, and let $\bs P$ denote the kernel for~\ref{eq:ULD} run for time $h$.
    Assume that $\pi$ satisfies~\eqref{eq:pi} with constant $1/\alpha$, $0 \preceq \nabla^2 V \preceq \beta I$, $\gamma \asymp \sqrt{\alpha}$, and $\kappa \deq \beta/\alpha$.
    Write $\chi^2 \deq \chi^2(\bs\mu\mmid\bs\pi)$.
    Then, for all $\varepsilon \in (0, \widetilde O(\log^{1/2}(1+\chi^2))]$ and $h =\Thetatilde(\frac{\varepsilon^{2/3}}{\beta^{1/2} \kappa^{1/3} \,(d^{1/3} + \log^{1/3}(1+\chi^2)) \log^{1/3}(\chi^2/\varepsilon^2)})$, we have $\KL(\bs\mu\hat{\bs P}{}^{N-1}\hat{\bs P}{}' \mmid \bs\pi) \le {\varepsilon^2}$ provided that
    \begin{align*}
        N = \Thetatilde\Bigl(\frac{\kappa^{5/6}\,(d^{1/3} + \log^{1/3}(1+\chi^2))}{\varepsilon^{2/3}} \log^{4/3} \frac{\chi^2}{\varepsilon^2}\Bigr)\,.
    \end{align*}
\end{theorem}

This provides our second result for the log-concave setting (c.f., Theorem~\ref{thm:rmulmc_cvx}), since strong log-concavity $\nabla^2 V \succeq \alpha I$ implies LSI with constant $1/\alpha$, which implies Poincar\'e with constant $1/\alpha$. In particular, when $\log \chi^2 = \widetilde O(d)$, which is the case for the standard initialization in the strongly log-concave setting $\nabla^2 V\succeq \alpha I$, the rate reads $\kappa^{5/6} d^{5/3}/\varepsilon^{2/3}$ up to logarithmic factors.
The dimension dependence arises from the use of a Poincar\'e-type inequality; if we have a warm start $\chi^2 = \widetilde O(1)$, or an entropic version of Lemma~\ref{lem:spacetime-poincare}, then the rate would read $\kappa^{5/6} d^{1/3}/\varepsilon^{2/3}$.
Theorem~\ref{thm:rmulmc-spacetime} is the first discretization bound using the space-time Poincar\'e inequality which achieves $o(\kappa)$ dependence, and hence constitutes a first step toward acceleration for log-concave sampling.

    \paragraph*{Acknowledgments.} JMA acknowledges funding from a Sloan Research Fellowship and a Seed Grant Award from Apple. This work was initiated while JMA and SC were visiting the Institute for Mathematical Sciences at the National University of Singapore in January 2024. MSZ is funded by an NSERC CGS-D award. 

    \newpage
    \appendix
    
    \section{Proofs for Section~\ref{scn:continuous-time}}\label{app:harnack}

\subsection{Proof of Example~\ref{ex:opt-shift}}\label{app:opt_shift}

Consider the following two systems:
\[
\begin{alignedat}{2}
  \D X_t &= P_t \,\D t\,, 
    &\quad 
  \D \bar{X}_t  &= \bar{P}_t \,\D t\,, \\[1ex]
  \D P_t \; &= \sqrt{2\gamma}\,\D B_t\,, 
    &\quad 
  \D \bar{P}_t  &= \sqrt{2\gamma}\,\D B_t\,,
\end{alignedat}
\]
evolving from two initial points $(x, p)$, $(\bar x, \bar p)$ respectively. Define the auxiliary process
\begin{align*}
    \D X_t^\aux &= P_t^\aux \, \D t\,, \\
    \D P_t^\aux &= \eta_t(X_t, P_t, X_t^\aux, P_t^\aux) \, \D t + \sqrt{2\gamma} \, \D B_t\,,
\end{align*}
and note that we are only at liberty to add a drift in the momentum component, as otherwise the path measure of the auxiliary process would be singular w.r.t.\ the original path measure. Here, the added drift $\eta_t(X_t, P_t, X_t^\aux, P_t^\aux)$ can be understood as a control, and we will omit its arguments when convenient. We now want to minimize the KL divergence bound from Girsanov's theorem, subject to two marginal constraints, corresponding respectively to interpolation of the position $X_T^{\aux} = X_T$ and the momentum $P_T^{\aux} = P_T$:
\begin{align*}
    \eta_\cdot^{\msf{opt}} = \argmin_{\eta_\cdot : [0,T]\to L^2(\Pr)} \Bigl\{\int_0^T \norm{\eta_t}^2 \, \D t \quad:\quad \bar x + \bar p T + \int_0^T \int_0^t \eta_s \, \D s \, \D t = x + pT\,, \quad \bar p + \int_0^T \eta_t \, \D t = p \Bigr\}\,,
\end{align*}
noting that the Brownian motions can be ignored, as they cancel out via synchronous coupling. Solving this constrained variational problem, we find that
\begin{align*}
    \eta_t^{\msf{opt}} = \frac{(p - \bar p)}{T}\,\Bigl( 4-\frac{6t}{T}\Bigr) + \frac{(x-\bar x)}{T^2}\, \Bigl(6 - \frac{12t}{T} \Bigr)\,.
\end{align*}
Substituting the optimal shift into the dynamics, we find that
\begin{align*}
    \begin{bmatrix}
        P_t - P_t^\aux \\
        X_t - X_t^\aux
    \end{bmatrix} &= \biggl(\begin{bmatrix}
        1-\frac{4t}{T} + \frac{3t^2}{T^2} & -\frac{6t}{T^2} + \frac{6t^2}{T^3} \\[0.5em]
        t - \frac{2t^2}{T} + \frac{t^3}{T^2} & 1 - \frac{3t^2}{T^2} + \frac{2t^3}{T^3} 
    \end{bmatrix}\otimes I\biggr)     \begin{bmatrix}
        p - \bar p \\
        x - \bar x
    \end{bmatrix}\,.
\end{align*}
Inverting this to obtain expressions for $p, \bar p, x, \bar x$, and substituting this expression into the optimal shift above, we obtain
\begin{align*}
    \eta_t^{\msf{opt}} = \frac{4\,(P_t - P_t^\aux)}{T-t} + \frac{6\,(X_t - X_t^\aux)}{(T-t)^2}\,.
\end{align*}
Note that the optimal shift naturally separates into a sum of shifts for the position and momentum coordinates.
This yields the following bound for the KL divergence:
\begin{align*}
    \KL(\operatorname{law}(X_T, P_T) \mmid \operatorname{law}(\bar X_T, \bar P_T)) \le \frac{1}{4\gamma} \int_0^T \norm{\eta_t^{\msf{opt}}}^2 \, \D t = \frac{\norm{p-\bar p}^2}{\gamma T} + \frac{3 \inner{p- \bar p, x- \bar x}}{\gamma T^2} + \frac{3\,\norm{x - \bar x}^2}{\gamma T^3}\,.
\end{align*}
On the other hand, a direct computation gives
\begin{align*}
    (X_T, P_T) \sim \cN\biggl(\begin{bmatrix}
        x + pT \\
        p
    \end{bmatrix}\,,\; 2\gamma \begin{bmatrix}
        \frac{T^3}{3} & \frac{T^2}{2} \\[0.5em]
        \frac{T^2}{2} & T
    \end{bmatrix} \otimes I \biggr)\,,
\end{align*}
and $(\bar X_T, \bar P_T)$ is similarly distributed but with $(\bar x, \bar p)$ in place of $(x,p)$. Using the formula for the KL divergence between Gaussians, we can check that this exactly matches the bound above.

\subsection{Proof of Lemma~\ref{lem:random-bt-property}}\label{app:lemma_bt}

For shorthand, in this proof, we write $b_t = b_t^\lambda$ and $\eta_t = \eta_t^\msp$.

\paragraph*{Proof of (a). 
}
This follows simply because $\dot\gamma_t = \dot\eta_t$ and $t\mapsto\eta_t$ is strictly increasing.

\paragraph*{Proof of (b).}
In all cases,
    $\dot\eta_t = \omega \eta_t + \frac{\eta_t^2}{c_0}
    $.
Therefore, we must lower bound
\begin{align*}
    b_t^\alpha
    = \alpha + \bigl( \gamma_t - \omega - \frac{\eta_t}{c_0} \bigr)\,\frac{\eta_t}{2}
    = \alpha + \bigl( \bigl(1-\frac{1}{c_0}\bigr)\,\eta_t + \gamma - \omega \bigr)\,\frac{\eta_t}{2}\,.
\end{align*}

\noindent We verify the claim by cases. In the high friction case with $\alpha \ge 0$, it follows since $3\gamma^2 \ge 96\beta \ge 2\alpha$, i.e., $\gamma \ge \sqrt{32\beta} \ge 2\omega = 2\alpha/(3\gamma)$, hence the entire expression is at least $\alpha + ((1-1/c_0)\,\eta_t + \gamma/2)\,\eta_t/2 \ge \alpha + \gamma_t\eta_t/4$ for $c_0\ge 2$.
In the other cases, $\eta_t \ge \eta_0 \ge c_0\absomega$ and $\omega < 0$.
In the high friction case with $\alpha < 0$, the second term is lower bounded by $\gamma_t\eta_t/4 \ge c_0\gamma\absomega/4 \ge c_0\abs\alpha/12 \ge -2\alpha$ provided $c_0 \ge 24$, so the overall expression is at least $\gamma_t \eta_t/8$.
In the low friction case, the second term contains a term $-\omega\eta_t/2 \ge c_0\omega^2/2 = c_0\beta/6 \ge -\alpha$ provided $c_0 \ge 6$ (here, we use $\alpha > -\beta$), so the overall expression is at  least $\gamma_t \eta_t/4$.
In all cases, the claimed lower bound holds.

\paragraph*{Proof of (c).}
We wish to show
$
    b_t^\beta = \beta + \,\gamma_t \eta_t/2 - \,\dot \eta_t/2 \leq 3\gamma_t^2/4$.
Since ${\dot\eta_t \ge 0}$ and $\eta_t \le \gamma_t$, it suffices to show $\beta \le \gamma_t^2/4$.
In the high friction case, we have $\gamma_t^2 \geq \gamma^2 \ge 32\beta$. In the low friction case, we instead use $\gamma_t^2 \geq \eta_t^2 \geq c_0^2 \omega^2 = c_0^2 \beta/9 \ge 4\beta$ for $c_0 \ge 6$.

\subsection{Strongly convex and semi-convex settings}\label{app:uld_regularity}

In this section, we complete the proof of Theorem~\ref{thm:uld_regularity} in the strongly convex and semi-convex cases.
In all cases, we note that
\begin{align*}
    \int_0^t \eta_s\,\D s = c_0 \log \frac{1-\exp(-\omega\, T)}{1-\exp(-\omega\,(T-t))}\,.
\end{align*}
Hence,
\begin{align}\label{eq:cont_aux_dist_general}
    \mathtt d_t
    &\le \exp(-c'\omega_+ t)\,\Bigl( \frac{1-\exp(-\omega\,(T-t))}{1-\exp(-\omega\, T)}\Bigr)^{c'c_0}\,\mathtt d_0\,.
\end{align}
Write $\KL \deq \KL(\law(X_T,P_T) \mmid \law(\bar X_T,\bar P_T))$.

\paragraph{Strongly convex case.} %
Here, for $2c'c_0 \ge 4$ (e.g., $c_0 \ge 192$ suffices),
\begin{align*}
    \KL
    &\lesssim \frac{\mathtt d_0^2}{\gamma} \int_0^T (\gamma^2 + (\eta_t^\msp)^2)\, (\eta_t^\msp)^2 \exp(-2c'\absomega t)\,\Bigl( \frac{1-\exp(-\absomega\,(T-t))}{1-\exp(-\absomega\, T)}\Bigr)^{2c'c_0}\, \D t \\
    &\lesssim \gamma \omega^2 \mathtt d_0^2 \int_0^T \frac{\exp(-2c'\absomega t)}{(\exp(\absomega\,(T-t)) - 1)^2} \,\Bigl( \frac{1-\exp(-\absomega\,(T-t))}{1-\exp(-\absomega\, T)}\Bigr)^2 \, \D t \\
    &\qquad + \frac{\omega^4\mathtt d_0^2}{\gamma} \int_0^T \frac{\exp(-2c'\absomega t)}{(\exp(\absomega\,(T-t)) - 1)^4} \,\Bigl( \frac{1-\exp(-\absomega\,(T-t))}{1-\exp(-\absomega\, T)}\Bigr)^4 \, \D t \\
    &= \frac{\gamma \omega^2 \mathtt d_0^2}{(1-\exp(-\absomega T))^2} \int_0^T \exp(-2c'\absomega t-2\absomega\,(T-t))\, \D t \\
    &\qquad + \frac{\omega^4\mathtt d_0^2}{\gamma\,(1-\exp(-\absomega T))^4} \int_0^T \exp(-2c'\absomega t-4\absomega\,(T-t))\, \D t \\
    &\lesssim \Bigl[\frac{\gamma \absomega\,(\exp(-2c'\absomega T) - \exp(-2\absomega T))}{(1-\exp(-\absomega T))^2} + \frac{\absomega^3\,(\exp(-2c'\absomega T) - \exp(-4\absomega T))}{\gamma\,(1-\exp(-\absomega T))^4}\Bigr]\,\mathtt d_0^2\,.
\end{align*}
When $T \le \absomega^{-1}$, we can use the weakly convex bound.
On the other hand, when $T \ge \absomega^{-1}$, the bound can be simplified to $(\gamma \absomega + \absomega^3/\gamma) \exp(-2c'\absomega T)\,\mathtt d_0^2 \lesssim \alpha \exp(-c'\absomega T)\,\mathtt d_0^2$, which is equivalent up to constant factors to the stated bound.

\paragraph{Semi-convex case.}
In this case,
\begin{align*}
    \KL
    &\lesssim \frac{\mathtt d_0^2}{\gamma} \int_0^T (\gamma^2 + (\eta_t^\msp)^2)\, (\eta_t^\msp)^2\, \Bigl( \frac{\exp(\absomega\,(T-t)) - 1}{\exp(\absomega\, T) - 1}\Bigr)^{2c'c_0}\, \D t \\
    &\lesssim \gamma \omega^2 \mathtt d_0^2 \int_0^T \frac{1}{(1-\exp(-\absomega\,(T-t)))^2}\,\Bigl( \frac{\exp(\absomega\,(T-t)) - 1}{\exp(\absomega\, T) - 1}\Bigr)^2\,\D t \\
    &\qquad{} + \frac{\omega^4 \mathtt d_0^2}{\gamma} \int_0^T \frac{1}{(1-\exp(-\absomega\,(T-t)))^4}\,\Bigl( \frac{\exp(\absomega\,(T-t)) - 1}{\exp(\absomega\, T) - 1}\Bigr)^4\,\D t \\
    &= \frac{\gamma \omega^2 \mathtt d_0^2}{(\exp(\absomega T)-1)^2} \int_0^T \exp(2\absomega\,(T-t))\,\D t
    + \frac{\omega^4 \mathtt d_0^2}{\gamma\,(\exp(\absomega T)-1)^4} \int_0^T \exp(4\absomega\,(T-t))\,\D t \\
    &\lesssim \Bigl[\frac{\gamma\absomega\,(\exp(2\absomega T)-1)}{(\exp(\absomega T)-1)^2} + \frac{\absomega^3\,(\exp(4\absomega T)-1)}{\gamma\,(\exp(\absomega T)-1)^4}\Bigr]\,\mathtt d_0^2\,.
\end{align*}
We conclude by bounding $(a^{n+1} - 1)/(a-1)^{n+1} \lesssim a^{n}/(a-1)^{n}$ for $a \geq 1$ and $n \in \{1,3\}$ and the fact that $1/(1-\exp(-c\absomega T))$ is increasing in $c$.

\paragraph{Ensuring that the auxiliary process hits.}
In all cases,~\eqref{eq:cont_aux_dist_general} shows that $\mathtt d_{T-\delta} = O(\delta^{c'c_0})$ as $\delta \searrow 0$.
Therefore, by the same argument as in the weakly convex setting, $(X_{T-\delta}^\aux, P_{T-\delta}^\aux) \to (X_T, P_T)$ in $L^2$ as $\delta \searrow 0$.

    \section{Proofs for Section~\ref{sec:discretization}}\label{app:local}

\subsection{Proof of Lemma~\ref{lem:discrete_contraction}}\label{app:discrete_contraction}

Due to Lemmas~\ref{lem:decay} and~\ref{lem:discrete-decay}, in order to prove Lemma~\ref{lem:discrete_contraction} it suffices to show that $\varepsilon_t \le 1/2$.
We first record a useful bound for $\upphi_t$.

\begin{lemma}[Bounds on $\upphi_t$]\label{lem:phi_bds}
    For sufficiently small $c_0/A$, the following hold.
    \begin{enumerate}
        \item $\norm{\upphi_t-I}_{\rm op} = O(h\eta_t^\msp) \le 1/2$.
        \item $\norm{\upphi_t^{-1}-I}_{\rm op} = O(h\eta_t^\msp)$.
    \end{enumerate}
\end{lemma}
\begin{proof}
    By the definition~\eqref{eq:def_upphi}, $\norm{\upphi_t-I}_{\rm op} \lesssim \mc I\eta^\msx_t/\gamma_t \vee \mc I\eta^\msp_t \lesssim h\eta_t^\msp$.
    Since $\eta_t^\msp \lesssim c_0/(Ah)$, for sufficiently small $c_0/A$ we can make the bound on $\norm{\upphi_t-I}_{\rm op}$ at most $1/2$, which implies $\norm{\upphi_t^{-1}}_{\rm op} \le 2$.
    Then, it implies $\norm{\upphi_t^{-1}-I}_{\rm op} \le \norm{\upphi_t^{-1}}_{\rm op}\,\norm{\upphi_t-I}_{\rm op} \lesssim h\eta_t^\msp$.
\end{proof}

We next state bounds for the error terms in $\varepsilon_t$.

\begin{lemma}[Control of $\mc N_t$]\label{lem:perturbation-mat-bound}
    For any $c > 0$, if $h\lesssim \gamma^{-1} \wedge {\gamma/\beta}$ and $c_0\lesssim A$ for sufficiently small absolute constants depending only on $c$, then
    \begin{align*}
        \norm{\mc N_t}_{\operatorname{op}} \le c\eta_t^\msp\,.
    \end{align*}
\end{lemma}

\begin{lemma}[Control of $\mc M_t + \mc M_t^{\rm skew} + \mc N_t$]\label{lem:perturbation_mat_bd_ii}
    If $h\lesssim \gamma^{-1} \wedge {\gamma/\beta}$ and $c_0\lesssim A$ for sufficiently small absolute constants, then
    \begin{align*}
        \norm{\mc M_t}_{\rm op} + \norm{\mc M_t^{\rm skew}}_{\rm op} + \norm{\mc N_t}_{\rm op}
        &\lesssim \frac{\beta}{\gamma_t} + \eta_t^\msp\,.
    \end{align*}
\end{lemma}

Given these two lemmas, we can conclude the proof of Lemma~\ref{lem:discrete_contraction}.
Recall that we assume $h\lesssim \gamma^{-1} \wedge \gamma/\beta$ and $c_0\lesssim A$ for sufficiently small absolute constants, and that $\eta^\msp_t \lesssim c_0/(Ah)$.
For two quantities $C$, $D$, let us write $C \ll D$ if we can ensure that for any $c > 0$, $C \le cD$ provided that the conditions on $h$ and $c_0/A$ hold with sufficiently small constants.

\begin{proof}[Proof of Lemma~\ref{lem:discrete_contraction}]
    From Lemma~\ref{lem:contraction-cont-time}, $\lambda_{\min}(\mc M_t)\gtrsim \eta_t^\msp$.
    Since $h\eta^\msp_t \lesssim c_0/A \ll 1$, Lemmas~\ref{lem:phi_bds},~\ref{lem:perturbation-mat-bound}, and~\ref{lem:perturbation_mat_bd_ii} yield
    \begin{align*}
        \varepsilon_t
        &\lesssim \frac{\norm{\mc N_t}_{\rm op} + h\eta_t^\msp \,(\beta/\gamma_t + \eta_t^\msp)}{\eta_t^\msp}
        \ll 1\,,
    \end{align*}
    since $\beta h/\gamma_t \le \beta h/\gamma \ll 1$.
\end{proof}

We conclude this section by proving Lemmas~\ref{lem:perturbation-mat-bound} and~\ref{lem:perturbation_mat_bd_ii}.

\begin{proof}[Proof of Lemma~\ref{lem:perturbation-mat-bound}]
    Decomposing into blocks,
    \begin{align*}
        \norm{\mc N_t}_{\rm op}
        &\lesssim \mc I\eta^\msx_t + \gamma \,\mc I\eta^\msp_t + \frac{\dot\gamma_t\,\mc I\eta^\msp_t}{\gamma_t} + \frac{\dot\gamma_t \,\mc I\eta^\msx_t}{\gamma_t^2} + \frac{\beta\,\mc I\eta^\msp_t}{\gamma_t}
        \lesssim h\gamma_t\eta_t^\msp + \frac{h\dot\eta_t^\msp\eta_t^\msp}{\gamma_t} + \frac{\beta h\eta_t^\msp}{\gamma_t}\,.
    \end{align*}
    Writing $\eta_t \deq \eta_t^\msp$, and using the differential equation for $\eta$,
    \begin{align*}
        \norm{\mc N_t}_{\rm op}
        &\lesssim h\eta_t\,\Bigl[\gamma_t + \frac{\omega \eta_t + \eta_t^2/c_0 + \beta}{\gamma_t}\Bigr]\,.
    \end{align*}
    Then, we consider the terms one by one.
    First, $h\gamma_t = h\gamma + h\eta_t \lesssim h\gamma + c_0/A \ll 1$, so $h\eta_t\gamma_t \ll \eta_t$.
    Second, either $\omega < 0$ in which case we can drop the $\omega\eta_t$ term; or, $\omega > 0$, in which case $\omega \eta_t/\gamma_t \le \alpha \eta_t/\gamma^2 \le \eta_t$.
    Third, $\eta_t^2/(c_0\gamma_t) \le \eta_t/c_0 \le \eta_t$.
    This proves that
    \begin{align}\label{eq:dotgam_over_gam}
        \frac{\dot\gamma_t}{\gamma_t}
        = \frac{\omega \eta_t + \eta_t^2/c_0}{\gamma_t}
        &\lesssim \eta_t\,.
    \end{align}
    Then, $h\eta_t \dot\gamma_t/\gamma_t \lesssim h\eta_t^2 = (c_0/A)\,\eta_t \ll \eta_t$.
    Lastly, $h\eta_t \beta/\gamma_t \le (\beta h/\gamma)\,\eta_t \ll \eta_t$.
    All in all, $\norm{\mc N_t}_{\rm op} \ll \eta_t$, which completes the proof.
\end{proof}

\begin{proof}[Proof of Lemma~\ref{lem:perturbation_mat_bd_ii}]
    Since $\mc N_t$ is already handled in Lemma~\ref{lem:perturbation-mat-bound}, it suffices to bound the first two terms.
    By the proof of Lemma~\ref{lem:contraction-cont-time},
    \begin{align*}
        \norm{\mc M_t}_{\rm op}
        &\lesssim \frac{\dot\gamma_t}{\gamma_t} + \sup_{\lambda \in [\alpha,\beta]}{\bigl\{\gamma_t - \frac{b_t^\lambda}{\gamma_t}\bigr\}}
        \le \frac{\omega \eta_t + \eta_t^2/c_0}{\gamma_t} + \gamma_t
        \lesssim \gamma_t\,,
    \end{align*}
    by~\eqref{eq:dotgam_over_gam}; here, we use the fact that $\mc M_t\succeq 0$, so we only need to consider the upper eigenvalues.

    Finally, by~\eqref{eq:dotgam_over_gam},
    \begin{align*}
        \norm{\mc M_t^{\rm skew}}_{\rm op}
        &\lesssim \frac{\eta^\msx_t}{\gamma_t} + \frac{\dot\gamma_t}{\gamma_t} + \frac{\beta}{\gamma_t}
        \lesssim \frac{\beta}{\gamma_t} + \eta_t\,.
    \end{align*}
    Note that if $\gamma \gtrsim \eta_t$, then $\beta/\gamma_t \gtrsim \beta/\gamma \gtrsim \gamma$, so the overall bound can be written $\beta/\gamma_t + \eta_t$.
\end{proof}

\subsection{Proof of Lemma~\ref{lem:zeta}}\label{app:zeta}

\begin{proof}[Proof of Lemma~\ref{lem:zeta}]
    In the language of \S\ref{sec:diffuse_then_shift}, we must control
    \begin{align*}
        \zeta = \frac{\norm{\uppsi^\sh_{t_+} - \uppsi_{t_+} - (\uppsi^\sh_{t_-} - \uppsi_{t_-})}_{L^2}}{\norm{\uppsi^\sh_{t_-} - \uppsi_{t_-}}_{L^2}}\,.
    \end{align*}
    Recall from our earlier computation in Lemma~\ref{lem:discrete-decay} the expression~\eqref{eq:dynamics_diffuse_shift}, which yields for $t\in [t_-,t_+]$
    \begin{align*}
        \norm{\upphi_t(\delta X_t, \delta Z_t) - (\delta X_{t_-},\delta Z_{t_-})}
        &\le \biggl\lVert \int_{t_-}^t \begin{bmatrix}
            (-\gamma_t /2)\,I & (\gamma_t/2)\,I \\
            a_1(t) \,I + 2\,(\mc I\eta^\msp_t-1)\, \mc H_t/\gamma_t & a_2(t)\,I \end{bmatrix} \begin{bmatrix}
                \delta X_t \\
                \delta Z_t
            \end{bmatrix} \, \D t \biggr\rVert \\
        &\le \int_{t_-}^{t} \norm{(\mc M_t + \mc M_t^{\rm skew} + \mc N_t)\,\upphi_t^{-1}\,\upphi_t(\delta X_t, \delta Z_t)}\,\D t \\
        &\le \int_{t_-}^{t} \norm{\mc M_t + \mc M_t^{\rm skew} + \mc N_t}_{\rm op}\,\norm{\upphi_t^{-1}}_{\rm op}\,\norm{\upphi_t(\delta X_t, \delta Z_t)}\,\D t \\
        &\le \Bigl(2\int_{t_-}^t \norm{\mc M_t + \mc M_t^{\rm skew} + \mc N_t}_{\rm op} \, \D t\Bigr)\,\norm{(\delta X_{t_-}, \delta Z_{t_-})}\,,
    \end{align*}
    where we used Lemma~\ref{lem:phi_bds} and Lemma~\ref{lem:discrete_contraction}.
    From here, we can apply Lemma~\ref{lem:perturbation_mat_bd_ii} to conclude that the coupling constant is $O(\beta h/\gamma_{t_-} + h\eta^\msp_{t_+})$.
    Finally, it is straightforward to check (e.g., via the differential equation for $\eta^\msp$) that $\eta_{t_+}^\msp \lesssim \eta_{t_-}^\msp$.
\end{proof}

\subsection{Integral estimates}\label{app:integrals}

The goal of this subsection is to prove Theorem~\ref{thm:kl_local_error} by bounding the integrals from Lemma~\ref{lem:kl_integral}.
If we write $\ttd_t \deq \ttd_t^\sh$, then
\begin{align*}
    \ttd_t^2
    &\lesssim \exp(-c\Lambda_t)\,\ttd_0^2 + \frac{1}{h} \int_0^t \exp(-c\,(\Lambda_t - \Lambda_s))\,\ms E_s\,\D s\,, \qquad \ms E_t \deq \frac{1}{\gamma_t^2}\,\Bigl[\frac{(\bar{\mc E}^{\rm w})^2}{\lambda_t h} + \bigl(1+\frac{\beta^2 h}{\gamma_t^2\lambda_t}\bigr)\, (\bar{\mc E}^{\rm s})^2\Bigr]\,.
\end{align*}
We must bound, for $T \deq Nh$,
\begin{align*}
    \KL
    &\lesssim \frac{1}{\gamma} \int_0^T \gamma_t^2 \eta_t^2 \ttd_t^2\,\D t + \frac{1}{\gamma h^3}\,\ttd_T^2 + \bar b^2\,.
\end{align*}
We also note that the integral contains a term of the form $\gamma^{-1} \int_0^T \gamma_t^2 \eta_t^2 \exp(-c\Lambda_t)\,\ttd_0^2 \,\D t$, which is bounded by a constant times the quantity $C(\alpha,\beta,\gamma, T+Ah)$ from Theorem~\ref{thm:uld_regularity} provided $c_0$ is chosen sufficiently large.
It remains to control the remaining terms, beginning with $\ttd_t$.

\begin{lemma}[Auxiliary distance bounds]\label{lem:final-dist-bounds}
    Under the same assumptions as Theorem~\ref{thm:kl_local_error},
    \begin{align*}
        \ttd_t^2
        &\le \exp(-c\omega_+ t)\,\Bigl(\frac{1-\exp(-\omega\,(T+Ah-t))}{1-\exp(-\omega\,(T+Ah))}\Bigr)^{cc_0}\,\ttd_0^2 + \tte_t^2\,,
    \end{align*}
    where $\tte_t$ is bounded as follows.
    \begin{enumerate}
        \item \textbf{Convex and high friction case:}
        Let $T_0 \deq T+Ah-\absomega^{-1}$ and $T_1 \deq T+Ah-\absomega^{-1}\log(1+c_0\absomega/\gamma)$.
        Then,
        \begin{align*}
            \tte_t^2 \lesssim \begin{dcases}
                \frac{(\bar{\mc E}^{\rm w})^2}{\alpha^2 h^2} + \frac{(\bar{\mc E}^{\rm s})^2}{\alpha\beta^{1/2} h}\,, & t\le T_0\,, \\[0.25em]
                \Bigl[\frac{(\bar{\mc E}^{\rm w})^2}{\beta h^2} + (\bar{\mc E}^{\rm s})^2\Bigr] \, (T+Ah-t)^2 + \frac{(\bar{\mc E}^{\rm s})^2}{\beta h}\,(T+Ah-t)\,, & t\in [T_0, T_1]\,, \\[0.25em]
                \frac{(\bar{\mc E}^{\rm w})^2}{h^2}\,(T+Ah-t)^4 + \frac{(\bar{\mc E}^{\rm s})^2}{h}\,(T+Ah-t)^3 
            +\beta^2\,(\bar{\mc E}^{\rm s})^2\,(T+Ah-t)^6\,, & t \ge T_1\,.
            \end{dcases}
        \end{align*}

    \item \textbf{Semi-convex and low friction case:}
        Let $T_0 \deq T+Ah-\absomega^{-1}$.
        Then,
        \begin{align*}
            \tte_t^2 \lesssim \begin{dcases}
                \frac{(\bar{\mc E}^{\rm w})^2}{\beta^2 h^2} + \frac{(\bar{\mc E}^{\rm s})^2}{\beta^{3/2} h}\,,
                & t\le T_0\,, \\[0.25em]
                \frac{(\bar{\mc E}^{\rm w})^2}{h^2}\,(T+Ah-t)^4 + \frac{(\bar{\mc E}^{\rm s})^2}{h}\,(T+Ah-t)^3\,, & t \ge T_0\,.
            \end{dcases}
        \end{align*}
    
    \end{enumerate}
    Moreover, in all cases, $\tte_T^2 \lesssim h^2\,(\bar{\mc E}^{\rm w} + \bar{\mc E}^{\rm s})^2$.
\end{lemma}
\begin{proof}
    An explicit computation shows that
    \begin{align*}
        \exp(-c\,(\Lambda_t-\Lambda_s))
        &= \exp(-c\omega_+\, (t-s))\,\Bigl(\frac{1-\exp(-\omega\,(T+Ah-t))}{1-\exp(-\omega\,(T+Ah-s))}\Bigr)^{cc_0}\,.
    \end{align*}
    Note that there are three distinct integrals we must compute:
    \begin{align*}
        \int_0^t \frac{\exp(-c\,(\Lambda_t-\Lambda_s))}{\gamma_s^2\lambda_s}\,\D s\,, \qquad \int_0^t \frac{\exp(-c\,(\Lambda_t - \Lambda_s))}{\gamma_s^2}\,\D s\,, \qquad \int_0^t \frac{\exp(-c\,(\Lambda_t - \Lambda_s))}{\gamma_s^4\lambda_s}\,\D s\,.
    \end{align*}

    \textbf{Convex and high friction case.}
    We split into three regimes. Namely, let $T_0 \deq [T+Ah-\absomega^{-1}]_0^t$ and $T_1 \deq [T+ Ah-\absomega^{-1} \log (1+c_0\absomega/\gamma)]_0^t$, where $[\cdot]_0^t \deq (\cdot \vee 0) \wedge t$.
    When $\alpha = 0$, this reduces to $T_0 = 0$ and $T_1 = [T + Ah - c_0/\gamma]_0^t$.
    Note that $T_0 \le T_1$ provided that $\kappa \deq \beta/\alpha$ is sufficiently large, which we can and will assume.
    The definitions of these regimes ensures
    \begin{align*}
        \eta_t \asymp \begin{cases}
            \absomega \exp(-\absomega\,(T-t))\,, & t \le T_0\,, \\
            1/(T+Ah-t)\,, & t > T_0\,,
        \end{cases} \qquad\qquad \lambda_t \asymp \begin{cases}
            \absomega\,, & t \le T_0\,, \\
            \eta_t\,, & t > T_0\,,
        \end{cases}
    \end{align*}
    and
    \begin{align*}
        \gamma_t
        &\asymp \begin{cases}
            \gamma\,, & t \le T_1\,, \\
            \eta_t\,, & t > T_1\,.
        \end{cases}
    \end{align*}
    We split the integrals into $[0,T_0]$, $[T_0, T_1]$, and $[T_1, t]$.
    \begin{itemize}
        \item On $[0,T_0]$, all three integrals are proportional to $\int_0^{T_0} \exp(-c\,(\Lambda_t-\Lambda_s))\,\D s$.
        By Lemma~\ref{lem:integrals}(\ref{int1}), for all $t\in [0,T]$ (including $t > T_0$),
        \begin{align}\label{eq:int_step1}
            \frac{1}{h} \int_0^{T_0} \exp(-c\,(\Lambda_t-\Lambda_s))\,\ms E_s\,\D s
            &\lesssim \Bigl(\frac{(\bar{\mc E}^{\rm w})^2}{\gamma^2\absomega h} + \bigl(\frac{1}{\gamma^2} + \frac{\beta^2 h}{\gamma^4 \absomega}\bigr)\,(\bar{\mc E}^{\rm s})^2\Bigr) \,\frac{[1-\exp(-\absomega\,(T+Ah-t))]^{cc_0}}{\absomega h} \\
            &\lesssim \frac{(\bar{\mc E}^{\rm w})^2}{\alpha^2 h^2} + \frac{(\bar{\mc E}^{\rm s})^2}{\alpha\beta^{1/2} h} \nonumber
        \end{align}
        provided $h\lesssim 1/(\beta^{1/2}\kappa)$ (note that this requirement is not needed when $\alpha = 0$, since then this regime does not exist).
        On the other hand, for $t=T$, $[1-\exp(-\absomega Ah)]^{cc_0} \lesssim \omega^4 h^4$ for $cc_0 \ge 4$, and in this case $\tte_T^2 \lesssim h^2\,(\bar{\mc E}^{\rm w} + \bar{\mc E}^{\rm s})^2$.
        
        \item For $t\in [T_0, T_1]$, we must evaluate $\int_{T_0}^{T_1} \eta_t^{-\ell} \exp(-c\,(\Lambda_t - \Lambda_s))\,\D s$ for $\ell \in \{0,1\}$.
        By Lemma~\ref{lem:integrals}(\ref{int2}),
        \begin{align*}
            \frac{1}{h} \int_{T_0}^{T_1} \exp(-c\,(\Lambda_t - \Lambda_s))\,\ms E_s\,\D s
            &\lesssim \Bigl[\frac{(\bar{\mc E}^{\rm w})^2}{\gamma^2 h^2} + \frac{\beta^2\,(\bar{\mc E}^{\rm s})^2}{\gamma^4}\Bigr] \,(T+Ah-t)^2 + \frac{(\bar{\mc E}^{\rm s})^2}{\gamma^2 h}\,(T+Ah-t) \\
            &\lesssim \Bigl[\frac{(\bar{\mc E}^{\rm w})^2}{\beta h^2} + (\bar{\mc E}^{\rm s})^2\Bigr] \, (T+Ah-t)^2 + \frac{(\bar{\mc E}^{\rm s})^2}{\beta h}\,(T+Ah-t)\,.
        \end{align*}
        By~\eqref{eq:int_step1}, using $cc_0 \ge 2$ and $1-\exp(-x) \le x$, the integral $h^{-1} \int_0^{T_0} \exp(-c\,(\Lambda_t - \Lambda_s))\,\ms E_s\,\D s$ is bounded by terms of the same form as the immediately preceding display equation, which yields the overall bound on $\tte_t$. %
        
        Finally, for $t=T$, Lemma~\ref{lem:integrals}(\ref{int2}) yields
        \begin{align*}
            &\frac{1}{h} \int_{T_0}^{T_1} \exp(-c\,(\Lambda_T -\Lambda_s))\,\ms E_s\,\D s \\
            &\qquad \lesssim \Bigl[\frac{(\bar{\mc E}^{\rm w})^2}{\beta h^2} + (\bar{\mc E}^{\rm s})^2\Bigr] \,\frac{{(Ah)}^{c'c_0}}{{(T+Ah-T_1)}^{c'c_0-2}} + \frac{(\bar{\mc E}^{\rm s})^2}{\beta h}\,\frac{{(Ah)}^{c'c_0}}{{(T+Ah-T_1)}^{c'c_0-1}} \\
            &\qquad \lesssim \Bigl[\frac{(\bar{\mc E}^{\rm w})^2}{\beta h^2} + (\bar{\mc E}^{\rm s})^2\Bigr] \,\frac{{(A\gamma h)}^{c'c_0}}{\gamma^2} + \frac{(\bar{\mc E}^{\rm s})^2}{\beta h}\,\frac{{(A\gamma h)}^{c'c_0}}{\gamma}\,.
        \end{align*}
        For $c'c_0 \ge 4$, combined with the bound on $[0,T_0]$ in the preceding regime, $\tte_T^2\lesssim h^2\,(\bar{\mc E}^{\rm w} + \bar{\mc E}^{\rm s})^2$.
        \item For $t\in [T_1, t]$, we must evaluate $\int_{T_1}^t \eta_t^{-\ell} \exp(-c\,(\Lambda_t-\Lambda_s))\,\D s$ for $\ell \in \{2,3,5\}$. By Lemma~\ref{lem:integrals}(\ref{int2}),
        \begin{align}\label{eq:int_step2}
            \begin{aligned}
            \frac{1}{h} \int_{T_1}^t \exp(-c\,(\Lambda_t-\Lambda_s))\,\ms E_s\,\D s
            &\lesssim \frac{(\bar{\mc E}^{\rm w})^2}{h^2}\,(T+Ah-t)^4 + \frac{(\bar{\mc E}^{\rm s})^2}{h}\,(T+Ah-t)^3 \\
            &\qquad{} +\beta^2\,(\bar{\mc E}^{\rm s})^2\,(T+Ah-t)^6\,.
            \end{aligned}
        \end{align}
        From~\eqref{eq:int_step1}, $cc_0\ge 6$, and $1-\exp(-x)\le x$, the integral $h^{-1} \int_0^{T_0} \exp(-c\,(\Lambda_t - \Lambda_s))\,\D s$ is bounded by terms of the same form.
        Also, from Lemma~\ref{lem:integrals}(\ref{int2}),
        \begin{align*}
            &\frac{1}{h} \int_{T_0}^{T_1} \exp(-c\,(\Lambda_t-\Lambda_s))\,\ms E_s\,\D s \\
            &\qquad \lesssim \Bigl[\frac{(\bar{\mc E}^{\rm w})^2}{\beta h^2} + (\bar{\mc E}^{\rm s})^2\Bigr] \,\frac{{(T+Ah-t)}^{c'c_0}}{{(T+Ah-T_1)}^{c'c_0-2}} + \frac{(\bar{\mc E}^{\rm s})^2}{\beta h}\,\frac{{(T+Ah-t)}^{c'c_0}}{{(T+Ah-T_1)}^{c'c_0-1}}\,,
        \end{align*}
        and these terms can again be absorbed into the preceding ones. 

        For $t=T$, Lemma~\ref{lem:integrals}(\ref{int2}) again yields $\tte_T^2 \lesssim h^2\,(\bar{\mc E}^{\rm w} + \bar{\mc E}^{\rm s})^2$, similarly to the preceding regime.
    \end{itemize}

    \textbf{Semi-convex and low friction case.}
    Again, let $T_0 \deq [T+Ah-\absomega^{-1}]_0^t$; here,
    \begin{align*}
        \eta_t \asymp \begin{cases}
            \beta^{1/2}\,, & t \le T_0\,, \\
            1/(T+Ah-t)\,, & t > T_0\,,
        \end{cases} \qquad\qquad \lambda_t\asymp \gamma_t \asymp \eta_t\,.
    \end{align*}
    Therefore, $\beta^2 h/(\gamma_t^2 \lambda_t) \lesssim \beta^2 h/\eta_t^3 \lesssim 1$ is always negligible.
    \begin{itemize}
        \item On $[0,T_0]$, Lemma~\ref{lem:integrals}(\ref{int3}) yields for $t\in [0,T]$
        \begin{align}
            &\frac{1}{h} \int_0^{T_0} \exp(-c\,(\Lambda_t-\Lambda_s))\,\ms E_s\,\D s \nonumber\\
            &\qquad \lesssim \Bigl[\frac{(\bar{\mc E}^{\rm w})^2}{\beta^{3/2}h^2} + \frac{(\bar{\mc E}^{\rm s})^2}{\beta h}\Bigr]\,\frac{[\exp(\absomega\,(T+Ah-t))-1]^{cc_0}}{\absomega}\exp(-cc_0\absomega\,(T+Ah-T_0)) \label{eq:int_step3} \\
            &\qquad \lesssim \Bigl[\frac{(\bar{\mc E}^{\rm w})^2}{\beta^{3/2}h^2} + \frac{(\bar{\mc E}^{\rm s})^2}{\beta h}\Bigr]\,\frac{\exp(-cc_0 \absomega\,(t-T_0))}{\absomega}
            \lesssim \frac{(\bar{\mc E}^{\rm w})^2}{\beta^2 h^2} + \frac{(\bar{\mc E}^{\rm s})^2}{\beta^{3/2} h}\,. \nonumber
        \end{align}
        Also, for $t=T$,
        \begin{align*}
            \frac{1}{h} \int_0^{T_0} \exp(-c\,(\Lambda_T-\Lambda_s))\,\ms E_s\,\D s
            &\lesssim \Bigl[\frac{(\bar{\mc E}^{\rm w})^2}{\beta^{3/2}h^2} + \frac{(\bar{\mc E}^{\rm s})^2}{\beta h}\Bigr]\,\frac{(O(\absomega Ah))^{cc_0}}{\absomega}
            \lesssim h^2\,(\bar{\mc E}^{\rm w} + \bar{\mc E}^{\rm s})^2\,,
        \end{align*}
        for $cc_0 \ge 4$.
        \item On $[T_0, t]$, $h^{-1} \int_{T_0}^t \exp(-c\,(\Lambda_t - \Lambda_s))\,\ms E_s\,\D s$ is bounded by~\eqref{eq:int_step2} by exactly the same argument.
        Moreover, by~\eqref{eq:int_step3}, $h^{-1} \int_0^{T_0} \exp(-c\,(\Lambda_t-\Lambda_s))\,\ms E_s\,\D s$ is also bounded by the same terms.
        We can also drop the ${(T+Ah-t)}^6$ term, since $\beta^2\,{(T+Ah-t)}^3 \lesssim \beta^{1/2} \lesssim 1/h$.

        Similar arguments show that $\tte_T^2 \lesssim h^2\,(\bar{\mc E}^{\rm w} + \bar{\mc E}^{\rm s})^2$.
    \end{itemize}
\end{proof}

With these estimates, we can bound the main term.

\begin{lemma}[Main term in the KL bound]\label{lem:kl-dt-shifted-comp}
    Under the same assumptions as Theorem~\ref{thm:kl_local_error}, the following bounds hold.
    \begin{enumerate}
        \item \textbf{Strongly convex, high friction:} 
        \begin{align*}
            \frac{1}{\gamma} \int_0^T \gamma_t^2 \eta_t^2 \tte_t^2 \,\D t \lesssim \frac{1}{\alpha h^2}\,(\bar{\mc E}^{\rm w})^2 + \Bigl[\frac{1}{\beta^{1/2} h} \log\frac{1}{\absomega h}\Bigr]\,(\bar{\mc E}^{\rm s})^2\,.
        \end{align*}

        \item \textbf{Weakly convex, high friction:}
        \begin{align*}
            \frac{1}{\gamma} \int_0^T \gamma_t^2 \eta_t^2 \tte_t^2 \,\D t \lesssim \frac{T}{\beta^{1/2} h^2}\,(\bar{\mc E}^{\rm w})^2 + \Bigl[\frac{1}{\beta^{1/2} h} \log\frac{T}{h} + \beta^{1/2} T\Bigr]\,(\bar{\mc E}^{\rm s})^2\,.
        \end{align*}

        \item \textbf{Semi-convex, low friction:}
        \begin{align*}
            \frac{1}{\gamma} \int_0^T \gamma_t^2 \eta_t^2 \tte_t^2 \,\D t \lesssim \frac{T}{\gamma h^2}\,(\bar{\mc E}^{\rm w})^2 + \frac{1}{\gamma h}\, \Bigl[\log\frac{\beta^{-1/2} \wedge T}{h} + \beta^{1/2} T\Bigr]\,(\bar{\mc E}^{\rm s})^2\,.
        \end{align*}

    \end{enumerate}
\end{lemma}
\begin{proof}
    We split the proof into cases.
    
    \textbf{Convex and high friction case.}
    Recall the definitions of $T_0$ and $T_1$ from Lemma~\ref{lem:final-dist-bounds}, except that we now set $t=T$.
    \begin{itemize}
        \item On the interval $[0,T_0]$,
        \begin{align*}
            \frac{1}{\gamma} \int_0^{T_0} \gamma_t^2 \eta_t^2 \tte_t^2 \,\D t
            &\lesssim \gamma \int_0^{T_0} \eta_t^2 \,\Bigl(\frac{(\bar{\mc E}^{\rm w})^2}{\alpha^2 h^2} + \frac{(\bar{\mc E}^{\rm s})^2}{\alpha\beta^{1/2} h}\Bigr)\,\D t
            \lesssim \gamma\absomega \,\Bigl(\frac{(\bar{\mc E}^{\rm w})^2}{\alpha^2 h^2} + \frac{(\bar{\mc E}^{\rm s})^2}{\alpha\beta^{1/2} h}\Bigr)
            \lesssim \frac{(\bar{\mc E}^{\rm w})^2}{\alpha h^2} + \frac{(\bar{\mc E}^{\rm s})^2}{\beta^{1/2} h}\,.
        \end{align*}
        \item On the interval $[T_0, T_1]$,
        \begin{align*}
            \frac{1}{\gamma} \int_{T_0}^{T_1} \gamma_t^2 \eta_t^2 \tte_t^2 \,\D t
            &\lesssim \gamma \int_{T_0}^{T_1} \Bigl\{\frac{(\bar{\mc E}^{\rm w})^2}{\beta h^2} + (\bar{\mc E}^{\rm s})^2 + \frac{(\bar{\mc E}^{\rm s})^2}{\beta h}\,\frac{1}{T+Ah-t}\Bigr\} \,\D t \\
            &\lesssim \Bigl[\frac{(\bar{\mc E}^{\rm w})^2}{\beta^{1/2} h^2} + \beta^{1/2}\,(\bar{\mc E}^{\rm s})^2\Bigr] \,(\absomega^{-1} \wedge T_1) + \frac{(\bar{\mc E}^{\rm s})^2}{\beta^{1/2} h} \log \frac{T+Ah-T_0}{T+Ah-T_1}\,.
        \end{align*}
        In the strongly convex case $\alpha > 0$, this is further bounded by
        \begin{align*}
            \Bigl[\frac{(\bar{\mc E}^{\rm w})^2}{\beta^{1/2} h^2} + \beta^{1/2}\,(\bar{\mc E}^{\rm s})^2\Bigr] \,\frac{\beta^{1/2}}{\alpha} + \frac{(\bar{\mc E}^{\rm s})^2}{\beta^{1/2} h} \log \frac{\absomega^{-1}}{T+Ah-T_1}
            &\lesssim \frac{(\bar{\mc E}^{\rm w})^2}{\alpha h^2} + \frac{(\bar{\mc E}^{\rm s})^2}{\beta^{1/2} h} \,\Bigl(1+\log\frac{\absomega^{-1}}{T+Ah-T_1}\Bigr)\,.
        \end{align*}
        for $h \lesssim 1/(\beta^{1/2}\kappa)$.
        \item On the interval $[T_1,T]$,
        \begin{align*}
            \frac{1}{\gamma} \int_{T_1}^T \gamma_t^2 \eta_t^2 \tte_t^2\,\D t
            &\lesssim \frac{1}{\gamma} \int_{T_1}^T \Bigl[\frac{(\bar{\mc E}^{\rm w})^2}{h^2} + \frac{(\bar{\mc E}^{\rm s})^2}{h}\,\frac{1}{T+Ah-t} +\beta^2\,(\bar{\mc E}^{\rm s})^2\,(T+Ah-t)^2\Bigr]\,\D t \\
            &\lesssim \frac{(\bar{\mc E}^{\rm w})^2}{\beta^{1/2} h^2}\,\Bigl[\absomega^{-1} \log\bigl(1+\frac{c_0\absomega}{\gamma}\bigr) \vee (T-T_1)\Bigr] + \frac{(\bar{\mc E}^{\rm s})^2}{\beta^{1/2} h} \log\frac{T+Ah-T_1}{Ah} \\
            &\qquad{} + \beta^{3/2}\,(\bar{\mc E}^{\rm s})^2 \,(T+Ah-T_1)^3\,.
        \end{align*}
        Adding up the bounds on each interval yields the overall bound.
    \end{itemize}
    
    \textbf{Semi-convex and low friction case.} Recall the definition of $T_0$ from Lemma~\ref{lem:final-dist-bounds}, except that we now set $t=T$.
    \begin{itemize}
        \item On the interval $[0,T_0]$,
        \begin{align*}
            \frac{1}{\gamma} \int_0^{T_0} \gamma_t^2 \eta_t^2 \tte_t^2\,\D t
            &\lesssim \frac{\beta^2}{\gamma} \int_0^{T_0}\Bigl[\frac{(\bar{\mc E}^{\rm w})^2}{\beta^2 h^2} + \frac{(\bar{\mc E}^{\rm s})^2}{\beta^{3/2} h}\Bigr]\,\D t
            \lesssim \Bigl[\frac{(\bar{\mc E}^{\rm w})^2}{h^2} + \frac{\beta^{1/2}\,(\bar{\mc E}^{\rm s})^2}{h}\Bigr]\,\frac{T_0}{\gamma}\,.
        \end{align*}

        \item On the interval $[T_0, T]$,
        \begin{align*}
            \frac{1}{\gamma} \int_{T_0}^T \gamma_t^2 \eta_t^2 \tte_t^2\,\D t
            &\lesssim \frac{1}{\gamma} \int_{T_0}^T\Bigl[\frac{(\bar{\mc E}^{\rm w})^2}{h^2} + \frac{(\bar{\mc E}^{\rm s})^2}{h}\,\frac{1}{T+Ah-t} \Bigr]\,\D t \\
            &\lesssim \frac{T-T_0}{\gamma h^2}\,(\bar{\mc E}^{\rm w})^2 + \Bigl[\frac{1}{\gamma h} \log \frac{\beta^{-1/2} \wedge T}{h}\Bigr]\,(\bar{\mc E}^{\rm s})^2\,.
        \end{align*}
    \end{itemize}
    Adding up the bounds on each interval yields the overall bound.
\end{proof}

\begin{proof}[Proof of Theorem~\ref{thm:kl_local_error}]
    By Lemmas~\ref{lem:kl_integral} and~\ref{lem:final-dist-bounds},
    \begin{align*}
        \KL(\bs\mu_N^\alg \mmid \bs \nu_N)
        &\lesssim \frac{1}{\gamma} \int_0^T \gamma_t^2 \eta_t^2 \exp(-c\omega_+ t)\,\Bigl(\frac{1-\exp(-\omega\,(T+Ah-t))}{1-\exp(-\omega\,(T+Ah))}\Bigr)^{cc_0}\,\ttd_0^2\,\D t + \frac{1}{\gamma} \int_0^T \gamma_t^2 \eta_t^2 \tte_t^2\,\D t \\
        &\qquad{}
        + \underbrace{\frac{\exp(-c\omega_+ T)}{\gamma h^3}\,\Bigl(\frac{1-\exp(-\omega\,Ah)}{1-\exp(-\omega\,(T+Ah))}\Bigr)^{cc_0}}_{(\star)}\,\ttd_0^2
        + \frac{(\bar{\mc E}^{\rm w} + \bar{\mc E}^{\rm s})^2}{\gamma h} + \bar b^2\,.
    \end{align*}
    The first term is bounded by $C(\alpha,\beta,\gamma,T+Ah)\,\ttd_0^2$, and the second term is handled via Lemma~\ref{lem:kl-dt-shifted-comp}.
    For the first term on the second line, we claim that if $cc_0 \ge 3$, then $(\star)\lesssim C(\alpha,\beta,\gamma,T)$.
    This is shown by cases: if $\alpha = 0$, then $C(0,\beta,\gamma,T) \gtrsim 1/(\gamma T^3) \gtrsim (\star)$; if $T \le 1/\absomega$ and $\alpha \ne 0$, then $C(\alpha,\beta,\gamma, T) \gtrsim 1/(\gamma T^3) \ge \absomega^3/\gamma \gtrsim (\star)$; if $T \ge 1/\absomega$ and $\alpha > 0$, then $C(\alpha,\beta,\gamma,T) \ge \alpha\exp(-c\absomega T) \gtrsim (\absomega^3/\gamma) \exp(-c\absomega T) \gtrsim (\star)$; and if $T \ge 1/\absomega$ and $\alpha < 0$, then $C(-\beta,\beta,\gamma, T) \gtrsim \absomega^3/\gamma \gtrsim (\star)$.
    Finally, it can also be seen that the term $(\bar{\mc E}^{\rm w} + \bar{\mc E}^{\rm s})^2/(\gamma h)$ never dominates the final bound.
\end{proof}

Below, we state the integral estimates used above.

\begin{lemma}[Integral estimates]\label{lem:integrals}
    The following hold for sufficiently large $c_0$.
    \begin{enumerate}
        \item\label{int1} For $\omega > 0$ and $T_0 \le (T+Ah-1/\absomega)\wedge t$, $\int_0^{T_0} \exp(-c\,(\Lambda_t -\Lambda_s))\,\D s \lesssim [1-\exp(-\absomega\,(T+Ah-t))]^{cc_0}/\absomega$.
        \item\label{int3} For $\omega < 0$ and $T_0 \le (T+Ah-1/\absomega)\wedge t$, $\int_0^{T_0} \exp(-c\,(\Lambda_t-\Lambda_s))\,\D s \lesssim [\exp(\absomega\,(T+Ah-t))-1]^{cc_0} \exp(-cc_0\absomega\,(T+Ah-T_0))/\absomega$.
        \item\label{int2} For $t \ge T_1 \ge T_0 \ge T+Ah-1/\absomega$ and $\ell\ge 0$ an absolute constant, $\int_{T_0}^t \eta_s^{-\ell} \exp(-c\,(\Lambda_t-\Lambda_s))\,\D s \lesssim {(T+Ah-t)}^{\ell+1}$.
        Moreover, $\int_{T_0}^{T_1} \eta_s^{-\ell} \exp(-c\,(\Lambda_t-\Lambda_s))\,\D s \lesssim (T+Ah-t)^{c'c_0}/(T+Ah-T_1)^{c'c_0-\ell-1}$.
    \end{enumerate}
\end{lemma}
\begin{proof}\mbox{}
    \begin{enumerate}
        \item In this regime,
        \begin{align*}
            &\int_0^{T_0} \exp(-c\absomega\, (t-s))\,\Bigl(\frac{1-\exp(-\absomega\,(T+Ah-t))}{1-\exp(-\absomega\,(T+Ah-s))}\Bigr)^{cc_0} \, \D s \\
            &\qquad \lesssim [1-\exp(-\absomega\,(T+Ah-t))]^{cc_0} \int_0^{T_0} \exp(-c\absomega\,(t-s))\,\D s \\
            &\qquad \lesssim \frac{[1-\exp(-\absomega\,(T+Ah-t))]^{cc_0}}{\absomega}\,.
        \end{align*}
        \item Here,
        \begin{align*}
            &\int_0^{T_0} \Bigl(\frac{\exp(\absomega\,(T+Ah-t))-1}{\exp(\absomega\,(T+Ah-s))-1}\Bigr)^{cc_0}\,\D s \\
            &\qquad \lesssim [\exp(\absomega\,(T+Ah-t))-1]^{cc_0} \int_0^{T_0} \exp(-cc_0\absomega\,(T+Ah-s))\,\D s \\
            &\qquad \lesssim \frac{[\exp(\absomega\,(T+Ah-t))-1]^{cc_0}}{\absomega} \,[\exp(-cc_0\absomega\,(T+Ah-T_0)) - \exp(-cc_0\absomega\,(T+Ah))]\,.
        \end{align*}
        \item In this regime, $\eta_t \asymp 1/(T+Ah-t)$, and there is an absolute constant $c' > 0$ so that
        \begin{align*}
            \exp(-c\,(\Lambda_t - \Lambda_s))
            &= \Bigl(\frac{T+Ah-t}{T+Ah-s}\Bigr)^{c'c_0}\,.
        \end{align*}
        Hence, for $c'c_0 > \ell$,
        \begin{align*}
            \int_{T_0}^t {(T+Ah-s)}^\ell \,\Bigl(\frac{T+Ah-t}{T+Ah-s}\Bigr)^{c'c_0}\,\D s
            &\le {(T+Ah-t)}^{c'c_0} \int_{T_0}^t \frac{1}{{(T+Ah-s)}^{c'c_0-\ell}}\,\D s \\
            &\lesssim {(T+Ah-t)}^{\ell+1}\,.
        \end{align*}
        In the second case, the integral is bounded by
        \begin{align*}
            {(T+Ah-t)}^{c'c_0} \int_{T_0}^{T_1} \frac{1}{{(T+Ah-s)}^{c'c_0-\ell}}\,\D s
            &\lesssim \frac{{(T+Ah-t)}^{c'c_0}}{{(T+Ah-T_1)}^{c'c_0-\ell-1}}\,.
        \end{align*}
    \end{enumerate}
\end{proof}

    \section{Proofs for Section~\ref{sec:app}}\label{app:sampling}

\subsection{Recursive error control}

The application of Theorem~\ref{thm:kl_local_error} to sampling requires additional bookkeeping in order to bound the error terms along the algorithm iterations.
We first state two useful results which ultimately reduce the problem of bounding errors along the algorithm to bounding errors at stationarity.
Then, we provide high-level descriptions of the recursive error bounds used for the results in \S\ref{sec:app}.

The first result shows that the errors computed with respect to the algorithm iterates can be controlled in terms of errors computed with respect to the auxiliary process, provided that the errors satisfy certain Lipschitz conditions.

\begin{assumption}[Lipschitz errors]\label{as:lipschitz-errors}
    We assume that $\mc E^{\rm w}$, $\mc E^{\rm s}$, $b$ satisfy the following Lipschitz conditions:
    \begin{align*}
        \abs{\mc E^{\rm w}(x,p) - \mc E^{\rm w}(\bar x,\bar p)}
        &\le L_{\rm w,\msx}\,\norm{x-\bar x} + L_{\rm w,\msp}\,\norm{p-\bar p}\,, \\
        \abs{\mc E^{\rm s}(x,p) - \mc E^{\rm s}(\bar x,\bar p)}
        &\le L_{\rm s,\msx}\,\norm{x-\bar x} + L_{\rm s,\msp}\,\norm{p-\bar p}\,, \\
        \abs{b(x,p) - b(\bar x,\bar p)}
        &\le L_{b,\msx}\,\norm{x-\bar x} + L_{b,\msp}\,\norm{p-\bar p}\,.
    \end{align*}
    Furthermore, suppose that
    \begin{align*}
        \bar L_n \deq \max\biggl\{\frac{(L_{\rm w,\msx} + \gamma_{nh} L_{\rm w,\msp})^2}{(\omega_+ +\eta_{nh}^\msp)\,\gamma_{nh}^2 h},\,\bigl(1+\frac{\beta^2 h}{\gamma_{nh}^2\,(\omega_+ + \eta_{nh}^\msp)}\bigr)\, \frac{(L_{\rm s,\msx} + \gamma_{nh} L_{\rm s,\msp})^2}{\gamma_{nh}^2}\biggr\} \lesssim 1-L_n
    \end{align*}
    and
    \begin{align}\label{eq:cross_reg_lip}
        (L_{b,\msx} + h^{-1} L_{b,\msp})^2\,\gamma h^3 \lesssim 1
    \end{align}
    for sufficiently small implied constants.
\end{assumption}

\begin{lemma}[Changing expectations using Lipschitz errors]\label{lem:change_to_aux}
    Suppose that the assumptions of Theorem~\ref{thm:kl_local_error} hold, as well as Assumption~\ref{as:lipschitz-errors}.
    Then, the conclusion of Theorem~\ref{thm:kl_local_error} holds, except that the averaged errors $\bar b$, $\bar{\mc E}^{\rm w}$, $\bar{\mc E}^{\rm s}$ are instead computed with respect to $\{\bs\nu_n^\aux\}_{n=0}^{N-1}$, the laws of the auxiliary process.
\end{lemma}
\begin{proof}
    We temporarily write $\bar{\mc E}^{\alg,\rm w}$, $\bar{\mc E}^{\aux,\rm w}$, $\bar{\mc E}^{\alg,\rm s}$, $\bar{\mc E}^{\aux,\rm s}$ to distinguish between the errors.
    We start with the distance recursion from Lemma~\ref{lem:dsh-recursion}, which reads
    \begin{align*}
        (\mathtt d^\sh_{n+1})^2
        &\leq L_n\, (\mathtt d^\sh_n)^2 + \frac{1}{\gamma_{nh}^2}\, O\Bigl(\frac{(\bar{\mc E}^{\alg,\rm w})^2}{(\omega_+ + \eta_{nh}^\msp)\,h} + \bigl(1+\frac{\beta^2 h}{\gamma_{nh}^2\,(\omega_+ + \eta_{nh}^\msp)}\bigr)\, (\bar{\mc E}^{\alg,\rm s})^2\Bigr)\,.
    \end{align*}
    Using the Lipschitzness of the errors and Lemma~\ref{lem:distance-control-shift-inversion},
    \begin{align*}
        \bar{\mc E}^{\alg,\rm w}
        &\le \bar{\mc E}^{\aux,\rm w} + L_{\rm w,\msx}\,d_n^\msx + L_{\rm w,\msp}\,d_n^\msp
        \lesssim \bar{\mc E}^{\aux,\rm w} + (L_{\rm w,\msx} + \gamma_{nh} L_{\rm w,\msp})\,\ttd_n^\sh\,.
    \end{align*}
    (Although $\bar{\mc E}$ is technically defined as a maximum of the averaged errors over all iterations, inspection of the proof of Lemma~\ref{lem:dsh-recursion} shows that in the above inequality, it suffices to only consider the expectation over the processes at iteration $n$.)
    A similar inequality holds for $\bar{\mc E}^{\alg,\rm s}$.
    By substituting this into the distance recursion,
    \begin{align*}
        (\mathtt d^\sh_{n+1})^2
        &\leq \Bigl[L_n + \frac{1}{\gamma_{nh}^2}\, O\Bigl(\frac{(L_{\rm w,\msx} + \gamma_{nh} L_{\rm w,\msp})^2}{(\omega_+ +\eta_{nh}^\msp)\,h} +\bigl(1+\frac{\beta^2 h}{\gamma_{nh}^2\,(\omega_+ + \eta_{nh}^\msp)}\bigr)\, (L_{\rm s,\msx} + \gamma_{nh} L_{\rm s,\msp})^2\Bigr) \Bigr]\, (\mathtt d^\sh_n)^2 \\
        &\qquad{} + \frac{1}{\gamma_{nh}^2}\, O\Bigl(\frac{(\bar{\mc E}^{\aux,\rm w})^2}{(\omega_+ + \eta_{nh}^\msp)\,h} + \bigl(1+\frac{\beta^2 h}{\gamma_{nh}^2\,(\omega_+ + \eta_{nh}^\msp)}\bigr)\, (\bar{\mc E}^{\aux,\rm s})^2\Bigr)\,.
    \end{align*}
    By Assumption~\ref{as:lipschitz-errors}, the contraction factor can be bounded by $L_n^{1/2}$, say, which only replaces the constant $c$ with $c/2$ in Lemma~\ref{lem:discrete_contraction}.
    Thus, we can carry through the proof of Theorem~\ref{thm:kl_local_error} until the final iteration (cross-regularity).

    For the last step,
    \begin{align*}
        (\bar b^\alg)^2
        &\lesssim (\bar b^\aux)^2 + (L_{b,\msx} + \gamma_{(N-1)h} L_{b,\msp})^2\,(\ttd_{N-1}^\sh)^2 \\
        &\lesssim (\bar b^\aux)^2 + (L_{b,\msx} + h^{-1} L_{b,\msp})^2\,\gamma h^3 C(\alpha,\beta,\gamma,T)\,(\ttd_0^\sh)^2\,,
    \end{align*}
    where the last line follows from the proof of Theorem~\ref{thm:kl_local_error}.
    Hence, under Assumption~\ref{as:lipschitz-errors}, the extra error term can be absorbed into the remaining terms.
\end{proof}

The next results are used to bound expectations under the auxiliary process via expectations at stationarity using change of measure.

\begin{lemma}[{Change of measure; adapted from~\cite[Lemma C.3]{scr3}}]\label{lem:gradient-bound}
    Suppose that $-\beta I \preceq \nabla^2 V \preceq \beta I$ and $\bs\pi(x,p) \propto \exp(-V(x) - \frac{1}{2}\, \norm{p}^2)$. Then, for any probability measure $\bs\mu$ over $\R^d\times\R^d$,
    \begin{align*}
        \E_{\bs\mu}[\norm{\nabla V(X)}^2] &\lesssim \beta d + \beta \KL(\bs\mu \mmid \bs\pi)\,, \\
        \E_{\bs\mu}[\norm{P}^2] &\lesssim d + \KL(\bs\mu \mmid \bs\pi)\,.
    \end{align*}
\end{lemma}
\begin{proof}
    The first statement is~\cite[Lemma C.3]{scr3} and the second statement is proven along the same lines.
    (Alternatively, both statements are a special case of~\cite[Lemma C.3]{scr3} after a change of coordinates.)
\end{proof}

\begin{lemma}[{Recursive error control; adapted from~\cite[Lemma C.4]{scr3}}]\label{lem:gradient-recurse}
    Suppose that $-\beta I \preceq \nabla^2 V \preceq \beta I$ and $\bs\pi(x,p) \propto \exp(-V(x) - \frac{1}{2}\, \norm{p}^2)$.
    Write $G_n^\msx \deq \max_{k < n}\E_{\bs\nu^\aux_k}[\norm{\nabla V(X)}^2]$, $G_n^\msp \deq \max_{k < n}\E_{\bs\nu^\aux_k}[\norm{P}^2]$.
    Suppose that the following holds for iterations $n \geq 0$:
        \begin{align*}
            \KL(\bs\nu_n^\aux \mmid \bs\pi) \leq \msf A_\msx^2\,(G_n^\msx)^2 + \msf A_\msp^2\, (G_n^\msp)^2 + \msf B^2\,. 
        \end{align*}
    If $\msf A_\msx^2\beta \vee \msf A_\msp^2 \lesssim 1$ for a sufficiently small implied constant, then for all $n \geq 0$,
        \begin{align*}
            \KL(\bs\nu_n^\aux \mmid \bs\pi) &\lesssim (\msf A_\msx^2 \beta + \msf A_\msp^2)\, d + \msf B^2\,, \\
            \text{and} \qquad (G_n^\msx)^2 &\lesssim \beta d + \msf B^2 \beta\,, \\
            \text{and} \qquad (G_n^\msp)^2 &\lesssim d + \msf B^2\,.
        \end{align*}
\end{lemma}
\begin{proof}
        From Lemma~\ref{lem:gradient-bound},
        \begin{align*}
            \max_{k \leq n} \KL(\bs\nu_k^\aux \mmid \bs\pi)
            &\lesssim (\msf A_\msx^2 \beta \vee \msf A_\msp^2) \max_{k < n} \KL(\bs\nu_k^\aux \mmid \bs\pi) + (\msf A_\msx^2 \beta + \msf A_\msp^2)\, d + \msf B^2\,.
        \end{align*}
        Rearranging yields the first inequality, and substitution into Lemma~\ref{lem:gradient-bound} yields the error bounds. %
\end{proof}

\subsubsection{Convex case}\label{app:cvx_strategy}

We now describe the proof overview for the results of \S\ref{sec:app} in the convex case.
First, we compute the local errors and check that Assumption~\ref{as:lipschitz-errors} holds.
For $T \deq Nh$, an application of Theorem~\ref{thm:kl_local_error} (via Lemma~\ref{lem:change_to_aux}) then yields, for all $n < N$,
\begin{align*}
    \KL(\bs\nu_n^\aux \mmid \bs\pi)
    &\lesssim \msf A_\msx^2\,(G_n^\msx)^2 + \msf A_\msp^2 \,(G_n^\msp)^2 + \msf B^2
\end{align*}
for certain quantities $\msf A_\msx$, $\msf A_\msp$, $\msf B$; in particular, $\msf B^2$ contains a term of the form $C(\alpha,\beta,\gamma, T)\,\mc W_2^2(\bs\mu,\bs\pi)$. Note that we can use the pessimistic bound on the first term in Theorem~\ref{thm:kl_local_error}, which integrates on $[0, Nh]$ rather than $[0, nh]$, as the integrand is non-negative.
If $\msf A_\msx^2 \beta \vee \msf A_\msp^2 \lesssim 1$, then Lemma~\ref{lem:gradient-recurse} yields
\begin{align*}
    \KL(\bs\nu_n^\aux \mmid \bs\pi)
    &\lesssim (\msf A_\msx^2 \beta + \msf A_\msp^2) \,d + \msf B^2\,.
\end{align*}
In particular, this holds for iteration $n=N-1$.
Hence, Theorem~\ref{thm:kl_local_error} and Lemma~\ref{lem:change_to_aux} imply
\begin{align*}
    \KL(\bs \mu_N^\alg \mmid \bs\pi)
    &\lesssim (\msf A_\msx^2 \beta + \msf A_\msp^2) \,d + \msf B^2 + \bar b^2\,,
\end{align*}
where $\bar b$ can be computed using the bounds on $G_N^\msx$, $G_N^\msp$ provided by Lemma~\ref{lem:gradient-recurse}.
To streamline this last part, we handle the cross-regularity part when we use~\ref{eq:ULMC} for the last step.

\begin{lemma}[{\ref{eq:ULMC} cross-regularity step}]\label{lem:simplify_creg}
    Let $-\beta I \preceq \nabla^2 V \preceq \beta I$ and suppose that we use~\ref{eq:ULMC} for the last step of the algorithm.
    Then, the condition~\eqref{eq:cross_reg_lip} of Assumption~\ref{as:lipschitz-errors} is satisfied provided $h\lesssim 1/\beta^{1/2}$.
    Also, if $h \lesssim 1/\gamma$, then
    \begin{align*}
        \bar b^2
        &\lesssim \frac{\beta^2 h^3}{\gamma}\,(d+\msf B^2)\,.
    \end{align*}
\end{lemma}
\begin{proof}
    By Lemma~\ref{lem:ulmc-cross-reg}, we can take
    \begin{align*}
        b(x,p) = \frac{\beta^2 h^3}{\gamma}\,\norm p^2 + \beta^2 dh^4  + \frac{\beta^2 h^5}{\gamma}\,\norm{\nabla V(x)}^2\,,
    \end{align*}
    which satisfies the Lipschitz condition with $L_{b,\msx} \lesssim \beta^2 h^{5/2}/\gamma^{1/2}$ and $L_{b,\msp} \lesssim \beta h^{3/2}/\gamma^{1/2}$.
    Then,~\eqref{eq:cross_reg_lip} reads $(\beta^2 h^{5/2} + \beta h^{1/2})^2\,h^3 \lesssim 1$, which is satisfied for $h \lesssim 1/\beta^{1/2}$.
    The bound on $\bar b$ then follows from Lemma~\ref{lem:gradient-recurse} and simplifying.
\end{proof}

When $\beta^2 h^3/\gamma \lesssim 1$, corresponding to $h\lesssim 1/\beta^{1/2}$ when $\gamma\asymp\sqrt\beta$ and $h\lesssim 1/(\beta^{1/2} \kappa^{1/6})$ when $\gamma \asymp\sqrt\alpha$, this yields the following simplified final bound:
\begin{align}\label{eq:cvx_final_bd}
    \KL(\bs \mu_N^\alg \mmid \bs\pi)
    &\lesssim \bigl(\msf A_\msx^2 \beta + \msf A_\msp^2 + \frac{\beta^2 h^3}{\gamma}\bigr) \,d + \msf B^2\,.
\end{align}
To summarize, the main steps are (1) to check Assumption~\ref{as:lipschitz-errors}, (2) to check $\msf A_\msx^2 \beta \vee \msf A_\msp^2 \lesssim 1$, and (3) to write down the final bound.

\subsubsection{Non-convex case}\label{app:noncvx_strategy}

In the non-convex case $\alpha < 0$, the term $C(\alpha,\beta,\gamma,T)$ does not tend to zero as $T\to\infty$, so we take a different approach.
We instead apply Theorem~\ref{thm:kl_local_error} (again via Lemma~\ref{lem:change_to_aux}), where the second process $\{\bs\nu_n\}_{n\ge 0}$ follows~\ref{eq:ULD} initialized at $\bs\mu$.
Thus, for $n < N$,
\begin{align*}
    \KL(\bs\nu_n^\aux \mmid \bs\nu_n)
    &\lesssim \msf A_\msx^2\,(G_n^\msx)^2 + \msf A_\msp^2 \,(G_n^\msp)^2 + \msf B^2\,,
\end{align*}
where $\msf B$ no longer contains a term involving the initial Wasserstein distance, because the two processes share the same initialization.
Now, if we assume that the initialization satisfies $\log(1+\chi^2(\bs\mu_0\mmid\bs\pi)) \le C_{\chi^2} d$, then the weak triangle inequality (Proposition~\ref{prop:weak-triangle}) and the data-processing inequality (Proposition~\ref{prop:data-processing}) yield
\begin{align*}
    \KL(\bs\nu_n^\aux \mmid\bs\pi)
    &\le 2\KL(\bs\nu_n^\aux \mmid \bs \nu_n) + \log(1+\chi^2(\bs\nu_n\mmid \bs \pi))
    \le 2\KL(\bs\nu_n^\aux \mmid \bs \nu_n) + \log(1+\chi^2(\bs\mu_0\mmid \bs \pi)) \\
    &\lesssim \msf A_\msx^2\,(G_n^\msx)^2 + \msf A_\msp^2 \,(G_n^\msp)^2 + \msf B^2 + C_{\chi^2} d\,.
\end{align*}
By Lemma~\ref{lem:gradient-recurse}, if $\msf A_\msx^2 \beta \vee \msf A_\msp^2 \lesssim 1$, then
\begin{align*}
    (G_n^\msx)^2 \lesssim (1+C_{\chi^2})\,\beta d + \msf B^2 \beta \qquad\text{and}\qquad (G_n^\msp)^2 \lesssim (1+C_{\chi^2})\,d + \msf B^2\,.
\end{align*}
Substituting this back into Theorem~\ref{thm:kl_local_error},
\begin{align*}
    \KL(\bs\mu_N^\alg \mmid \bs\nu_N)
    &\lesssim (1+C_{\chi^2})\,(\msf A_\msx^2 \beta + \msf A_\msp^2)\,d + \msf B^2 + \bar b^2\,.
\end{align*}
By Lemma~\ref{lem:simplify_creg}, for $h \lesssim \gamma^{1/3}/\beta^{2/3}$, the cross-regularity for~\ref{eq:ULMC} yields
\begin{align}\label{eq:noncvx_final_bd}
    \KL(\bs\mu_N^\alg \mmid \bs\nu_N)
    &\lesssim \bigl[(1+C_{\chi^2})\,(\msf A_\msx^2 \beta + \msf A_\msp^2) + \frac{\beta^2 h^3}{\gamma}\bigr]\,d + \msf B^2\,.
\end{align}
This discretization bound can then be combined with convergence results $\bs\nu_N\to\bs\pi$ for~\ref{eq:ULD}.

\subsection{Proofs for Subsection~\ref{ssec:ulmc}}\label{app:ulmc}

We now verify Assumption~\ref{as:lipschitz-errors} for~\ref{eq:ULMC}.

\begin{lemma}[Lipschitz errors for~\ref{eq:ULMC}]\label{lem:lipschitz-error-ulmc}
    For~\ref{eq:ULMC}, under the assumptions of Lemma~\ref{lem:local_error_ulmc}, Assumption~\ref{as:lipschitz-errors} holds in both strongly and weakly convex cases (with high friction) if additionally $h \lesssim 1/(\beta^{1/2}\kappa)$ in the strongly convex case or $h \lesssim 1/(\beta T)$ in the weakly convex case.
\end{lemma}
\begin{proof}
    In each case, we take
    \begin{align*}
        L^2_{\rm w, \msx} \lesssim \beta^4 h^6\,, \qquad L^2_{\rm w,\msp} \lesssim \beta^2 h^4\,,
    \end{align*}
    and take $L^2_{\rm s, \msx}, L^2_{\rm s, \msp}$ to be the same values (ignoring them as they do not dominate the maximum, as $(\absomega + \eta_{nh})\,h \lesssim 1$ for all $n$).  In the strongly convex case, it suffices to have $h\lesssim \absomega/\beta = 1/(\beta^{1/2}\kappa)$.
    In the weakly convex case, we instead need $h\lesssim \eta_{nh}/\beta$, so it suffices to have $h\lesssim 1/(\beta T)$ as claimed.
\end{proof}

\begin{proof}[Proof of Theorem~\ref{thm:ulmc_cvx}]
    We follow the outline of \S\ref{app:cvx_strategy}.
    In this proof, we use the trivial bound $\bar{\mc E}^{\rm w} \le \bar{\mc E}^{\rm s}$ by Jensen's inequality and we therefore ignore the strong error terms.

    \textbf{Strongly convex case, $\alpha > 0$:}
    Theorem~\ref{thm:kl_local_error} (via Lemmas~\ref{lem:change_to_aux} and~\ref{lem:lipschitz-error-ulmc}) yields, for $n < N$,
    \begin{align*}
        \KL(\bs\nu_n^\aux \mmid \bs\pi)
        &\lesssim \underbrace{\frac{\beta^2 h^4}{\alpha}}_{\msf A_\msx^2}\,(G_n^\msx)^2 + \underbrace{\frac{\beta^2 h^2}{\alpha}}_{\msf A_\msp^2}\,(G_n^\msp)^2 + \underbrace{C(\alpha,\beta,\gamma,T)\,\mc W_2^2(\bs\mu,\bs\pi) + \frac{\beta^{5/2} dh^3}{\alpha}}_{\msf B^2}\,.
    \end{align*}
    The condition $\msf A_\msx^2 \beta \vee \msf A_\msp^2 \lesssim 1$ is met for $h\lesssim 1/(\beta^{1/2} \kappa^{1/2})$.
    Plugging this into~\eqref{eq:cvx_final_bd} and simplifying,
    \begin{align*}
        \KL(\bs\mu_N^\alg \mmid \bs\pi)
        &\lesssim C(\alpha,\beta,\gamma,T)\,\mc W_2^2(\bs\mu,\bs\pi) + \frac{\beta^2 dh^2}{\alpha}\,.
    \end{align*}
    Taking $h$ as given and noting that $C(\alpha,\beta,\gamma,T)\,\mc W_2^2(\bs\mu,\bs\pi) \le \varepsilon^2$ for $T \gtrsim \absomega^{-1}\log(\alpha W^2/\varepsilon^2)$ finishes the proof in this case.

    \textbf{Weakly convex case, $\alpha = 0$:}
    Theorem~\ref{thm:kl_local_error} (via Lemmas~\ref{lem:change_to_aux} and~\ref{lem:lipschitz-error-ulmc}) yields, for $n < N$,
    \begin{align*}
        \KL(\bs\nu_n^\aux \mmid \bs\pi)
        &\lesssim \underbrace{\beta^{3/2} h^4 T}_{\msf A_\msx^2}\,(G_n^\msx)^2 + \underbrace{\beta^{3/2} h^2 T}_{\msf A_\msp^2}\,(G_n^\msp)^2 + \underbrace{C(0,\beta,\gamma,T)\,\mc W_2^2(\bs\mu,\bs\pi) + \beta^2 dh^3 T}_{\msf B^2}\,.
    \end{align*}
    The condition $\msf A_\msx^2 \beta \vee \msf A_\msp^2 \lesssim 1$ is met for $h\lesssim 1/(\beta^{3/4} T^{1/2})$.
    Plugging this into~\ref{eq:cvx_final_bd} and simplifying,
    \begin{align*}
        \KL(\bs\mu_N^\alg \mmid \bs\pi)
        &\lesssim C(0,\beta,\gamma,T)\,\mc W_2^2(\bs\mu,\bs\pi) + \beta^{3/2} dh^2 T\,.
    \end{align*}
    The first term is at most $\varepsilon^2$ if we take $T \asymp \beta^{1/2} W^2/\varepsilon^2$.
    We then choose $h$ to make the second term small, which finishes the proof in this case.
\end{proof}

\subsection{Proofs for Subsection~\ref{ssec:rmulmc}}\label{app:rmulmc}

Before proceeding to the proof of Lemma~\ref{lem:local_error_rmulmc}, we state two helpful inequalities. The first follows as a straightforward adaptation of~\cite[\S5]{chewibook}.

\begin{lemma}[{\ref{eq:ULD} movement bound}]\label{lem:uld-gronwalls}
    Let ${(X_t, P_t)}_{t \in [0,h]}$ denote the solution of~\ref{eq:ULD} started at $(x,p)$. With $h \lesssim 1/\beta^{1/2}$ and $\gamma \lesssim \sqrt{\beta}$, we have the following bounds:
    \begin{align*}
        \bigl\lVert\sup_{t \in [0, h]} \, \norm{\nabla V(X_t) -\nabla V(x)}\bigr\rVert_{L^2} &\leq \beta\, \bigl\lVert\sup_{t \in [0, h]} \, \norm{X_t-x}\bigr\rVert_{L^2} \\
        &\lesssim \beta h\, \norm{p} + \beta \gamma^{1/2} d^{1/2} h^{3/2} + \beta h^2\,\norm{\nabla V(x)}\,.
    \end{align*}
\end{lemma}

\begin{lemma}[Midpoint bound]\label{lem:grad-var-rmulmc}
    Let $h \lesssim 1/\beta^{1/2}$, $\gamma \lesssim \sqrt{\beta}$, and let $(\hat X_{\msf u h}^+, X_{\msf uh})$ be as defined in~\ref{eq:rm-ulmc} and~\ref{eq:ULD} respectively, where both processes are started at $(x, p)$ and synchronously coupled with the same Brownian motion ${\{B_t\}}_{t \geq 0}$. Then, for any $\msf u \in [0, 1]$,
    \begin{align*}
        \norm{\nabla V(\hat X_{\msf u h}^+) - \nabla V(X_{\msf u h})}_{L^2} \lesssim \beta^{2} h^3\, \norm{p}
        + \beta^{2} \gamma^{1/2}d^{1/2} h^{7/2}
        + \beta^{2} h^4\,\norm{\nabla V(x)}\,.
    \end{align*}
    In particular, this also holds when $\msf u$ is drawn from the distribution specified in \eqref{eq:interpolant-dist}.
\end{lemma}
\begin{proof}
    Recalling~\eqref{eq:uld_sol},
    \begin{align*}
        \norm{\nabla V(\hat X^+_{\msf u h}) -  \nabla V(X_{\msf u h})}
        &\leq \beta \, \norm{\hat X^+_{\msf u h} -  X_{\msf u h}} \\
        &= \frac{\beta}{\gamma} \,\Bigl\lVert\int_0^{\msf u h} (1-e^{-\gamma\, (\msf u h-s)})\,(\nabla V(X_s)-\nabla V(x)) \, \D s\Bigr\rVert \\
        &\leq \frac{\beta}{\gamma} \, \int_0^{\msf u h} (1 - e^{-\gamma\,(\msf u h - s)}) \, \D s \cdot \sup_{s \in [0, h]} \norm{\nabla V(X_s)-\nabla V(x)} \\
        &\lesssim \beta h^2 \sup_{s \in [0, h]}  \, \norm{\nabla V(X_s)-\nabla V(x)}\,.
    \end{align*}
    The bound on $\norm{\sup_{t \in [0, h]} \norm{\nabla V(X_s) - \nabla V(x)}}_{L^2}$ follows from Lemma~\ref{lem:uld-gronwalls}, and we note that this bound does not depend at all on the choice of $\msf u$.
\end{proof}

\begin{proof}[{Proof of Lemma~\ref{lem:local_error_rmulmc}}]
    The proof is an adaptation of~\cite[Lemma 2]{shen2019randomized}.

    \textbf{Weak error.}
    Comparing~\eqref{eq:uld_sol} to~\eqref{eq:rmulmc_expectation} after one iteration, we have
    \begin{align*}
        \norm{\E\hat X_h - \E X_h}
        &= \frac{1}{\gamma} \,\Bigl\lVert \int_0^h (1-e^{-\gamma\,(h-t)})\,\E[\nabla V(\hat X_t^+) - \nabla V(X_t)]\,\D t\Bigr\rVert
        \lesssim h^2 \sup_{t\in [0,h]}{\norm{\nabla V(\hat X_t^+) - \nabla V(X_t)}_{L^2}}
    \end{align*}
    and
    \begin{align*}
        \norm{\E \hat P_{h} - \E P_h}
        &= \Bigl\lVert \int_0^h e^{-\gamma\, (h-t)}\, \E[\nabla V(\hat X_t^{++}) - \nabla V(X_t)] \, \D t\Bigr\rVert
        \lesssim h \sup_{t\in 0,h]}{\norm{\nabla V(\hat X_t^+) -\nabla V(X_t)}_{L^2}}\,,
    \end{align*}
    since $(\hat X_t^+, X_t)$ and $(\hat X_t^{++}, X_t)$ are equal almost surely.
    The bound on the weak error now follows from Lemma~\ref{lem:grad-var-rmulmc}.
    
    \textbf{Strong error.} Comparing~\ref{eq:rm-ulmc} with~\eqref{eq:uld_sol} after one iteration,
    \begin{align*}
        \norm{\hat X_h - X_h}
        &= \frac{1}{\gamma}\,\Bigl\lVert\int_0^h (1-e^{-\gamma\, (h - t)})\, (\nabla V(X_{\msf u h}^+) - \nabla V(X_t)) \, \D t\Bigr\rVert \\
        &\le \frac{1}{\gamma}\,\bigl(h-\frac{1-e^{-\gamma h}}{\gamma}\bigr)\, \norm{\nabla V( X_{\msf u h}^+) - \nabla V(X_{\msf u h})} \\
        &\qquad{} + \frac{1}{\gamma}\,\Bigl\lVert \int_0^h (1-e^{-\gamma\, (h - t)})\, (\nabla V(X_{\msf u h}) - \nabla V(X_t)) \, \D t\Bigr\rVert \\
        &\lesssim h^2\,\Bigl(\norm{\nabla V(X_{\msf u h}^+) - \nabla V(X_{\msf u h})} + \sup_{t \in [0, h]}{\norm{\nabla V(X_t) - \nabla V(x)}}\Bigr)\,,
    \end{align*}
    and
    \begin{align*}
        \norm{\hat P_h - P_h}
        &= \Bigl\lVert \int_0^h e^{-\gamma\, (h-t)}\, (\nabla V(\hat X_{\msf vh}^{++}) - \nabla V(X_t)) \, \D t\Bigr\rVert \\
        &\le \frac{1-\e^{-\gamma h}}{\gamma}\,\norm{\nabla V(\hat X_{\msf v h}^{++}) -\nabla V(X_{\msf v h})} + \Bigl\lVert\int_0^h e^{-\gamma\, (h-t)}\,(\nabla V(X_{\msf v h})- \nabla V(X_t)) \, \D t\Bigr\rVert \\
        &\lesssim h\,\Bigl(\norm{\nabla V(X_{\msf v h}^{++}) - \nabla V(X_{\msf v h})} + \sup_{t \in [0, h]}{\norm{\nabla V(X_t) - \nabla V(x)}}\Bigr)\,.
    \end{align*}
    The bound on the strong error now follows from Lemmas~\ref{lem:uld-gronwalls} and~\ref{lem:grad-var-rmulmc}.
\end{proof}

\begin{remark}[Justification for double midpoint]\label{rmk:double_midpt_justify}
    Here, we justify the design of our double midpoint implementation of~\ref{eq:rm-ulmc}.
    Consider the momentum variable over the time horizon $[0,h]$.
    Recall from~\eqref{eq:uld_sol} that~\ref{eq:ULD} satisfies
    \begin{align*}
        P_h = e^{-\gamma h}\, p - \int_0^h e^{-\gamma\,(h-t)}\,\nabla V(X_t)\,\D t + \xi_{h,h}^{(2)}\,.
    \end{align*}
    We seek a discretization of the form
    \begin{align*}
        \hat P_h = e^{-\gamma h}\, p - c(\msf u)\,\nabla V(\hat X^{++}_{\msf uh}) + \xi_{h,h}^{(2)}\,,
    \end{align*}
    where $c(\msf u)$ is a coefficient that could potentially depend on $\msf u$ (as well as other parameters such as $\gamma$, $h$).
    In order to have good \emph{strong} error, we impose the condition that $c(\msf u) = \int_0^h e^{-\gamma\,(h-t)}\,\D t$ so that the middle term in the difference $\hat P_h - P_h$ can be brought under the same integral.
    This forces $c(\msf u) = (1-e^{-\gamma h})/\gamma \eqqcolon c$, which does not depend on $\msf u$.
    On the other hand, in order to achieve good \emph{weak} error, we require $c\,\E_{\msf u} \nabla V(\hat X_{\msf uh}^{++}) = \int_0^h e^{-\gamma\,(h-t)}\,\nabla V(\hat X_t^{++})\,\D t$, where $\E_{\msf u}$ denotes the expectation over the law of $\msf u$.
    This second requirement forces us to take the law for $\msf u$ given in~\eqref{eq:interpolant-dist}.
    Analogous considerations apply for discretization of the position coordinate.
\end{remark}

Next, we check Assumption~\ref{as:lipschitz-errors}.

\begin{lemma}[{Lipschitz errors for~\ref{eq:rm-ulmc}}]\label{lem:lipschitz-error-rmulmc}
    Under the assumptions of Lemma~\ref{lem:local_error_rmulmc}, Assumption~\ref{as:lipschitz-errors} holds for~\ref{eq:rm-ulmc} in all cases if the following conditions on the step size $h$ hold, where $\kappa \deq \beta/\alpha$.
    \begin{enumerate}
        \item \textbf{Strongly convex, high friction ($\alpha > 0$, $\gamma \asymp\sqrt\beta$):}  $h \lesssim \frac{1}{\beta^{1/2}\kappa}$.

        \item \textbf{Weakly convex, high friction ($\alpha =0$, $\gamma\asymp\sqrt\beta$):} $h \lesssim \frac{1}{\beta^{2/3} T^{1/3}} \wedge \frac{1}{\beta^{3/4} T^{1/2}}$.

        \item \textbf{Semi-convex, any friction ($\alpha=-\beta$, $\gamma\lesssim\sqrt\beta$):} $h \lesssim \frac{1}{\beta^{1/2}}$.

    \end{enumerate}
\end{lemma}
\begin{proof}
    In each case, we take
    \begin{align*}
        L_{\rm w,\msx}^2 \lesssim \beta^6 h^{10}\,, \qquad
        L_{\rm w, \msp}^2 \lesssim \beta^4 h^8\,, \qquad L_{\rm s,\msx}^2 \lesssim \beta^4 h^6\,, \qquad L_{\rm s, \msp}^2 \lesssim \beta^2 h^4\,.
    \end{align*}
    Since $h\lesssim \gamma/\beta$, the $L_{\rm w,\msx}$ and $L_{\rm s,\msp}$ terms are negligible.
    
    In all cases but the weakly convex case, we can neglect the term involving $\beta^2 h/(\gamma_{nh}^2\,(\omega_+ + \eta_{nh}^\msp))$ (requiring $h \lesssim 1/(\beta^{1/2}\kappa)$ in the strongly convex case), and the condition becomes
    \begin{align*}
        \max\Bigl\{\frac{\beta^4 h^8}{(\omega_+ + \eta_{nh}^\msp)\,h}, \, \beta^2 h^4\Bigr\}\lesssim (\omega_+ + \eta_{nh}^\msp)\,h\,.
    \end{align*}
    In the strongly convex, high friction case, it suffices to have $h \lesssim \frac{1}{\beta^{1/2}\kappa^{1/3}}$. In the semi-convex case, since $\eta_{nh}^\msp \gtrsim \beta^{1/2}$, it suffices to take $h\lesssim \frac{1}{\beta^{1/2}}$. %
    
    Finally, in the weakly convex case, we add on the extra condition
    \begin{align*}
        \frac{\beta^2 h}{\gamma_{nh}^2 \eta_{nh}^\msp}\,\beta^2 h^4 \lesssim \eta_{nh}^\msp h\,.
    \end{align*}
    It suffices to have $h \lesssim \frac{1}{\beta^{2/3} T^{1/3}} \wedge \frac{1}{\beta^{3/4} T^{1/2}}$.
\end{proof}

\subsubsection{Log-concave case}

\begin{proof}[Proof of Theorem~\ref{thm:rmulmc_cvx}]
    We again follow the outline of \S\ref{app:cvx_strategy}.

    \textbf{Strongly convex case, $\alpha > 0$:} 
    Theorem~\ref{thm:kl_local_error} (via Lemmas~\ref{lem:change_to_aux} and~\ref{lem:lipschitz-error-rmulmc}) yields, for $h\lesssim 1/(\beta^{1/2}\kappa)$ and $n < N$,
    \begin{align*}
        \KL(\bs\nu_n^\aux \mmid \bs\pi)
        &\lesssim \underbrace{\beta^{3/2} h^5 \log\frac{1}{\absomega h}}_{\msf A_\msx^2}\,(G_n^\msx)^2 + \underbrace{\beta^{3/2} h^3 \log\frac{1}{\absomega h}}_{\msf A_\msp^2}\,(G_n^\msp)^2 \\
        &\qquad{} + \underbrace{C(\alpha,\beta,\gamma,T)\,\mc W_2^2(\bs\mu,\bs\pi) + \beta^2 dh^4 \log\frac{1}{\absomega h}}_{\msf B^2}\,.
    \end{align*}
    The condition $\msf A_\msx^2 \beta \vee \msf A_\msp^2 \lesssim 1$ is met for $h = \widetilde O(1/\beta^{1/2})$.
    Plugging this into~\eqref{eq:cvx_final_bd} and simplifying,
    \begin{align*}
        \KL(\bs\mu_N^\alg \mmid \bs\pi)
        &\lesssim C(\alpha,\beta,\gamma,T)\,\mc W_2^2(\bs\mu,\bs\pi) + \beta^{3/2} d h^3 \log\frac{1}{\absomega h}\,.
    \end{align*}
    Taking $h$ as given and noting that $C(\alpha,\beta,\gamma,T)\,\mc W_2^2(\bs\mu,\bs\pi) \le \varepsilon^2$ for $T \gtrsim \absomega^{-1}\log(\alpha W^2/\varepsilon^2)$ finishes the proof in this case.
    The condition on $\varepsilon$ arises because of the requirement $h \lesssim 1/(\beta^{1/2}\kappa)$.

    \textbf{Weakly convex case, $\alpha = 0$:}
    Theorem~\ref{thm:kl_local_error} (via Lemmas~\ref{lem:change_to_aux} and~\ref{lem:lipschitz-error-rmulmc}) yields, for $n < N$,
    \begin{align*}
        \KL(\bs\nu_n^\aux \mmid \bs\pi)
        &\lesssim \underbrace{\bigl(\frac{1}{\beta^{1/2} h} \log\frac{T}{h} + \beta^{1/2} T\bigr)\,\beta^2 h^6}_{\msf A_\msx^2}\,(G_n^\msx)^2 + \underbrace{\bigl(\frac{1}{\beta^{1/2} h} \log\frac{T}{h} + \beta^{1/2} T\bigr)\,\beta^2 h^4}_{\msf A_\msp^2}\,(G_n^\msp)^2 \\
        &\qquad{} + \underbrace{C(0,\beta,\gamma,T)\,\mc W_2^2(\bs\mu,\bs\pi) + \bigl(\frac{1}{\beta^{1/2} h} \log\frac{T}{h} + \beta^{1/2} T\bigr)\,\beta^{5/2}dh^5}_{\msf B^2}\,.
    \end{align*}
    The condition $\msf A_\msx^2 \beta \vee \msf A_\msp^2 \lesssim 1$ is met for $h\lesssim 1/(\beta^{5/8} T^{1/4})$.
    Plugging this into~\ref{eq:cvx_final_bd} and simplifying,
    \begin{align*}
        \KL(\bs\mu_N^\alg \mmid \bs\pi)
        &\lesssim C(0,\beta,\gamma,T)\,\mc W_2^2(\bs\mu,\bs\pi) + \bigl(\frac{1}{\beta^{1/2} h} \log\frac{T}{h} + \beta^{1/2} T\bigr)\,\beta^2 dh^4\,.
    \end{align*}
    The first term is at most $\varepsilon^2$ if we take $T \asymp \beta^{1/2} W^2/\varepsilon^2$.
    We then choose $h$ to make the second term small, together with the constraint $h \lesssim 1/(\beta^{3/4} T^{1/2})$, retaining only the leading order terms with respect to $1/\varepsilon$ for simplicity.
\end{proof}

\subsubsection{Non-log-concave case}

\begin{proof}[Proof of Theorem~\ref{thm:rmulmc_discretization}]
    We follow the outline of \S\ref{app:noncvx_strategy}.
    Theorem~\ref{thm:kl_local_error} (via Lemmas~\ref{lem:change_to_aux} and~\ref{lem:lipschitz-error-rmulmc}) yields, for $n > N$,
    \begin{align*}
        \KL(\bs \nu_n^\aux \mmid \bs \nu_n)
        &\lesssim \underbrace{\bigl(\frac{1}{\gamma h} \log\frac{\beta^{-1/2} \wedge T}{h} + \frac{\beta^{1/2} T}{\gamma h}\bigr)\,\beta^2 h^6}_{\msf A_\msx^2}\,(G_n^\msx)^2 + \underbrace{\bigl(\frac{1}{\gamma h} \log\frac{\beta^{-1/2} \wedge T}{h} + \frac{\beta^{1/2} T}{\gamma h}\bigr)\,\beta^2 h^4}_{\msf A_\msp^2}\,(G_n^\msp)^2 \\
        &\qquad{} + \underbrace{\bigl(\frac{1}{\gamma h} \log\frac{\beta^{-1/2} \wedge T}{h} + \frac{\beta^{1/2} T}{\gamma h}\bigr)\,\beta^2 \gamma dh^5}_{\msf B^2}\,.
    \end{align*}
    The condition $\msf A_\msx^2 \beta \vee \msf A_\msp^2 \lesssim 1$ is met for $h \lesssim \gamma^{1/3}/(\beta^{5/6} T^{1/3})$.
    The final bound follows from~\eqref{eq:noncvx_final_bd}.
\end{proof}

\begin{proof}[Proof of Theorem~\ref{thm:rmulmc-lsi}]
    From Lemma~\ref{lem:entropic-hypocoercivity}, we have $\KL(\bs\mu \bs P^N \mmid \bs\pi) \lesssim \varepsilon^2$ provided that $T \asymp \frac{\beta^{1/2}}{\alpha} \log \frac{\mf L(\bs\mu)}{\varepsilon^2}$.
    We apply Theorem~\ref{thm:rmulmc_discretization} with this choice of $T$, noting that $\beta^{1/2} T \gtrsim 1$, so that $\KL(\bs\mu\hat{\bs P}{}^{N-1}\hat{\bs P}{}' \mmid \bs\pi) \lesssim \varepsilon^2$ for the specified choice of step size. The result follows from Pinsker's inequality.
\end{proof}

\begin{proof}[Proof of Theorem~\ref{thm:rmulmc-spacetime}]
    From Lemma~\ref{lem:spacetime-poincare}, to ensure $\chi^2(\bs\mu \bs P{}^N \mmid \bs\pi) \lesssim \varepsilon^2$, we take $T \asymp \frac{1}{\alpha^{1/2}} \log \frac{\chi^2}{\varepsilon^2}$, where we abbreviate $\chi^2 \deq \chi^2(\bs\mu\mmid \bs\pi)$ in this proof.
    We apply Theorem~\ref{thm:rmulmc_discretization} with this choice of $T$, noting that $\beta^{1/2} T\gtrsim 1$, which yields $\KL(\bs\mu\hat{\bs P}{}^{N-1}\hat{\bs P}{}' \mmid \bs\pi) \lesssim \varepsilon^2$ for the specified choice of $h$.
    The result follows from the weak triangle inequality (Proposition~\ref{prop:weak-triangle}).
\end{proof}

    \addcontentsline{toc}{section}{References}
    \printbibliography{}

\end{document}